\UseRawInputEncoding
\documentclass{article}

\usepackage{stmaryrd,graphicx}
\usepackage{amssymb,mathrsfs,amsmath,amscd,amsthm}
\usepackage{float}
\usepackage{caption}
\usepackage{subcaption}
\usepackage{mathabx}
\usepackage[all,cmtip]{xy}
\DeclareMathAlphabet{\mathpzc}{OT1}{pzc}{m}{it}
\usepackage{amsfonts,latexsym,wasysym}
\usepackage[shortlabels]{enumitem}
\setlist{nosep}
\setenumerate{label=(\arabic*)}

\renewcommand{\int}{\operatorname{int}}
\newcommand{\shadj}{\operatorname{\bf Adj}}
\newcommand{\supp}{\operatorname{supp}}
\newcommand{\im}{\operatorname{Im}}
\newcommand{\wtphi}{\wt{\Phi}}

\newcommand{\bfs}{\mathbf{s}}
\newcommand{\sw}{\wt{\bigvee}}

\newcommand{\hX}{\widehat{X}}

\newcommand{\ds}{\displaystyle}
\newcommand{\ui}{I}

\newcommand{\bfx}{\mathbf{x}}
\newcommand{\tx}{\tilde{x}}

\newcommand{\tX}{\widetilde{X}}

\newcommand{\wt}{\widetilde}
\newcommand{\scrw}{\mathscr{W}}

\newcommand{\pionex}{\pi_{1}(X,x_0)}

\newcommand{\mca}{\mathcal{A}}

\newcommand{\mcz}{\mathcal{Z}}

\newcommand{\mcd}{\mathcal{D}}

\newcommand{\scra}{\mathscr{A}}
\newcommand{\scrb}{\mathscr{B}}
\newcommand{\scrc}{\mathscr{C}}
\newcommand{\scrd}{\mathscr{D}}

\newcommand{\scrf}{\mathscr{F}}
\newcommand{\scrg}{\mathscr{G}}

\newcommand{\scru}{\mathscr{U}}
\newcommand{\scrr}{\mathscr{R}}
\newcommand{\scrs}{\mathscr{S}}

\newcommand{\mcy}{\mathcal{Y}}

\newcommand{\bbh}{\mathbb{H}}
\newcommand{\bbn}{\mathbb{N}}

\newcommand{\bbr}{\mathbb{R}}

\newcommand{\bbz}{\mathbb{Z}}

\newcommand{\ov}{\overline}

\newtheorem{theorem}{Theorem}[section]
\newtheorem{lemma}[theorem]{Lemma}
\newtheorem{proposition}[theorem]{Proposition}
\newtheorem{corollary}[theorem]{Corollary}
\theoremstyle{definition}\newtheorem{definition}[theorem]{Definition}
\newtheorem{example}[theorem]{Example}

\newtheorem{remark}[theorem]{Remark}

\newtheorem{terminology}[theorem]{Local Terminology}
\begin{document}
\title{Sequential $n$-connectedness and infinite factorization in higher homotopy groups}

\author{Jeremy Brazas}


\date{\today}

\maketitle

\begin{abstract}
A space $X$ is ``sequentially $n$-connected" at $x\in X$ if for every $0\leq k\leq n$ and sequence of maps $f_1,f_2,f_3,\dots:S^k\to X$ that converges toward a point $x\in X$, the maps $f_m$ contract by a sequence of null-homotopies that converge toward $x$. We use this property, in conjunction with the Whitney Covering Lemma, as a foundation for developing new methods for characterizing higher homotopy groups of finite dimensional Peano continua. Among many new computations, a culminating result of this paper is: if $Y$ is a space obtained by attaching an infinite shrinking sequence $A_1,A_2,A_3,\dots$ of sequentially $(n-1)$-connected CW-complexes to a one-dimensional Peano continuum $X$ along a sequence of points in $X$, then there is a canonical injection $\Phi:\pi_n(Y)\to \prod_{j\in\bbn}\bigoplus_{\pi_1(X)}\pi_n(A_j)$. Moreover, we characterize the image of $\Phi$ using generalized covering space theory. As a case of particular interest, this provides a characterization of $\pi_n(\bbh_1\vee \bbh_n)$ where $\bbh_n$ denotes the $n$-dimensional Hawaiian earring.
\end{abstract}

\tableofcontents

\section{Introduction}

In the past two decades, the homotopy theory of Peano continua and the algebraic theory of natural infinitary operations that arise in their homotopy groups has witnessed incredible progress, led by the pioneering work of Katsuya Eda. However, this progress has almost exclusively centered around fundamental groups and homotopy types of low-dimensional spaces. For instance, Eda's remarkable classification theorem \cite{Edaonedim} asserts that homotopy types of one-dimensional Peano continua are completely determined by the isomorphism type of their fundamental groups \cite{Edaonedim}. With the exception of the asphericity of one-dimensional \cite{CFhigher} and planar spaces \cite{CCZ} and the work of Eda-Kawamura \cite{EK00higher} on shrinking wedges of spaces like the $n$-dimensional Hawaiian earring $\bbh_n$, essentially no significant progress has been made in the effort to further characterize higher homotopy groups that admit natural infinite product operations. This stagnation is primarily due to the lack of techniques for suitably constructing homotopies on higher dimensional domains. In this paper, we develop new techniques that allow us to make new computations and, which, are expected to support a renewed effort in developing the infinitary algebraic topology of Peano continua. Three key new insights include the following:
\begin{enumerate}
\item We introduce and study the \textit{sequentially $n$-connected property} (Definition \ref{scndefinition}), which is closed under many important constructions (including shrinking wedges and infinite direct products) and provides a robust framework for extending methods from classical homotopy theory to the infinitary setting.
\item We introduce a process for constructing intricate homotopies by applying the Whitney Covering lemma \cite{Whitney}, a classical result in analysis which ensures the existence of convenient cubical coverings of arbitrary open sets in $\bbr^n$.
\item We apply the generalized universal covering spaces introduced in \cite{FZ07} in the same way that classical covering spaces are used to characterize higher homotopy groups of non-simply connected spaces.
\end{enumerate}
The primary obstruction to understanding higher homotopy groups of locally complicated spaces is the difficulty in uniquely characterizing the homotopy class of an arbitrary $n$-loop, i.e. a map $f:S^n\to X$. One must be able to replace $f$ with a homotopic map $g:S^n\to X$, which is structurally much simpler than $f$ or which satisfies some kind of special form. For instance, in classical homotopy theory, if $X$ is a CW-complex, then $f$ may be replaced by a cellular map or if $X=X_1\vee X_2$ is a wedge of $(n-1)$-connected spaces, then $f$ is homotopic to a product $g_1\cdot g_2$ where $g_i:S^n\to X_i$ is a based map.

In this paper, we overcome the lack of classical homological methods by constructing required homotopies by ``brute force." Indeed, Eda and Kawamura's work \cite{EK00higher} has also suggested that there is little hope of circumventing such technicalities. These homotopies will typically deform infinitely many parts of a map simultaneously and so the primary difficulty is identifying a construction that is actually continuous. Although several of our proofs are quite technical, the first two key insights mentioned above make our methods possible and intuitive. In particular, ``sequential $n$-connectedness" is a convenient sequential analogue of a space being both $n$-connected and locally $n$-connected. We find that this property is precisely what is required to employ an inductive process for constructing homotopies that realize infinite factorizations of arbitrary homotopy classes.

To illustrate a concrete application of our methods, we focus on a particular construction: we begin with a Peano continuum $X$, a (possibly repeating) sequence $\{x_j\}_{j\in\bbn}$ in $X$, and a sequence of based spaces $\{(A_j,a_j)\}_{j\in\bbn}$. We construct a space $Y$ by attaching each $A_j$ to $X$ by identifying $a_j\sim x_j$ and giving $Y$ a topology so that whenever a subsequence $\{x_{j_i}\}$ converges to $x$, the corresponding attachment spaces $A_{j_i}$ ``shrink" toward $x$. We call the resulting space $Y$ a \textit{shrinking adjunction space} and we often write $Y=\shadj(X,x_j,A_j,a_j)$ to distinguish this decomposition of $Y$. For example, if $X=[0,1]$, we can construct a shrinking adjunction space $Y$ by attaching an $n$-sphere of diameter $\frac{1}{2^m}$ at each dyadic rational $\frac{j}{2^m}$, $m\in\bbn$, $0<j<2^m-1$ (see Figure \ref{fig1}). One should beware that (a) collapsing the arc $X$ will not result in a weakly homotopy equivalent quotient space $Y/X$ and (b) $Y$ does not satisfy the local contractability condition used in \cite{EK00higher}. Hence, completely new methods are required to address what, deceivingly, appears to be only a small modification of the $n$-dimensional Hawaiian earring $\bbh_n$.

\begin{figure}[h]
\centering \includegraphics[height=1.3in]{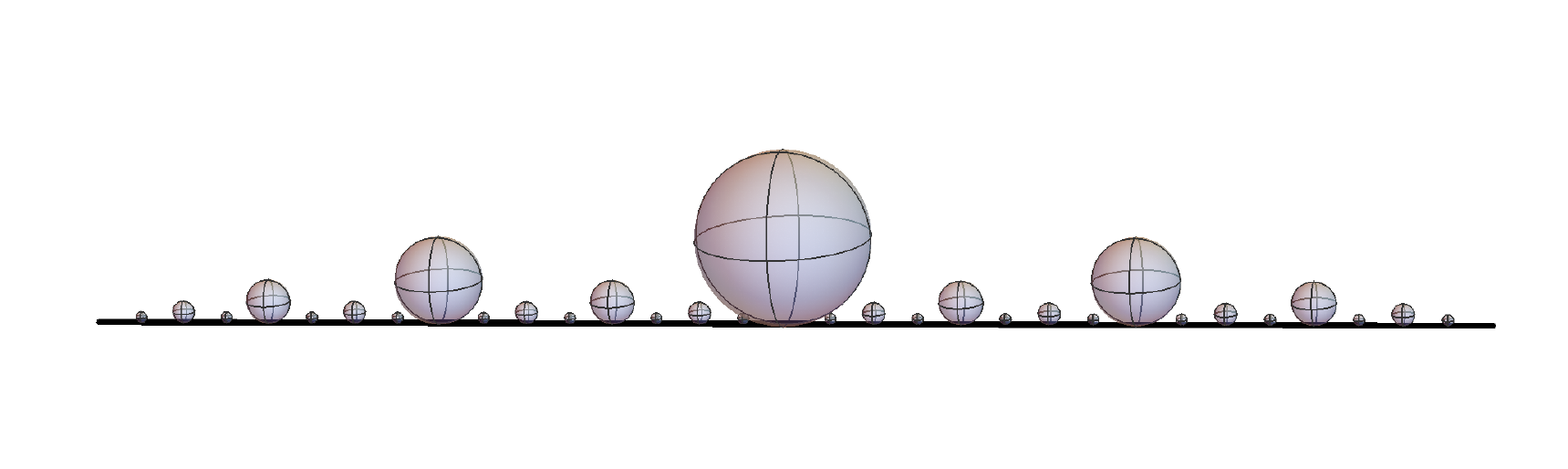}
\caption{\label{fig1}An example of a shrinking adjunction space constructed by attaching $2$-spheres attached to an arc along a dense set of points.}
\end{figure}

The culminating result of this paper is the following theorem, which allows one to understand $\pi_n(Y)$ as a subgroup of a product of the homotopy groups of the attachment spaces $A_j$. In this result, $X$ may be the Hawaiian earring, Sierpinski carpet, or Menger curve, all of which fail to have traditional universal covering spaces.

\begin{theorem}\label{mainthm}
Let $n\geq 2$ and $Y=\shadj(X,x_j,A_j,a_j)$ be the shrinking adjunction space with basepoint $y_0\in X$ where $X$ is a one-dimensional Peano continuum and each $(A_j,a_j)$ is an $(n-1)$-connected CW-complex. Then there is an injection
\[\Phi:\pi_n(Y)\to \prod_{j\in\bbn}\prod_{\pi_1(X)}\pi_n(A_j),\]
which is canonical after making a choice of a sequence of paths $\beta_j:(I,0,1)\to (X,y_0,x_j)$, $j\in\bbn$.
\end{theorem}

Moreover, we characterize the image of $\Phi$ in terms of the ``whisker topology" on generalized covering spaces. Theorem \ref{mainthm} was motivated, in part, to provide a solution to the open problem of characterizing $\pi_n(\bbh_1\vee \bbh_n)$ for $n\geq 2$ (see Figure \ref{figx2}). The primary difficulty embedded in this problem is the effect of the Hawaiian earring group $\pi_1(\bbh_1\vee \bbh_n)=\pi_1(\bbh_1)$ on $\pi_n(\bbh_1\vee \bbh_n)$. Due to the infinitary nature of $\pi_1(\bbh_1)$, this effect is substantially more complicated than the usual $\pi_1$-action. Our analysis shows that $\pi_n(\bbh_1\vee \bbh_n)$ embeds canonically into the group $\bbz^{\bbn\times\pi_1(\bbh_1)}$. In particular, $\pi_n(\bbh_1\vee \bbh_n)$ is isomorphic to the group of functions $g:\bbn\times \pi_1(\bbh_1)\to \bbz$, which have finite support in the second variable (and thus countable support overall) and for which $\{[\alpha]\in\pi_1(\bbh_1)\mid g(j,[\alpha])\neq 0\}$ has compact closure in $\pi_1(\bbh_1)$ with the whisker topology. See Example \ref{h1hn} for more details. This characterization of $\pi_n(\bbh_1\vee \bbh_n)$ is canonical and suffices for any computations or applications the author can imagine. Since it follows that $\pi_n(\bbh_1\vee \bbh_n)$ is cotorsion-free, a Fuchs-type decomposition \cite{Fuchs} may provide a non-canonical description of the isomorphism type of $\pi_n(\bbh_1\vee \bbh_n)$ as an abelian group.
\begin{figure}[h]
\centering \includegraphics[height=2.5in]{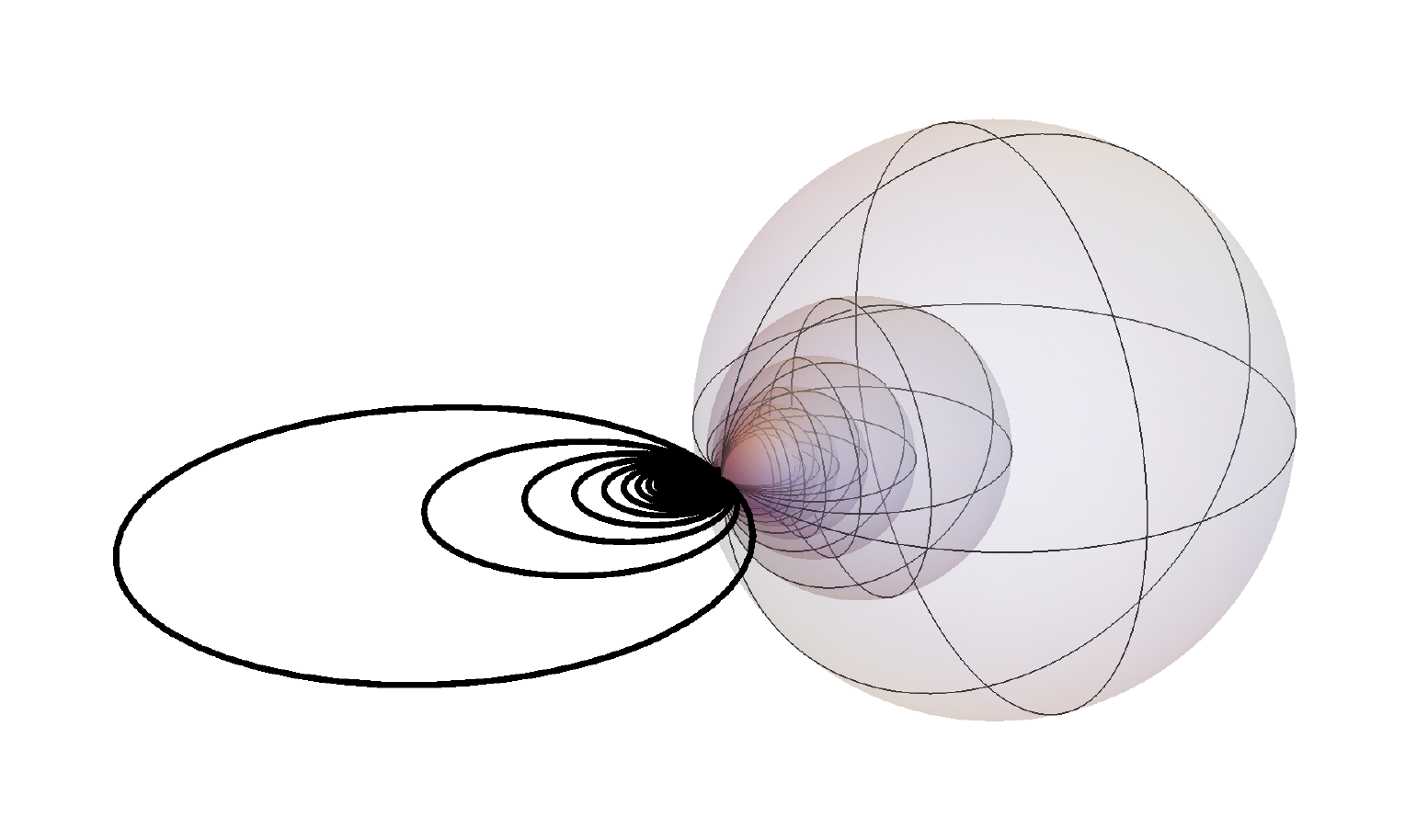}
\caption{\label{figx2}The space $\bbh_1\vee\bbh_2$.}
\end{figure}

Theorem \ref{mainthm} is proved as a special case of Theorem \ref{maintheorem}, which is the strongest result in this paper. We avoid the statement of Theorem \ref{maintheorem} here since it involves some terminology to be developed later on. The new techniques leading up to the proof of Theorem \ref{maintheorem}, particularly the factorization methods in Sections \ref{sectionwhitneycover} and \ref{sectiondendritecore} are expected to provide new methods for addressing several other fundamental open problems in the field.

Finally, we remark that the length of this paper is a result of the substantial technical hurdles that must be overcome to prove Theorem \ref{maintheorem} (and consequently Theorem \ref{mainthm}). Sections \ref{sectionprelim} and \ref{sectionshrinkingadjspaces} are dedicated to establishing required notation, terminology, and constructions. With the exception of scattered remarks and examples intended to provide intuition and context for the reader, Sections \ref{sectionwhitneycover}, \ref{sectiondendritecore}, and \ref{sectiongencovers} develop a step-by-step approach to proving Theorem \ref{maintheorem}.\\

\noindent\textbf{Acknowledgments:} The author thanks Katsuya Eda for conversations and questions that helped motivate this project.

\section{Infinite products in homotopy groups}\label{sectionprelim}

All topological spaces in this paper are assumed to be Hausdorff. Throughout, $I$ denotes the unit interval $[0,1]$ and $S^n\subseteq \bbr^{n+1}$ denotes the unit $n$-sphere with basepoint $e_n=(1,0,0,\dots,0)$. If $X$ and $Y$ are spaces, then $Y^X$ will denote the space of continuous functions $X\to Y$ with the compact-open topology. Generally, $c_y\in Y^X$ will denote the constant map at $y\in Y$. If $A\subseteq X$ and $B\subseteq Y$, then $(Y,B)^{(X,A)}$ will denote the subspace of $Y^X$ consisting of maps $f:X\to Y$ satisfying $f(A)\subseteq B$. For a based space $(X,x)$, we let $\Omega^{n}(X,x)$ denote the relative mapping space $(X,x)^{(I^n,\partial I^n)}$ so that $\pi_n(X,x)=\pi_0(\Omega^{n}(X,x))$ is the $n$-th homotopy group. We refer to elements $f\in \Omega^n(X,x)$ as \textit{$n$-loops}, which may also be viewed as based maps $(S^n,e_n)\to (X,x)$ using a fixed choice of homeomorphism $I^n/\partial I^n\cong S^n$. More generally, $[(X,x),(Y,y)]$ will denote the set of pointed-homotopy classes of maps $(X,x)\to (Y,y)$. We say that a homotopy $H:X\times I\to Y$ is \textit{constant on} $A\subseteq X$ (or is \textit{relative to }$A$) if for all $\bfx\in A$, $H(\bfx,t)$ is constant as $t$ varies.

For $0\leq m\leq n$, the term \textit{$m$-cube} will refer to subsets of $I^n$ of the form $R=\prod_{i=1}^{n}C_i$ where $C_i$ is a closed interval for $m$-values of $i$ and a single point for the other $n-m$ values of $i$. A subset of $R$ of the form $F=\prod_{i=1}^{n}D_i$ where $D_i$ is either $C_i$ or one of the two boundary points of $C_i$ is called a \textit{face} of $R$. Specifically, $F$ is a \textit{$p$-face of} $R$ if $F$ is a $p$-cube. Generally, we will use the term boundary of a subset $A\subseteq I^n$ to mean the topological boundary $\partial A$ in $I^n$. The \textit{face-boundary} of a $k$-cube $R$, denoted $bd(R)$, is the union of the $(k-1)$-faces of $R$. Certainly, $\partial R=bd(R)$ for any $n$-cube $R\subseteq I^n$ and $\partial R=\emptyset$ for all $k$-cubes with $k<n$.

Given two $n$-cubes $R=\prod_{i=1}^{n}C_i$ and $R'=\prod_{i=1}^{n}C_i'$ in $I^n$ and maps $f:R\to X$ and $g:R'\to X$, we write $f\equiv g$ if $f=g\circ h$ for a homeomorphism $h:R\to R'$, which
\begin{itemize}
\item maps each $m$-face of $R$ onto an $m$-face of $R'$, 
\item maps each $1$-face $\prod_{i=1}^{n}D_i$ with non-degenerate component $D_j$ to the $1$-face $\prod_{i=1}^{n}D_i'$ of $R'$ with non-degenerate component $D_j'$ by a map whose $j$-th component $D_j\to D_j'$ is increasing.
\end{itemize}
Such a map $h$ need not be affine but the second condition ensures that $h$ maps $0$-faces to $0$-faces and never rotates or changes the orientation of any $p$-face. For example, let $L_{R,R'}:R\to R'$ denote the canonical map, which is increasing and linear in each component. If $f=g\circ L_{R,R'}$, then $f\equiv g$. If $f:(R,\partial R)\to (X,x)$ is a map, and if it does not cause confusion, we will identify $f$ with the $n$-loop $f\circ L_{I^n,R}\in \Omega^n(X,x)$ and still refer to $f$ as an $n$-loop.

Given maps $f_1,f_2,\dots, f_m\in \Omega^{n}(X,x)$, the \textit{$m$-fold concatenation} $\prod_{i=1}^{m}f_i=f_1\cdot f_2 \cdots f_m$ is the map $(I^n,\partial I^n)\to (X,x)$ whose restriction to $R_{i}=\left[\frac{i-1}{m},\frac{i}{m}\right]\times I^{n-1}$ is $f_i\circ L_{R_i,I^n}$. If $f\in X^{I^n}$, then the map $f^{-}\in X^{I^n}$ defined by $f^{-}(t_1,t_2,\dots,t_n)=f(1-t_1,t_2,\dots,t_n)$ is called the \textit{reverse} of $f$.

\begin{remark}\label{retractionremark}
If $f\in \Omega^{n}(X,y)$ and $\alpha:I\to X$ is a path from $x$ to $y$, then there is a canonical map $\alpha\ast f\in \Omega^n(X,x)$ that corresponds to the action of the fundamental groupoid on $\coprod_{a\in X}\pi_n(X,a)$. In particular, $(\alpha\ast f)|_{[1/3,2/3]^n}\equiv f$ and $(\alpha\ast f)(\partial ([s,1-s]^n))=\alpha(3s)$ for $s\in [0,1/3]$. We refer to $\alpha\ast f$ as the \textit{path-conjugate} of $f$ by $\alpha$.

We will later be confronted with the situation where an $n$-loop $f$ and path $\alpha$ appear together as a map $g:I^n\times \{1\}\cup (\partial I^n)\times I\to X$ where $g(\bfx,1)=f(\bfx)$ and $g(\bfx,t)=\alpha(t)$ if $(\bfx,t)\in (\partial I^n)\times I$. When this occurs, we note that there is a canonical piecewise-linear map $r:I^n\times I\to I^n\times\{0\}\cup (\partial I^n)\times I$ that retracts the $(n+1)$-cube down onto the ``hollow box without a top," specifically, so that $r(\bfx,0)=(\alpha\ast f)(\bfx)$.
\end{remark}

 A \textit{Peano continuum} is a connected locally path-connected compact metric space. The Hahn-Mazurkiewicz Theorem \cite[Theorem 8.14]{Nadler} implies that a Hausdorff space is a Peano continuum if and only if there exists a continuous surjection $\ui\to X$. Given any path-connected Hausdorff space $X$ and $x_0\in X$, the $n$-th homotopy group $\pi_n(X,x_0)$ is isomorphic to the canonical directed limit $\varinjlim_{A\in\mathcal{C}}\pi_n(A,x_0)$ where $\mathcal{C}$ is the directed set of Peano continua in $X$ containing $x_0$ (and directed by inclusion). This fact justifies our primary interest in Peano continua, although many of our results will apply far more generally.

\subsection{Infinite products over $n$-domains}

\begin{definition}
An \textit{$n$-domain} is a set $\mathscr{R}$ of $n$-cubes in $I^n$ whose interiors are pairwise disjoint, i.e. $\int(R)\cap \int(R')=\emptyset$ if $R\neq R'$ in $\scrr$.
\end{definition}

Notice that $n$-domains are always countable. Sometimes, we may have need to index $\scrr$ as $\{R_k\mid k\in K\}$ where $K$ is well-ordered.

\begin{definition}
Let $\scrr$ be an $n$-domain. We say that $\scrr$ is a \textit{locally finite cover} of a subset $A\subseteq I^n$ if $\bigcup \scrr=A$ and for fixed $R\in\scrr$, $R\cap R'\neq\emptyset$ for at most finitely many $R'\in \scrr$.
\end{definition}

\begin{definition}\label{convergesdef}
If $\{f_k\}_{k\in K}\subseteq Y^X$ is a collection of maps indexed by a countable set $K$, we say that $\{f_k\}_{k\in K}$ \textit{clusters at a point} $y\in Y$ if for every neighborhood $U$ of $y$, we have $\im(f_k)\subseteq U$ for all but finitely many $k\in K$. If $K$ is a well-ordered set, we may say that $\{f_k\}_{k\in K}$ \textit{converges to the point} $y$ since $\{f_k\}_{k\in K}$ clusters at $y$ if and only if $\{f_k\}_{k\in K}$ converges to the constant map $c_y$ in $Y^X$.
\end{definition}

Since $Y$ is assumed to be Hausdorff, a set of based maps $\{f_k\}_{k\in K} \subseteq (Y,y_0)^{(X,x_0)}$ can only cluster at the basepoint $y_0$.

\begin{definition}\label{rconcatenationdef}
Let $\mathscr{R}$ be an $n$-domain and consider a set $\{f_R\}_{R\in\scrr}\subseteq \Omega^n(X,x_0)$ that clusters at $x_0$. The \textit{$\scrr$-concatenation} of $\{f_R\}_{R\in\scrr}$ is the map $\prod_{\scrr}f_R\in \Omega^n(X,x_0)$ whose restriction to $R\in\scrr$ is $f_R\circ L_{R,I^n}$ and which maps $I^n\backslash \bigcup \scrr$ to $x_0$.
\end{definition}

If $\scrr$ is infinite, then the assumption that $\{f_R\}_{R\in\scrr}$ clusters at $x_0$ is required for the $\scrr$-concatenation to be continuous. The following theorem is an infinite commutativity result based on the work of Eda-Kawamura \cite{EK00higher}, which was refined and extended in \cite{Brazasncubeshuffle}; it will play a crucial supporting role in the current paper.

\begin{theorem}[Infinitary $n$-Cube Shuffle]\label{shuffletheorem} Suppose $\scrr$ and $\scrs$ are $n$-domains and $\{f_R\}_{R\in\scrr}$ and $\{g_S\}_{S\in\scrs}$ are subsets of $\Omega^n(X,x_0)$ that cluster at $x_0$. If there is a bijection $\phi:\scrr\to\scrs$ such that $f_{R}\equiv g_{\phi(R)}$ for all $R\in\scrr$, then we have $\prod_{\scrr}f_R\simeq\prod_{\scrs}g_S$ by a homotopy (rel. $\partial I^n$) with image in $\bigcup_{R\in\scrr}\im(f_R)$. 
\end{theorem}


\begin{definition}\label{concatenationdef}
The \textit{standard $n$-domain} is the $n$-domain $\scrr=\{R_k\mid k\in\bbn\}$ where $R_k=\left[\frac{k-1}{k},\frac{k}{k+1}\right]\times I^n$. The \textit{infinite concatenation} of a sequence $\{f_k\}_{k\in \bbn}$ in $\Omega^{n}(X,x)$ that converges to $x$ is the $\scrr$-concatenation $\prod_{\scrr}f_k$ where $\scrr$ is the standard $n$-domain. We will typically denote this $n$-loop as $\prod_{k=1}^{\infty}f_k$.
\end{definition}

The following ``$n$-domain shrinking" result is a modest generalization of \cite[Lemma 2.3]{Brazasncubeshuffle}. Since the proof is straightforward and essentially the same, we omit it.

\begin{lemma}[$n$-domain shrinking]\label{shrinkingcubelemma}
Let $f:I^n\to X$ be a map and $\scrr=\{R_k\mid k\in K\}$ be an $n$-domain such that for every $k\in K$, there is a point $x_k\in X$ such that $f(\partial R_k)=x_k$. If $\mathscr{S}=\{S_k\mid k\in K\}$ is a sub-$n$-domain of $\scrr$, i.e. $S_k\subseteq R_k$ for all $k\in K$, then there is a map $g:I^n\to X$ such that
\begin{itemize}
\item $g$ agrees with $f$ on $I^n\backslash \bigcup_{k\in K}\int(R_k)$,
\item for all $k\in K$, $g(R_k\backslash \int(S_k))=x_k$ and $g|_{S_k}\equiv f|_{R_k}$,
\item $f$ is homotopic to $g$ by a homotopy that is constant on $I^n\backslash \bigcup_{k\in K}\int(R_k)$ and which maps $R_k\times I$ into $f(R_k)$.
\end{itemize}
\end{lemma}

\subsection{Sequential homotopy}

\begin{definition}
Let $\{f_m\}_{m\in\bbn}$, $\{g_m\}_{m\in\bbn}$ be sequences of maps $X\to Y$ both of which converge to $y_0\in Y$. We say that $\{f_m\}_{m\in\bbn}$ and $\{g_m\}_{m\in\bbn}$ are \textit{sequentially homotopic} (and write $\{f_m\}_{m\in\bbn}\simeq\{g_m\}_{m\in\bbn}$) if there exists a sequence of free homotopies $H_m:X\times I\to Y$, $m\in\bbn$ from $f_m$ to $g_m$ such that $\{H_m\}_{m\in\bbn}$ converges to $x$. We refer to the sequence $\{H_m\}_{m\in\bbn}$ as a \textit{sequential homotopy.} We also define the following terms:
\begin{itemize}
\item If $\{f_m\}_{m\in\bbn}$ is sequentially homotopic to a constant sequence $\{c_{y_0}\}_{m\in\bbn}$ of constant maps, then we say that $\{f_m\}_{m\in\bbn}$ is \textit{sequentially null-homotopic}.
\item If $A\subseteq X$ and $\{f_m\}_{m\in\bbn}\simeq\{g_m\}_{m\in\bbn}$ by a sequential homotopy $\{H_m\}_{m\in\bbn}$ such that, for all $m\in\bbn$, $H_m$ is the constant homotopy on $A$, then we say that $\{f_m\}_{m\in\bbn}$ and $\{g_m\}_{m\in\bbn}$ are \textit{sequentially homotopic rel. $A$.} In particular, if $A=\{x_0\}$ and $f_m,g_m\in (Y,y_0)^{(X,x_0)}$ for all $m\in\bbn$, then we say $\{f_m\}_{m\in\bbn}$ and $\{g_m\}_{m\in\bbn}$ are \textit{sequentially homotopic rel. basepoint}.
\item If a sequence $\{f_m\}_{m\in\bbn}$ in $(Y,y_0)^{(X,x_0)}$ is sequentially homotopic rel. basepoint to the constant sequence $\{c_{y_0}\}_{m\in\bbn}$, we say $\{f_m\}_{m\in\bbn}$ is \textit{sequentially null-homotopic rel. basepoint}.
\end{itemize}
\end{definition}

The various notions of sequential homotopy define equivalence relations on the relevant sets of convergent sequences of maps. The set of pointed-sequential homotopy classes of maps $(X,x_0)\to (Y,y_0)$ will inherit natural group structure whenever the usual set of pointed-homotopy classes does.

\begin{remark}[Infinite horizontal concatenation]
If $\{H_m\}_{m\in\bbn}$ is a sequential homotopy of sequences $\{f_m\}_{m\in\bbn}$ and $\{g_m\}_{m\in\bbn}$ in $\Omega^n(X,x)$ that converge to $x$, then we may define a \textit{infinite horizontal concatenation} of $\{H_m\}_{m\in\bbn}$ as follows: $\prod_{m=1}^{\infty}H_m$ is the map $I^n\times I\to X$ whose restriction to $R_m=\left[\frac{m-1}{m},\frac{m}{m+1}\right]\times I^{n-1}\times I$ is $H_m\circ L_{R_m,I^{n+1}}$ and which maps $\{1\}\times I^{n-1}\times I$ to $x$. Then $\prod_{m=1}^{\infty}H_m$ is a homotopy rel. $\partial I^n$ between $\prod_{m=1}^{\infty}f_m$ and $\prod_{m=1}^{\infty}g_m$ (See Figure \ref{fighor}).
\end{remark}

\begin{figure}[h]
\centering \includegraphics[height=2.3in]{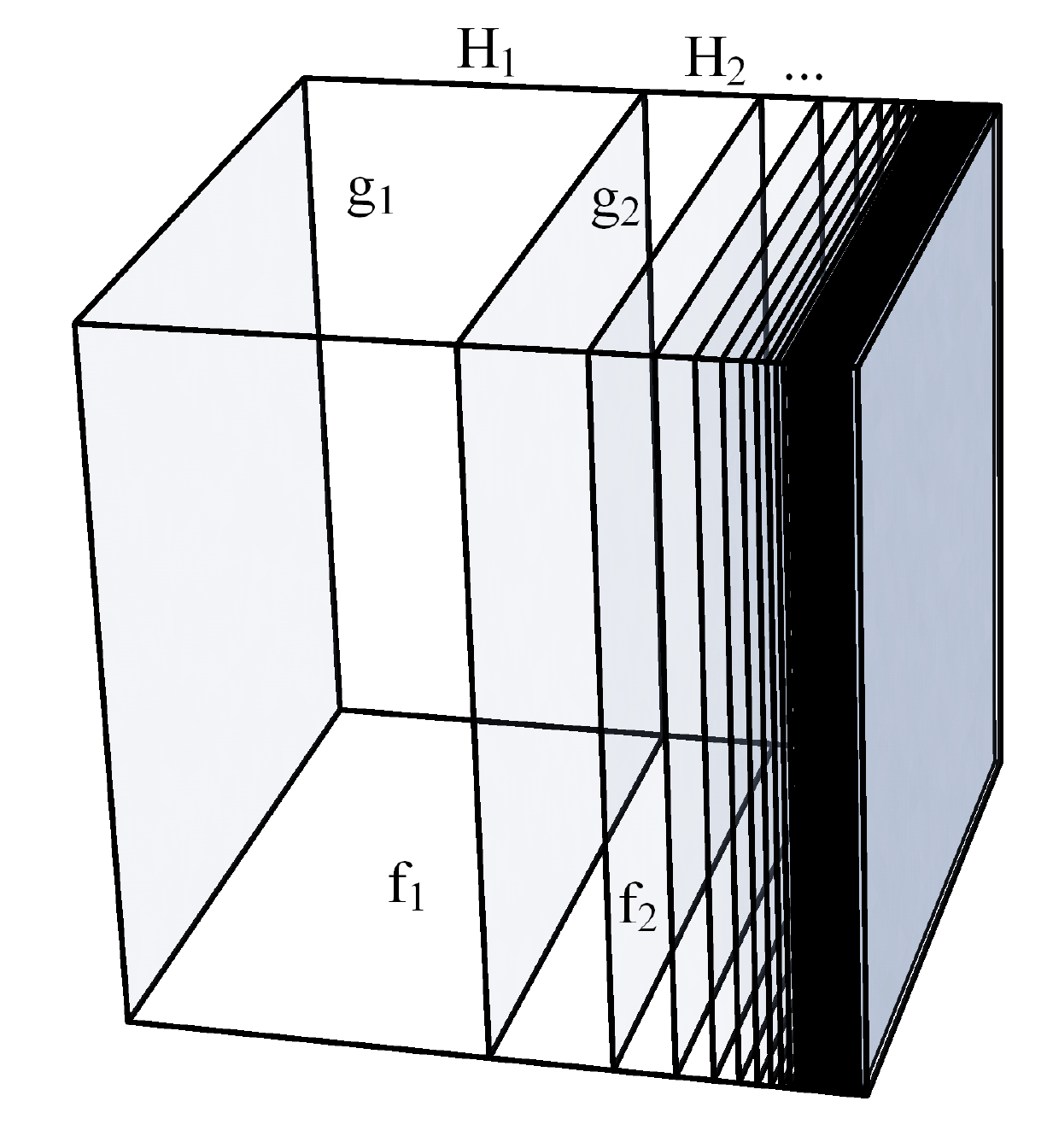}
\caption{\label{fighor}The domain of an infinite horizontal concatenation of homotopies $H_1,H_2,H_3,\dots$ of $2$-loops.}
\end{figure}

\begin{remark}[Infinite vertical concatenation]\label{verticalinfconcatremark}
In some of our more technical results, we will be required to form infinite ``vertical composition" of homotopies. Suppose that $\{g_p\}_{p\in\bbn}\to g_{\infty}$ in $\Omega^n(Y,y_0)$ (with respect to the compact-open topology). Suppose that $G_p:I^n\times I\to Y$ is a homotopy rel. $\partial I^n$ from $g_{p}$ to $g_{p+1}$. Exponential laws of mapping spaces imply that $G_p$ corresponds uniquely a path $\alpha_p:I\to \Omega^n(Y,y_0)$ from $g_p$ to $g_{p+1}$. If $\{\alpha_p\}_{p\in\bbn}$ converges to $g_{\infty}$ in the sense of Definition \ref{convergesdef}, then we may form the infinite path concatenation $\alpha_{\infty}=\prod_{p\in\bbn}\alpha_p:I\to \Omega^n(Y,y_0)$, which satisfies $\alpha_{\infty}(\frac{p-1}{p})=g_p$ and $\alpha_{\infty}(1)=g_{\infty}$. Now $\alpha_{\infty}$ corresponds uniquely to a homotopy $G_{\infty}:I^n\times I\to Y$ from $g_1$ to $g_{\infty}$. Moreover, $G_{\infty}(\bfx,\frac{p-1}{p})=g_{p}(\bfx)$ for all $p\in\bbn$ and $G_{\infty}(\bfx,1)=g_{\infty}(\bfx)$. When $G_{\infty}$ exists and is continuous, we refer to it as the \textit{infinite vertical concatenation} of the homotopies $\{G_p\}_{p\in\bbn}$ (See Figure \ref{figver}).

In practice, vertical infinite compositions are more difficult to construct than their horizontal counterparts. Generally, we will construct non-trivial infinite vertical concatenations in the following scenario: Consider a sequence $\{g_p\}_{p\in\bbn}$ in $\Omega^n(Y,y_0)$ constructed so that there a nested sequence of closed sets $C_1\supseteq C_2\supseteq C_{3}\supseteq \cdots$ such that \[C=\bigcap_{p\in\bbn}C_p=\bigcap_{p\in\bbn}\int(C_p)\neq \emptyset\] and such that $g_{p+1}$ agrees with $g_p$ on $I^n\backslash \int(C_{p})$. Suppose we have constructed a function $g_{\infty}:(I^n,\partial I^n)\to (Y,y_0)$ and a continuous homotopy $G_p:I^n\times I\to Y$ from $g_p$ to $g_{p+1}$ such that:
\begin{enumerate}
\item for every $p\in\bbn$, $G_p$ is the constant homotopy on $I^n\backslash \int(C_{p})$,
\item if $\bfx\in I^n\backslash  C$, then $g_{\infty}(\bfx)$ is the value of the eventually constant sequence $\{g_p(\bfx)\}_{p\in\bbn}$,
\item $(g_{\infty})|_{C}: C\to Y$ is continuous,
\item if $\{\bfx_p\}_{p\in\bbn}\to \bfx$ with $\bfx_p\in C_p$, then every neighborhood of $g_{\infty}(\bfx)$ must contain $G_p(\{\bfx_p\}\times I)$ for all but finitely many $p\in\bbn$.
\end{enumerate}
Under these four conditions, the vertical infinite concatenation $G_{\infty}$ of the sequence $\{G_{p}\}_{p\in\bbn}$ (with $G_{\infty}(\bfx,1)=g_{\infty}(\bfx)$) is well-defined and continuous. Since we will refer to this situation multiple times and since verifying continuity of $G_{\infty}$ is not entirely trivial, we give a proof below.
\end{remark}
\begin{figure}[h]
\centering \includegraphics[height=2.3in]{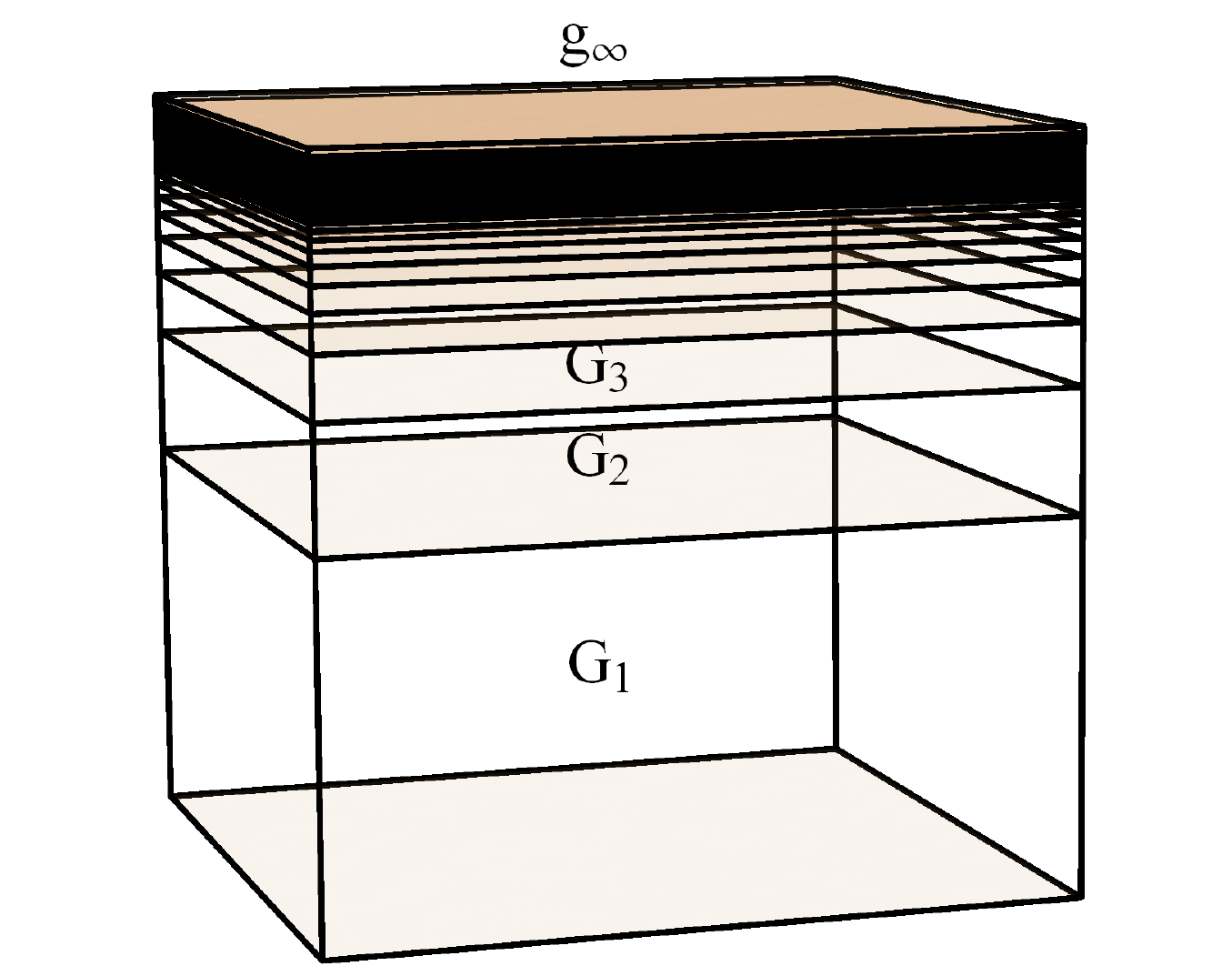}
\caption{\label{figver}The domain of an infinite vertical concatenation of homotopies $G_1,G_2,G_3,\dots$ that converge to $2$-loop $g_{\infty}$}
\end{figure}

\begin{lemma}
Consider a sequence of $n$-loops $\{g_{p}\}_{p\in\bbn}$ in $\Omega^n(X,x_0)$ and homotopies $G_{p}$ from $g_{p}$ to $ g_{p+1}$ satisfying Conditions (1)-(4) in Remark \ref{verticalinfconcatremark}. Then the map $g_{\infty}$ and the infinite vertical concatenation $G_{\infty}:I^n\times I\to Y$ of the sequence $\{G_{p}\}_{p\in\bbn}$ are continuous.
\end{lemma}

\begin{proof}
First, we check that $g_{\infty}$ is continuous. Since $g_p$ is continuous, Condition (2) implies that $g_{\infty}$ is continuous on the open set $I^n\backslash C$. By Condition (3), $g_{\infty}|_{C}:C\to Y$ is continuous. Therefore, it suffices to consider the following situation. Suppose $\{\bfx_i\}_{i\in\bbn}\to\bfx$ where $\bfx\in C$ and $\bfx_i\in I^n\backslash C$ for all $i\in\bbn$. Let $U$ be an open neighborhood of $g_{\infty}(\bfx)$ in $Y$. Since $\bfx\in C=\bigcap_{p\in\bbn}\int(C_p)$, we may replace $\{\bfx_i\}_{i\in\bbn}$ with a cofinal subsequence in $\int(C_1)$. Find $p_i\in \bbn$ such that $\bfx_i\in C_{p_i}\backslash C_{p_i+1}$. Then $G_{p_i}(\bfx_i,1)=g_{p_i+1}(\bfx_i)=g_{\infty}(\bfx_i)$. If $\{p_i\}_{i\in\bbn}$ was bounded above by $q\in\bbn$, then $\bfx\in \int(C_{q+1})$ but $\bfx_i\notin \int(C_{q+1})$ for all $i$; a contradiction of $\{\bfx_i\}_{i\in\bbn}\to \bfx$. Therefore, $\{p_i\}_{i\in\bbn}\to\infty$. If $G_{p_i}(\bfx_i,1)\notin U$ for infinitely many $i$, then, we replace $\{p_i\}_{i\in\bbn}$ by an increasing subsequence. Assuming $p_1<p_2<p_2<\cdots$, we can define $\bfx_p=\bfx$ whenever $p\notin\{p_1,p_2,p_3,\dots\}$. This gives a sequence $\{\bfx_p\}_{p\in\bbn}\to \bfx$ with $\bfx_p\in C_p$ and $G_{p}(\bfx_{p},1)\notin U$ for infinitely many $p$; a violation of Condition (4) Therefore, we must have $g_{\infty}(\bfx_i)=G_{p_i}(\bfx_i,1)\notin U$ for all but finitely many $i\in\bbn$. This gives $\{g_{\infty}(\bfx_i)\}_{i\in\bbn}\to g_{\infty}(\bfx)$, completing the proof of the continuity of $g_{\infty}$.

The vertical concatenation $G_{\infty}$ is well-defined by assumption. Conditions (1) and (2) imply that the only non-trivial case to check is the continuity of $G_{\infty}$ at points in $C\times \{1\}$. In particular, since we have already established the continuity of $g_{\infty}$, we only need to consider a sequence $\{(\bfx_i,t_i)\}_{i\in\bbn}\to (\bfx,1)$ with $\bfx\in C$ and such that $G_{\infty}(\bfx_i,t_i)\neq g_{\infty}(\bfx_i)$ for all $i\in\bbn$. In particular, we may assume $t_i<1$ for all $i$. Condition (4) allows us to reduce to the case where $\bfx_i\notin C$ for all $i\in\bbn$. Since $\bfx\in \int(C_1)$, we may replace $\{(\bfx_i,t_i)\}_{i\in\bbn}$ with a cofinal subsequence to find $q_i\in\bbn$ such that $\bfx_i\notin C_{q_i}\backslash C_{q_i+1}$ and $\{q_i\}_{i\in\bbn}\to\infty$. Let $U$ be a neighborhood of $G_{\infty}(\bfx,1)=g_{\infty}(\bfx)$ in $Y$. We will show that $G_{\infty}(\bfx_i,t_i)\in U$ for all but finitely many $i\in\bbn$.

Find $p_i\in\bbn$ such that $t_i$ lies in the domain of $G_{p_i}$, i.e. $t_i\in \left[\frac{p_i-1}{p_i},\frac{p_i}{p_i+1}\right]$. Write $s_i=L_{I^n,\left[\frac{p_i-1}{p_i},\frac{p_i}{p_i+1}\right]}(t_i)$ so that $G_{p_i}(\bfx_i,s_i)=G_{\infty}(\bfx_i,t_i)$. Since $G_{p_i}(\bfx_i,s_i)\neq g_{\infty}(\bfx_i)$, Conditions (1) and (2) ensure that $p_i\leq q_i$ and $\bfx_i\in C_{p_i}\backslash C_{q_i+1}$. Since $\{t_i\}_{i\in\bbn}\to 1$, we have $\{p_i\}_{i\in\bbn}\to\infty$. Suppose, to obtain a contradiction, that $G_{\infty}(\bfx_i,t_i)\in U$ for infinitely many $i$. Replacing $\{(\bfx_i,t_i)\}_{i\in\bbn}$ with a subsequence, if necessary, we may assume that $p_1<p_2<p_3<\cdots$. Now $G_{\infty}(\bfx_i,t_i)=G_{p_i}(\bfx_i,s_i)$ where $\bfx_i\in C_{p_i}$. Define $\mathbf{a}_p=\bfx_i$ when $p=p_i$ and $\mathbf{a}_p=\bfx$ otherwise. Similarly, define $u_p=s_i$ when $p=p_i$ and $u_p=1$ otherwise. Now we have a sequence $\mathbf{a}_p\in C_p$ such that $\{\mathbf{a}_p\}_{p\in\bbn}\to\bfx$ but for which $G_{p}(\mathbf{a}_p,u_p)\notin U$ for infinitely many $p$; a violation of Condition (4) We conclude that $G_{\infty}$ is continuous.
\end{proof}

Typically, we will apply the above situation when each $C_p$ is a disjoint union of finitely many $n$-cubes.

\subsection{Shrinking wedges}\label{subsectionshrinkingwedge}

To streamline our application of sequential homotopies, we employ the well-known notion of a shrinking wedge of based spaces.

\begin{definition}
The \textit{shrinking wedge} of countable set $\{(A_j,a_j)\}_{j\in J}$ of based spaces is the space $\sw_{j\in J}(A_j,a_j)$ whose underlying set is the usual one-point union $\bigvee_{j\in J}(A_j,a_j)$ with canonical basepoint $b_0$. A set $U$ is open in $\sw_{j\in J}A_j$ if
\begin{itemize}
\item $U\cap X_j$ is open in $A_j$ for all $j\in J$,
\item and whenever $b_0\in U$, we have $A_j\subseteq U$ for all but finitely many $j\in J$.
\end{itemize}
When the basepoints and/or indexing set are clear from context, we may write the shrinking wedge as $\sw_{J}A_j$. The special case $\bbh_n=\sw_{\bbn}S^n$ is called the \textit{$n$-dimensional Hawaiian earring} (see Figure \ref{fig2}).
\end{definition}

\begin{figure}[H]
\centering \includegraphics[height=2.3in]{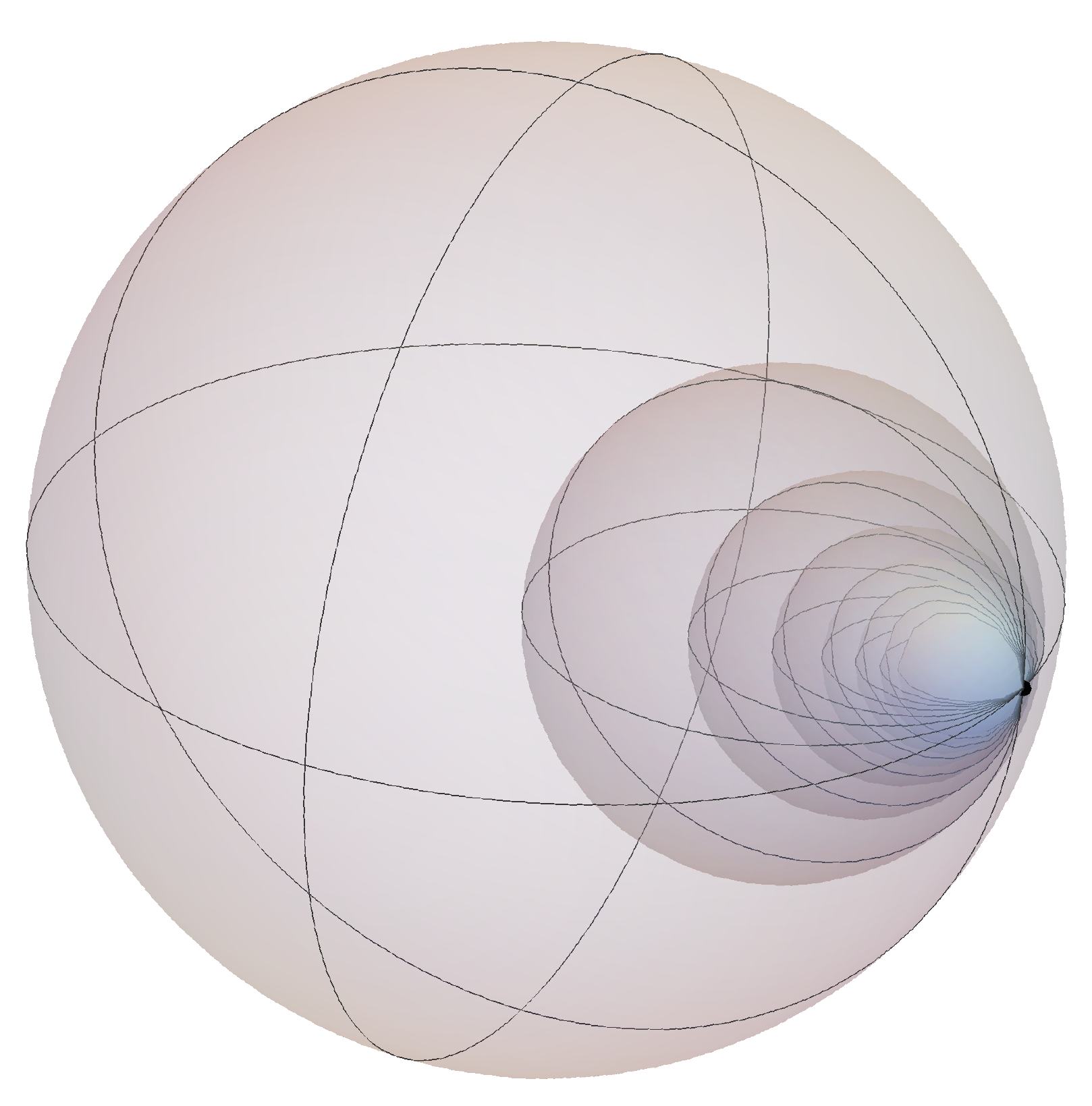}
\caption{\label{fig2}The $2$-dimensional Hawaiian earring $\bbh_2$, which appears in the work of Barratt-Milnor \cite{BarrattMilnor}.}
\end{figure}

Typically, we will identify $\bbh_0=\sw_{\bbn}(S^0,1)$ with the space $\{1,1/2,1/3,\dots,0\}$ consisting of a single convergent sequence and basepoint $0$. Since $\Sigma \bbh_{n-1}\cong \bbh_n$ for all $n\in\bbn$, $[(\bbh_1,b_0),(X,x)]$ is a group and $[(\bbh_n,b_0),(X,x)]$ is an abelian group for all $n\geq 2$. Some authors have referred to these groups as ``Hawaiian groups" \cite{BMM,krhawaiiangroups,krnoncontractible}. If $\ell_j:S^n\to \bbh_n$ denotes the inclusion of the $j$-th sphere, then there is a canonical function $\Theta_n:[(\bbh_n,b_0),(X,x)]\to \pi_n(X,x)^{\bbn}$ given by $\Theta_n([f])=([f\circ \ell_1],[f\circ \ell_2],[f\circ \ell_3],\dots)$, which is a homomorphism for $n\geq 1$.


\begin{remark}\label{mappingspaceremark}
If $(X,x)$ is a based space and $f:(\sw_{\bbn}X,b_0)\to (Y,y)$ is a based map, then the restriction $f_j:X\to Y$ to the $j$-th summand in the shrinking wedge forms a sequence $\{f_j\}_{j\in \bbn}$ that converges to $y$. Conversely, every sequence $\{f_j\}_{j\in\bbn}$ in $(Y,y)^{(X,x)}$ that converges to $y$ uniquely determines a map $f:\sw_{\bbn}X\to Y$, which is defined as $f_j$ on the $j$-th summand. If $X$ is locally compact Hausdorff, then exponential laws for mapping spaces give the existence of a canonical homeomorphism $(Y,y)^{(\sw_{\bbn}X,b_0)}\cong ((Y,y)^{(X,x)},c_{y})^{(\bbh_0,0)}$ of mapping spaces. In the case $X=S^n$, this means $(Y,y)^{(\bbh_n,b_0)}\cong (\Omega^n(Y,y),c_{y})^{(\bbh_0,0)}$
\end{remark}

\subsection{Sequential connectedness properties}\label{subsectionsequentialprops}

Recall the $k$-dimensional Hawaiian earring $\bbh_k$ and function $\Theta_n:[(\bbh_n,b_0),(X,x)]\to \pi_k(X,x)^{\bbn}$ from the previous section.

\begin{definition}\label{scndefinition}
Let $X$ be a topological space and $n\geq 0$.
\begin{enumerate}
\item We say a space $X$ is \textit{sequentially $n$-connected at }$x\in X$ if $[(\bbh_k,b_0),(X,x)]$ is trivial for all $0\leq k\leq n$.
\item We say a space $X$ is \textit{$\pi_n$-residual at} $x\in X$ if $\Theta_n:[(\bbh_n,b_0),(X,x)]\to \pi_k(X,x)^{\bbn}$ is injective.
\item We say a space $X$ is \textit{$\pi_n$-finitary at} $x\in X$ if $\Theta_n([(\bbh_n,b_0),(X,x)])$ lies in the weak direct product $\pi_n(X,x)^{(\bbn)}$ consisting of finitely supported functions $\bbn\to\pi_n(X,x)$.
\item We say a space $X$ is \textit{$n$-tame at} $x\in X$ if, for all $0\leq k\leq n$, $X$ is $\pi_n$-residual and $\pi_n$-finitary at $x$.
\end{enumerate}
Let $\mathscr{P}$ be one of the four pointed properties above. We say a based space $(X,x)$ has property $\mathscr{P}$ if and only if $X$ has property $\mathscr{P}$ at $x$. We say an unbased space $X$ has property $\mathscr{P}$ if and only if $X$ has property $\mathscr{P}$ at all of its points.
\end{definition}

Certainly, we have the following implications.

\[\xymatrix{
&& \pi_k\text{-residual, }0\leq k\leq n\\
\text{sequentially }n\text{-connected} \ar@/^1.5pc/@{=>}[r] & n\text{-tame}  \ar@/^1.5pc/@{=>}[l]^-{+n\text{-connected}} \ar@{=>}[ur] \ar@{=>}[dr]\\
&& \pi_k\text{-finitary},\, 0\leq k\leq n
}\]

All four of the above properties are sequential analogues of known properties. As properties of \textit{based} spaces, each one is an invariant of pointed-homotopy type but not of (free) homotopy type. Frequently, we will understand the these properties in terms of convergent sequences of $n$-loops. We take a moment to explain this connection for each property.

\begin{remark}[Sequential $n$-connectedness]\label{equivremark1}
A space $X$ is sequentially $n$-connected at $x$ if and only if every sequence of based maps $f_j:S^k\to X$, $j\in\bbn$ that converges to $x\in X$ is sequentially null-homotopic rel. basepoint. By considering eventually constant sequences, it is clear that if $X$ is path connected and sequentially $n$-connected at some point, then $X$ is $n$-connected. The sequentially $n$-connected property is similar, but not equivalent, to the conjoined property ``locally $n$-connected and $n$-connected."
\end{remark}

\begin{remark}[$\pi_n$-residual spaces]
A space $X$ is $\pi_n$-residual at $x$ if and only if every sequence $\{f_j\}_{j\in\bbn}$ of $n$-loops in $\Omega^n(X,x)$ that converges to $x$ and where each $f_j$ is null-homotopic in $X$ is sequentially null-homotopic rel. basepoint. This property is the sequential version of the $n$-dimensional analogue of the ``$1$-$UV_0$ property" introduced in \cite{BFTestMap}. This property has little to do with the groups $\pi_n(X,x)$ and more to do with the existence of ``enough small null-homotopies" near $x$. 
\end{remark}

\begin{remark}[$\pi_n$-finitary spaces]
A space $X$ is $\pi_n$-finitary at $x$ if and only if for every sequence of $n$-loops $f_j:S^n\to X$, $j\in\bbn$ that converges to $x\in X$, there exists $j_0\in\bbn$ such that $\{f_j\}_{j\geq j_0}$ is sequentially null-homotopic rel. basepoint. This property is the sequential analogue of the $n$-dimensional version of the semilocally simply connected property. Intuitively, $X$ is $\pi_n$-finitary at $x$ if there are no non-trivial infinite products in $\pi_n(X,x)$.
\end{remark}

\begin{remark}[$n$-tame spaces]
A space $X$ is $n$-tame at $x$ if and only if for every for all $0\leq k\leq n$ and every sequence of maps $f_j:S^k\to X$, $j\in\bbn$ that converges to $x\in X$, there exists $j_0\in\bbn$ such that $\{f_j\}_{j\geq j_0}$ is sequentially null-homotopic rel. basepoint. Intuitively, if $X$ is $n$-tame, then the groups $\pi_k(X,x)$, $0\leq k\leq n$ are completely ``tame" and the groups $[(\bbh_k,b_0),(X,x)]$, $0\leq k\leq n$ detect nothing more than the usual homotopy groups in the sense that $\Theta_k$ maps $[(\bbh_k,b_0),(X,x)]$ isomorphically onto $\bigoplus_{j\in\bbn}\pi_k(X,x)$. Certainly, manifolds and CW-complexes are $n$-tame at all of their points (for all $n\geq 0$).
\end{remark}

In the next proposition, we identify the local analogues of the $\pi_n$-residual and $\pi_n$-finitary properties. The proof is straightforward so we omit it.

\begin{proposition}
Suppose $X$ is first countable at $x\in X$. Then
\begin{enumerate}
\item $X$ is $\pi_n$-residual at $x$ if and only if for every neighborhood $U$ of $x$, there exist a neighborhood $V$ of $x$ such that \[\ker(\pi_n(V,x)\to \pi_n(X,x))\leq \ker(\pi_n(V,x)\to \pi_n(U,x))\] where the homomorphisms are those induced by inclusion,
\item $X$ is $\pi_n$-finitary at $x$ if and only if there exists a neighborhood $U$ of $x$ such that the homomorphism $\pi_n(U,x)\to \pi_n(X,x)$ induced by inclusion is trivial.
\end{enumerate}
\end{proposition}

The proofs in the remainder of this paper will make clear why sequential properties are far more practical than their local counterparts. Conveniently, three of these four properties agree in the $0$-dimensional case (in the presence of path-connectivity). The following lemma will provide a foundational case for higher dimensional situations later on.

\begin{lemma}\label{zerolemma}
For any space $X$, the following are equivalent:
\begin{enumerate}
\item $X$ is sequentially $0$-connected at $x\in X$,
\item $X$ is path connected and $0$-tame at $x\in X$,
\item $X$ is path connected and $\pi_0$-residual at $X$,
\item every sequence of (unbased) maps $\{f_{k}\}_{k\in\bbn}\subseteq X^{S^0}$ that converges to $x$ is sequentially null-homotopic,
\item for every convergent sequence $\{x_k\}_{k\in\bbn}\to x$, there exists a path $\alpha:I\to X$ such that $\alpha(1/k)=x_k$ and $\alpha(0)=x$.
\end{enumerate}
\end{lemma}

\begin{proof}
(1) $\Rightarrow$ (2) $\Rightarrow$ (3) is clear and (3) $\Rightarrow$ (1) follows from the fact that $\prod_{j\in\bbn}\pi_0(X,x)$ is a one-point set when $X$ is path-connected. Therefore, we focus on the equivalence of (1) with (4) and (5).

(1) $\Rightarrow$ (4) Let $\{f_{k}\}_{k\in\bbn}\subseteq X^{S^0}$ be a sequence that converges to $x$. Set $x_k=f_m(-1)$ and $y_k=f_m(1)$. By assumption, there is a sequence of paths $\alpha_k:I\to X$ from $x$ to $x_k$ that converges to $x$ and a sequence of paths $\beta_k:I\to X$ from $x$ to $x_k$ that converges to $x$. Now $\{\alpha_{k}^{-}\cdot \beta_k\}_{k\in\bbn}$ is a sequence of paths from $x_k$ to $y_k$ that converges to $x$. Since a (free) null-homotopy of a map $S^0\to X$ is precisely a path connecting the points in the image, $\{\alpha_{k}^{-}\cdot \beta_k\}_{k\in\bbn}$ is a sequentially null-homotopy of $\{f_k\}_{k\in\bbn}$. (4) $\Rightarrow$ (5) Suppose $\{x_k\}_{k\in\bbn}\to x$. Define $f_k:S^0\to X$ by $f_{k}(-1)=x_{k}$ and $f_k(1)=x_{k+1}$. Since $\{f_k\}_{k\in\bbn}$ converges to $x$, it is sequentially null-homotopic. In particular, there is a sequence of paths $\alpha_k$ from $x_k$ to $x_{k+1}$ that converges to $x$. Defining $\alpha$ to be the infinite concatenation $\prod_{k\in\bbn}\alpha_k$ gives the desired path. (5) $\Rightarrow$ (1) Suppose $\{x_k\}_{k\in\bbn}\to x$. By assumption, there is a path $\alpha:I\to X$ such that $\alpha(1/k)=x_k$ and $\alpha(0)=x$. Let $\alpha_k$ be a path with $\alpha_k\equiv\alpha|_{[0,1/k]}$. Then $\{\alpha_k\}_{k\in\bbn}$ is a sequence of paths from $x$ to $x_k$ that converges to $x$. Thus $X$ is sequentially $0$-connected.
\end{proof}

\begin{example}\label{nfinitarynonexamples}
We give examples to show that all four of the properties are distinct when $n\geq 1$. Certainly, $(S^n)^m$ is $k$-tame for all $k,m\geq 0$ but $S^n$ is not sequentially $n$-connected since it is not $n$-connected. It follows directly from results in \cite{EK00higher} that the $n$-dimensional Hawaiian earring $\bbh_n$ is sequentially $(n-1)$-connected and $\pi_n$-residual. However, the identity $\bbh_n\to\bbh_n$ shows that $\bbh_n$ is neither $\pi_n$-finitary nor $n$-tame. Consider the cone $C\bbh_{n}=(\bbh_{n}\times \ui)/(\bbh_{n}\times \{1\})$ over the $n$-dimensional Hawaiian earring and let $x_0$ be the image of $(b_0,0)$ in the quotient. Then $C\bbh_{n}$ is contractible and therefore $k$-connected for all $k\geq 0$. Since $\pi_k(C\bbh_n,x_0)$ is trivial, $\im(\Theta_k)=0=(\pi_k(C\bbh_n,x_0))^{(\bbn)}$. However, the inclusion $\bbh_n\to C\bbh_n$ does not extend to a basepoint-preserving null-homotopy. Thus $[(\bbh_n,b_0),(C\bbh_n,x_0)]\neq 0$ and $\Theta_n$ is not injective. Hence, for all $n\geq 0$, $C\bbh_n$ is $\pi_n$-finitary but not $\pi_n$-residual at $x_0$.
\end{example}

Just as $n$-connectedness plays a prominent role in standard homotopy group computations, sequentially $n$-connectedness will play a prominent role in this paper. We identify some topological constructions that preserve this property.

\begin{example}[Directed Limits]
If $X=\bigcup_{i\in\bbn}X_i$ has the weak topology with respect to subspaces $X_1\subseteq X_2\subseteq X_3\subseteq\cdots$ where each $X_i$ is $n$-tame, then $X$ is $n$-tame. Indeed, if $\{f_m\}_{m\in\bbn}$ is a sequence in $\Omega^k(X,x)$ that converges to $x\in X$, then $\bigcup_{m\in\bbn}\im(f_m)$ is a compact subspace of $X$ and therefore lies in $X_i$ for some $i$. Thus a cofinal subsequence of $\{f_m\}_{m\in\bbn}$ is sequentially null-homotopic in $X_i$. Since a CW-complex has the weak topology with respect to its finite subcomplexes (all of which are first countable), a similar argument shows that all CW-complexes are $n$-tame.
\end{example}

\begin{example}[Infinite Product Spaces]
The $n$-tame property is closed under taking finite direct products but not under infinite direct products, e.g. $S^n$ is $k$-tame for all $k\geq 0$ but $(S^n)^{\bbn}$ is not $n$-tame at any of its points since it is not $\pi_n$-finitary. However, if $\{(X_j,x_j)\mid j\in J\}$ is a family of sequentially $n$-connected based spaces, then the direct product $\prod_{j\in J}X_j$ is sequentially $n$-connected at $(x_j)_{j\in J}$.
\end{example}

\begin{example}[Unreduced Cones and Suspensions]\label{coneexample}
If $X$ is sequentially $n$-connected at $x_0\in X$, then the unreduced cone $CX=X\times \ui/X\times\{1\}$ is sequentially $n$-connected at all points in the arc $\{(x_0,t)\mid t\in\ui\}\subseteq CX$. Since $CX$ contracts by a pointed null-homotopy at the vertex point $v\in CX$, this is clear when $t=1$. If $s_0=(x_0,t)$ for $0\leq t<1$, there is a deformation retraction of $CX\backslash\{v\}$ onto $X\times\{t\}$. Let $r:CX\backslash\{v\}\to X\times \{t\}$ be the corresponding retraction and consider a sequence $\{f_j\}_{j\in\bbn}$ in $\Omega^n(CX,s_0)$ that converges to $s_0$. There exists $j_0\in\bbn$ such that $\im(f_j)\subseteq CX\backslash\{v\}$ for all $j\geq j_0$. If $j<j_0$, then $f_j$ is null-homotopic since $CX$ is contractible. Applying $r$, we see that $\{f_j\}_{j\geq j_0}$ is sequentially homotopic to $\{r\circ f_j\}_{j\geq j_0}$ in $CX$, the later of which is a sequence in $X\times\{t\}\subseteq CX$ that converges to $(x_0,t)$. Hence, if $X$ is sequentially $n$-connected at $x_0$, then $\{r\circ f_j\}_{j\geq j_0}$ is sequentially null-homotopic in $X\times\{t\}$. Thus $\{f_j\}_{j\in\bbn}$ is sequentially null-homotopic in $CX$.

Applying a similar argument, it follows that if $X$ is sequentially $n$-connected at $x_0\in X$, then the unreduced suspension $SX=CX/X\times\{0\}$ is also sequentially $n$-connected at all points in the arc $\{(x_0,t)\mid t\in\ui\}\subseteq SX$.
\end{example}

The next example plays a key role in our main result.

\begin{example}[Dendrites]\label{dendriteexample}
A \textit{dendrite} is a Peano continuum $D$ such that for any distinct points $a,b\in D$, there is a unique arc $A\subseteq D$ with endpoints $\{a,b\}$. Equivalently, a Peano continuum $D$ is a dendrite if and only if it contains no subspaces homeomorphic to $S^1$. There are many other useful characterizations of dendrites from Continuum Theory \cite[Chapter 10]{Nadler}, some of which we will call upon later. Every dendrite is contractible and has Lebesgue covering dimension $1$. Moreover every compact connected subset $D'$ of a dendrite $D$ is itself a dendrite and is a deformation retract of $D$. As a consequence, every map $f:S^n\to D$, $n\geq 1$ contracts by a homotopy with image in $\im(f)$. It follows that every dendrite is sequentially $n$-connected for all $n\geq 0$.
\end{example}

The following proposition illustrates the importance of the sequentially $0$-connected property. In the presence of this property, if one wishes to show that a sequence of maps is null-homotopic rel. basepoint, then it suffices to produce an \textit{unbased} sequential null-homotopy.

\begin{proposition}\label{cofibrationprop}
Suppose $(Y,y_0)$ is sequentially $0$-connected. If $(X,x_0)$ is $x_0$ well-pointed, i.e. if $\{x_0\}\to X$ is a cofibration, then every sequence $\{f_k\}_{k\in\bbn}$ in $ Y^X$ that converges to $y_0$ is sequentially homotopic to a sequence of based maps $\{g_k\}_{k\in\bbn}$ in $(Y,y_0)^{(X,x_0)}$ that converges to $y_0$.
\end{proposition}

\begin{proof}
Suppose $\{f_k\}_{k\in\bbn}\subseteq Y^X$ converges to $y$. Since $\{f_k(x_0)\}\to y_0$ in $Y$, by Lemma \ref{zerolemma} and $Y$ is sequentially $0$-connected at $y_0$, there is a sequence of paths $\alpha_k:\ui\to Y$ from $y_0$ to $f_k(x_0)$ that converges to $y_0$. Since $(X,x_0)$ is well-pointed, there is a retraction $r:X\times I\to X\times \{0\}\cup \{x_0\}\times I$. Define a map $\beta_k:X\times \{0\}\cup \{x_0\}\times I\to Y$ by $\beta(x,0)=f_k(x)$ and $\beta_k(x_0,t)=\alpha_k(t)$. Now define $g_k:X\to Y$ by $g_k(x)=\beta_k(r(x,1))$. Notice that $g_k(x_0)=y_0$ and that $g_k$ is freely homotopic to $f_k$ by a homotopy with image in $\im(\alpha_k)\cup \im(f_k)$. Therefore, $f_k$ and $g_k$ are sequentially homotopic.
\end{proof}

The following lemma is an immediate consequence.

\begin{lemma}\label{scnequivalencelemma}
Suppose $(X,x)$ is sequentially $0$-connected and $n\geq 1$. Then
\begin{enumerate}
\item $(X,x)$ is sequentially $n$-connected if and only if for all $0\leq k\leq n$, every sequence of unbased maps $f_j:S^k\to X$, $j\in\bbn$ that converges to $x$ is sequentially null-homotopic.
\item $(X,x)$ is $n$-tame at $x\in X$ if and only if for all $0\leq k\leq n$ and every sequence of unbased maps $f_j:S^k\to X$, $j\in\bbn$ that converges to $x$, there exists $j_0\in \bbn$ such that $\{f_j\}_{j\geq j_0}$ is sequentially null-homotopic.
\end{enumerate}
\end{lemma}

It was mentioned previously that the $n$-tame property is similar to the usual notion of local $n$-connectedness. Certainly, they are closely related. Formally, they are not equivalent even among metric spaces. In light of Lemma \ref{scnequivalencelemma}, it becomes clear that the $n$-tame property is more closely related to the point-wise ``$UV^n$ property" \cite{SakaiGAGT}: a space $X$ is said to be $UV^n$ at $x\in X$ if for every $0\leq k\leq n$ each neighborhood $U$ of $x$ contains a neighborhood $V$ of $x$ such that every map $f:S^k\to V$ is null-homotopic in $U$. 

\begin{proposition}
Suppose $X$ is first countable and locally path-connected at $x\in X$. Then
\begin{enumerate}
\item $X$ is $n$-tame at $x$ if and only if $X$ is $UV^n$ at $x$,
\item  $X$ is sequentially $n$-connected at $x$ if and only if $X$ is $n$-connected and $X$ is $UV^n$ at $x$.
\end{enumerate}
\end{proposition}

It is possible to construct non-first countable spaces that are $n$-tame at a point $x$ but not $UV^n$ at $x$. 

\section{Shrinking Adjunction Spaces}\label{sectionshrinkingadjspaces}

\subsection{Definition and examples}\label{subsectionsasdefsandexamples}

In this section, we focus on a generalization of shrinking wedges that will play a key role in both the motivation for and proof of our main results. 

\begin{definition}[Shrinking adjunction spaces]\label{shrinkingadjunctiondef}
Consider a space $X$ and a sequence $\{x_j\}_{j\in\bbn}$ (of not necessarily distinct elements) in $X$. Let $\{(A_j,a_j)\}_{j\in\bbn}$ be a sequence of based spaces. Let $Y_k$ be the adjunction space obtained by attaching $A_j$ to $X$ by $a_j\sim x_j$ for all $j\leq k$ (with the weak topology). There are retractions $\rho_{k+1,k}:Y_{k+1}\to Y_{k}$ that collapse $A_{k+1}$ to $x_{k+1}$ and are the identity elsewhere. The \textit{shrinking adjunction space obtained by attaching the sequence $\{(A_j,a_j)\}_{j\in\bbn}$ to $X$ along the sequence} $\{x_j\}_{j\in J}$ is the inverse limit space $Y=\varprojlim_{k\in\bbn}(Y_k,\rho_{k+1,k})$ topologized as a subspace of the direct product $\prod_{k\in\bbn}Y_k$. The projection maps will be denoted $\rho_k:Y\to Y_k$. 

We refer to $X$ as the \textit{core} of $Y$, the spaces $A_j$, $j\in\bbn$ as the \textit{attachment spaces}, and the sequence $\{x_j\}_{j\in\bbn}$ as the \textit{attachment points}. When it is necessary to refer to this specific decomposition of $Y$, we may write $Y=\shadj(X,\{x_j\}_{j\in\bbn},\{A_j\}_{j\in\bbn},\{a_j\}_{j\in\bbn})$ or just $Y=\shadj(X,x_j,A_j,a_j)$ when the indexing set is clear from context.
\end{definition}

Shrinking adjunction spaces are ordinary, finite adjunction spaces when $A_j=\{a_j\}$ for all but finitely many $j\in\bbn$. Additionally, a shrinking adjunction spaces is precisely a shrinking wedge $Y=\sw_{j\in\bbn}A_j$ when the core $X$ is a one-point space. Taking $Y$ as a subspace of $\prod_{k\in\bbn}Y_k$, we may identify the underlying set of the space $Y$ in Definition \ref{shrinkingadjunctiondef} with the underlying set of the ordinary adjunction space $\left(X\sqcup \coprod_{j\in\bbn}A_j\right)/\mathord{\sim}$ where $x_j\sim a_j$. For instance, under this identification, $A_j$ corresponds to $\left(\prod_{k=1}^{j-1}\{x_j\}\times A_j\times \prod_{k>j}Y_k\right)\cap Y$ and $X$ corresponds to $\prod_{k\in\bbn}X=\left(X\times \prod_{k>1}Y_k\right)\cap Y $. In the shrinking adjunction space $Y$, we may consider the $A_j$ as being ``shrinking in size." To make this more precise and more intuitive, we give a useful characterization of the inverse limit topology in Lemma \ref{opensetlemma}. Note that if $X$ is separable, then it is possible that $\{x_j\mid j\in\bbn\}$ is dense in $X$. Additionally, if $j_1<j_2<j_3<\cdots $ is such that $\{x_{j_i}\}_{i\in\bbn}$ is the constant sequence at $x\in X$, then this subsequence will result in a copy of the shrinking wedge $\sw_{i\in\bbn}A_{j_i}$ attached at $x$.

\begin{remark}\label{distinctpointsremark}
Identifying a given space $Y$ as a shrinking adjunction space is a choice of decomposition of $Y$ and will not be unique for a given space. For example, one could view $Y=\shadj(X,x_j,A_j,a_j)$ as $Y=\shadj\left(X\cup \bigcup_{i\in\bbn}A_{2i-1},x_{2i},A_{2i},a_{2i}\right)$, i.e. with core $X\cup \bigcup_{i\in\bbn}A_{2i-1}$ and attachment spaces $\{A_{2i}\}_{i\in\bbn}$. Moreover, if the core $X$ is a shrinking adjunction space, then there may be many distinct, relevant choices of decompositions of $Y$.

In general, we allow for the sequence $\{x_j\}_{j\in\bbn}$ to have repetitions and even have finite image. However, if we replace the sequences of attachment spaces $A_j$ with finite or shrinking infinite wedges of these spaces, then without changing the homeomorphism type of $Y$, we may always assume that $\{x_j\}_{j\in\bbn}$ is an infinite sequence of distinct points in $X$. Formally, let $Z=\{x_j\mid j\in\bbn\}$ and note that $Z=\{z_1,z_2,\dots\}$ may be finite or infinite. Define $T_m=\{j\in\bbn\mid x_j=z_m\}$. If $T_m$ is finite, let $B_m=\bigvee_{j\in T_m}(A_j,a_j)$ and if $T_m$ is infinite, let $B_m=\sw_{j\in T_j}(A_j,a_j)$. Take the basepoint $b_m\in B_m$ to be the wedge point in each case. If $Z$ is finite and written as $\{z_1,z_2,\dots z_M\}$, recursively choose any $z_{m}\in X\backslash \{z_1,z_2,\dots ,z_{m-1}\}$ for $m>M$ and, for all such $m$, let $B_m $ be a one-point space. Then the shrinking adjunction space obtained by attaching the sequence $\{(B_m,b_m)\}_{m\in\bbn}$ to the sequence $\{z_m\}_{m\in\bbn}$ of distinct points in $X$ is homeomorphic to the original shrinking adjunction space $Y$. For separable $X$, we may add more one-point attachment spaces to the sequence so that the sequence $\{x_j\}_{j\in J}$ is dense in $X$.
\end{remark}

\begin{remark}[Indexing by countable sets]
For any bijection $\sigma:\bbn\to\bbn$, basic properties of inverse limits give a homeomorphism $\shadj(X,x_j,A_j,a_j)\cong \shadj(X,x_{\sigma(j)},A_{\sigma(j)},a_{\sigma(j)})$ that permutes the coordinates of the inverse limit. Therefore, we are justified in replacing the indexing set $\bbn$ with any countably infinite set. 
\end{remark}

\begin{lemma}\label{opensetlemma}
Let $Y=\shadj(X,x_j,A_j,a_j)$ where $X$ and each $A_j$ is compact. Then a set $U\subseteq Y$ is open if and only if
\begin{enumerate}
\item $X\cap U$ is open in $X$,
\item $A_j\cap U$ is open in $A_j$ for all $j\in \bbn$,
\item whenever $\{x_{j_i}\}_{i\in\bbn}\to x$ in $X$ and $x\in U$, we have $A_{j_i}\subseteq U$ for all but finitely many $i\in\bbn$. 
\end{enumerate}
\end{lemma}

\begin{proof}
Note that the conditions (1)-(3) define a topology on the underlying set of $Y$. Let $Y_{sub}$ denote the resulting topological space and consider the identity function $i:Y_{sub}\to Y$. If $V$ is open in $Y_k$, then $\rho_{k}^{-1}(V)$ is the union of $V\cap X$, $V\cap A_j$ for $j\leq k$ and $\bigcup\{A_j\mid j>k\text{ and }x_j\in V\cap X\}$. Then $\rho_{k}^{-1}(V)$ satisfies (1)-(3) and is therefore open in $Y_{sub}$. Therefore, the projections $\rho_{k}\circ i:Y_{sub}\to Y_k$, $k\in\bbn$ are continuous. Since $Y$ is an inverse limit, it follows that $i:Y_{sub}\to Y$ is continuous.

Since $Y$ is Hausdorff, to ensure $i$ is a closed map, it suffices to check that $Y_{sub}$ is compact. Clearly, $Y_{sub}$ is Hausdorff and since the inclusions $X\to Y_{sub}$ and $A_j\to Y_{sub}$ are continuous, $X$ and each $A_j$ are compact subspaces of $Y_{sub}$. Let $\scru$ be an open cover of $Y_{sub}$. Find finitely many sets $U_1,U_2,\dots,U_M$ in $\scru$ that cover $X$. Additionally, finitely many sets in $\scru$ are required to cover each $A_j$ and so to find a finite subcover of $Y_{sub}$ it suffices to show that $A_j\subseteq \bigcup_{m=1}^{M}U_m$ for all but finitely many $j$. Suppose there are $j_1<j_2<j_3<\cdots$ and points $z_{j_i}\in A_j\backslash \bigcup_{m=1}^{M}U_m$. Now $\{x_{j_i}\}$ is a sequence in $X$. Since $X$ is compact, we may replace $\{x_{j_i}\}$ with a subsequence that converges to a point $x\in X$. Find a set $U_{m_0}$ containing $x$. By Condition (3), we have $A_{j_i}\subseteq U_{m_0}$ for all but finitely many $i$; a contradiction of the assumption that $z_{j_i}\notin \bigcup_{m=1}^{M}U_m$ for all $i\in\bbn$.
\end{proof}

We will refer to the topology on the underlying set of $Y$ determined by Conditions (1)-(3) in Lemma \ref{opensetlemma} as the \textit{subsequence topology} and continue to denote this space as $Y_{sub}$. Typically, we will assume a shrinking adjunction space $Y$ has the inverse limit topology. However, sometimes it will be more convenient to use the subsequence topology. For the remainder of this section, we have need to distinguish the two topologies on $Y$ and so we let $Y_{lim}$ denote $Y$ with the inverse limit topology. The author expects for these two topologies to agree for many non-compact situations which lie outside of the scope of Lemma \ref{opensetlemma}. However, a general investigation into the conditions for which these two topologies agree is a tedious endeavor and so we do not pursue it here. Despite this, we are able to provide a short proof that $Y_{lim}$ and $Y_{sub}$ always share the same Peano continuum subspaces. Therefore, if one is only interested in sets of homotopy classes of maps $K\to X$ from Peano continua $K$, then one may freely use either topology.

\begin{remark}\label{pcremark}
The shrinking adjunction space $Y_{lim}=\shadj(X,x_j,A_j,a_j)$ is a Peano continuum if and only if $X$ and each $A_j$ is a Peano continuum. One direction is clear since $X$ and each $A_j$ are retracts of $Y$. Conversely, if $X$ and each $A_j$ is a Peano continuum, then $Y_{lim}=Y_{sub}$ by Lemma \ref{opensetlemma}. It then suffices to show $Y_{sub}$ is a Peano continuum. This may be done by verifying the individual characterizing properties of a Peano continua or by constructing an onto path $I\to Y_{sub}$ using a choice of onto paths $I\to X$ and $(I,0)\to (A_j,a_j)$, $j\in \bbn$.
\end{remark}

In categorical language, the next proposition states that $Y_{sub}$ and $Y_{lim}$ have homeomorphic $\Delta$-generated coreflections\footnote{The $\Delta$-generated coreflection of $X$ is the space $\Delta(X)$ with the same underlying set as $X$ and with the topology of sets $U\subseteq X$ such that, for every map $f:K\to X$ from a Peano continuum $K$, $f^{-1}(U)$ is open in $K$.}.

\begin{proposition}
If $K$ is Peano continuum then a function $f:K\to Y_{sub}$ is continuous if and only if $f:K\to Y_{lim}$ is continuous.
\end{proposition}

\begin{proof}
In the proof of Lemma \ref{opensetlemma}, we showed that the identity function $i:Y_{sub}\to Y_{lim}$ is always continuous. The ``only if" direction follows. For the ``if" direction, suppose $f:K\to Y_{lim}$ is continuous. Let $Z=f(K)$ and let $Z_{lim}$ denote $Z$ with the subspace topology inherited from $Y_{lim}$. Then $Z_{lim}$ is a Peano continuum as a subspace of $Y_{lim}$. The canonical maps $r:Y\to X$ and $r_j:Y\to A_j$ are retractions with respect to both of the spaces $Y_{sub}$ and $Y_{lim}$. Hence, each space $Y_k$, $k\in\bbn$ is a retract of both $Y_{lim}$ and $Y_{sub}$. Therefore, if $Z_{lim}$ lies entirely the union of $X$ and finitely many $A_j$, then $f:K\to Y_{sub}$ is continuous. Therefore, we may assume that $Z$ meets $X$ and $J=\{j\in\bbn\mid Z\cap A_{j}\backslash\{a_j\}\neq\emptyset\}$ is infinite. Since $r$ and $r_j$ are retracts, $r(Z_{lim})=Z_{lim}\cap X$ and $r_j(Z_{lim})=Z_{\lim}\cap A_j$, $j\in J$ are Peano continua in $Y_{lim}$. If we let $Z_k=r(Z_{lim})\cup \bigcup\{r_j(Z_{lim})\mid 1\leq j\leq k, j\in J\}$, $k\in J$, then $Z_{lim}=\varprojlim_{k\in J}(Z_k,\rho_{k+1,k}|_{Z_{k+1}})$. Therefore, $Z_{lim}=\shadj(r(Z_{lim}),\{x_j\}_{j\in J},\{r_j(Z_{lim})\}_{j\in J},\{a_j\}_{j\in J})$. By Lemma \ref{opensetlemma}, we have $Z_{sub}=Z_{lim}$. The subspace topology on $Z$ inherited from $Y_{sub}$ agrees with the subsequence topology on $Z$ and according to Remark \ref{pcremark}, $Z_{sub}=Z_{lim}$. Therefore, $f:K\to Z_{sub}\subseteq Y_{sub}$ is continuous.
\end{proof}

\begin{corollary}
For every Peano continuum $X$, $A\subseteq X$, and $y\in Y$, the identity function $i:Y_{sub}\to Y_{lim}$ induces a bijection $[(X,A),(Y_{sub},y)]\to [(X,A),(Y_{lim},y)]$. In particular, the induced function $[(\bbh_n,b_0),(Y_{sub},y)]\to [(\bbh_n,b_0),(Y_{lim},y)]$ is bijective for all $n\geq 1$.
\end{corollary}

\begin{remark}\label{opensetremark}
An important feature of $Y_{sub}$ is that whenever $z_i\in A_{j_i}$, $i\in\bbn$ is a sequence that converges in $Y$ to a point $x\in X$, it must also be the case that the corresponding sequence of attachment points $\{x_{j_i}\}_{i\in\bbn}=\{a_{j_i}\}_{i\in\bbn}$ also converges to $x$ in $X$. This follows from the fact that if $U$ is a neighborhood in $Y$ that meets $X$, then $(U\cap X)\cup \bigcup\{A_j\cap U\mid x_j\in U\}$ is also an open set in $Y$.
\end{remark}

In Section \ref{subsectionshadjfactorization}, we will show that $Y$ is sequentially $n$-connected whenever its core and attachment spaces are sequentially $n$-connected. For $n\geq 2$, this proof is substantially more complicated and requires the use of Whitney cubes. We prove the case $n=0$ here.

\begin{lemma}\label{adjunctionlemmazero}
Let $Y=\shadj(X,x_j,A_j,a_j)$ and $y_0\in X$. If $(X,y_0)$ is sequentially $0$-connected and $(A_j,a_j)$ is sequentially $0$-connected for all $j\in \bbn$, then $(Y,y_0)$ is sequentially $0$-connected.
\end{lemma}

\begin{proof}
Since $X$ and each $A_j$ is sequentially $0$-connected at at least one point, all of these spaces are path connected. Let $\{y_k\}_{k\in\bbn}\to y_0$ be a convergent sequence in $Y$. Note that $\{r_k(y_k)\}_{k\in\bbn}$ is a sequence in $X$ that converges to $y_0$. Since $(X,y_0)$ is sequentially $0$-connected, there is a sequence of paths $\alpha_k:I\to X$ from $y_0$ to $r_k(y_k)$ that converges to $y_0$. We construct a sequence of paths $\beta_k:I\to Y$ from $y_k$ to $r_k(y_k)$ that converges to $y_0$. Once this is done, the sequence $\{\alpha_k\cdot\beta_k\}_{k\in\bbn}$ allows us to conclude that $(Y,y_0)$ is sequentially $0$-connected.

We define $\beta_k$ as follows. If $y_k\in X$, then $r(y_k)=y_k$ so we let $\beta_k$ be the constant path at $y_k$. For each $j\in \bbn$, let $T_j=\{k\in\bbn\mid y_k\in A_j\backslash \{a_j\}\}$. If $T_j$ is finite, then for every $k\in T_j$, we choose any path $\beta_k:I\to A_j$ from $a_j=r(y_k)$ to $y_k$. If $T_j$ is infinite, then, because $(A_j,a_j)$ is sequentially $0$-connected, we may choose a sequence of paths $\{\beta_k\}_{k\in T_j}$ in $A_j$ from $a_j=r(y_k)$ to $y_k$ that converges to $a_j$.

With $\beta_k$ defined for all $k\in\bbn$ it suffices to show that $\{\beta_k\}_{k\in\bbn}$ converges to $y_0$. Let $U$ be an open neighborhood of $y_0$ in $Y$. Since $\{r(y_k)\}_{k\in\bbn}\to y_0$, there exists $K\in\bbn$ such that $r(y_k)\in U$ for all $k\geq K$. Suppose there exists $K<k_1<k_2<k_3<\cdots$ and points $t_i\in\ui$ such that $\beta_{k_i}(t_i)\notin U$. Considering our choice of the paths $\beta_k$, this only occurs if $k_{i}\in T_{j_i}$ for some $j_i\in \bbn$. However, $\{r(y_{k_i})\}_{i\in\bbn}=\{a_{j_i}\}_{i\in\bbn}$ converges to $y_0$ and therefore $A_{j_i}\subseteq U$ for all but finitely many $i\in\bbn$. However, this contradicts the fact that $\beta_{k_i}(t_i)\in A_{j_i}\backslash U$ for all $i\in\bbn$.
\end{proof}

\begin{corollary}
If $(A_j,a_j)$ is sequentially $0$-connected for all $j\in \bbn$, then $\sw_{j\in \bbn}(A_j,a_j)$ is sequentially $0$-connected at the wedgepoint.
\end{corollary}

We will often take a subspace $Y'$ of $Y=\shadj(X,x_j,A_j,a_j)$ and wish to view $Y'$ as a shrinking adjunction space. The proof of the following lemma follows from basic properties of inverse limits.

\begin{lemma}\label{adjunctionsubspaces}
Let $Y=\shadj(X,x_j,A_j,a_j)$, $X'\subseteq X$ be a subspace, and suppose $J=\{j\in\bbn\mid x_j\in X'\}$ is infinite. The subspace topology on $Y'=X'\cup \bigcup_{j\in J}A_j$ agrees with the inverse limit topology on $Y'=\shadj(X',\{x_j\}_{j\in J},\{A_j\}_{j\in J},\{a_j\}_{j\in J})$.
\end{lemma}

\subsection{A Splitting Theorem}\label{subsectionsplittingtheorem}

Here, we begin our analysis of homotopy groups of shrinking adjunction spaces by generalizing a known splitting result for shrinking wedges, namely, \cite[Theorem 2.1]{Kawamurasuspensions}; see also \cite[Lemma 5.3]{Brazasncubeshuffle}. Interestingly, this result holds in all dimensions $n\geq 2$ and does not require application of Whitney covers.

\begin{theorem}\label{splittheorem}
Let $Y=\shadj(X,x_j,A_j,a_j)$ where $X$ is a Peano continuum and $y_0\in X$. Then for all $n\geq 2$, the canonical homomorphism \[\Psi:\pi_n(Y,y_0)\to \pi_n(X,y_0)\oplus\prod_{j\in\bbn}\pi_n(A_j,a_j),\] $\Phi([f])=([r\circ f],[r_1\circ f],[r_2\circ f],\dots)$ is a split epimorphism.
\end{theorem} 

\begin{proof}
Since $X$ is a retract of $Y$, the splitting will follow if we show the restricted homomorphism $\psi :\pi_n(Y,y_0)\to \prod_{j\in\bbn}\pi_n(A_j,a_j)$ of $\Phi_n$ splits. For each $j\in\bbn$, pick a map $f_j\in\Omega^n(A_j,a_j)$. We will construct a map $f\in\Omega^n(Y,y_0)$ such that $r_j\circ f\simeq f_j$ for all $j\in\bbn$. 

Let $\alpha:\ui\to X$ be a surjective path so that $\alpha(0)=x_0$ and $\{x_j\mid j\in\bbn\}\subseteq \alpha([0,1))$. By inserting countably many constant subpaths, we may assume that for each $j\in\bbn$, there is an open interval $(c_j,d_j)\subseteq \ui$ such that $\alpha([c_j,d_j])=x_j$. We parameterize concentric $n$-cubes in $I^n$ as follows: for $0\leq t<1$, let $C_t= [\frac{t}{2},\frac{1-t}{2}]^n$ and let $C_1=\{(1/2,1/2,\dots,1/2)\}$. The $n$-loop $\ell\in \Omega^n(X,y_0)$, which maps $\partial C_t$ to $\alpha(t)$ for $0\leq t<1$ and $\ell(C_1)=\alpha(1)$ for all $t\in\ui$ is null-homotopic and so we modify it to include appropriate portions that map to the spaces $A_j$. 

For each $j\in J$, we have a thickened $n$-sphere $K_j=C_{c_j}\backslash \int(C_{d_j})$, which has boundary $\partial K_j=\partial C_{c_j}\cup\partial C_{d_j}$ and such that $g(K_j)=\alpha([c_j,d_j])=x_j=a_j$. For each $j\in \bbn$, pick an $n$-cube $R_j\subseteq K_j$ (see Figure \ref{symfig}). Define $f$ so that \[f(\bfx)=\begin{cases}
\ell(\bfx) & \text{ if }\bfx\notin \bigcup_{j\in\bbn}R_j\\
f_j\circ L_{R_j,I^n}(\bfx) & \text{ if }\bfx\in R_j
\end{cases}\]

\begin{figure}[h]
\centering \includegraphics[height=3.3in]{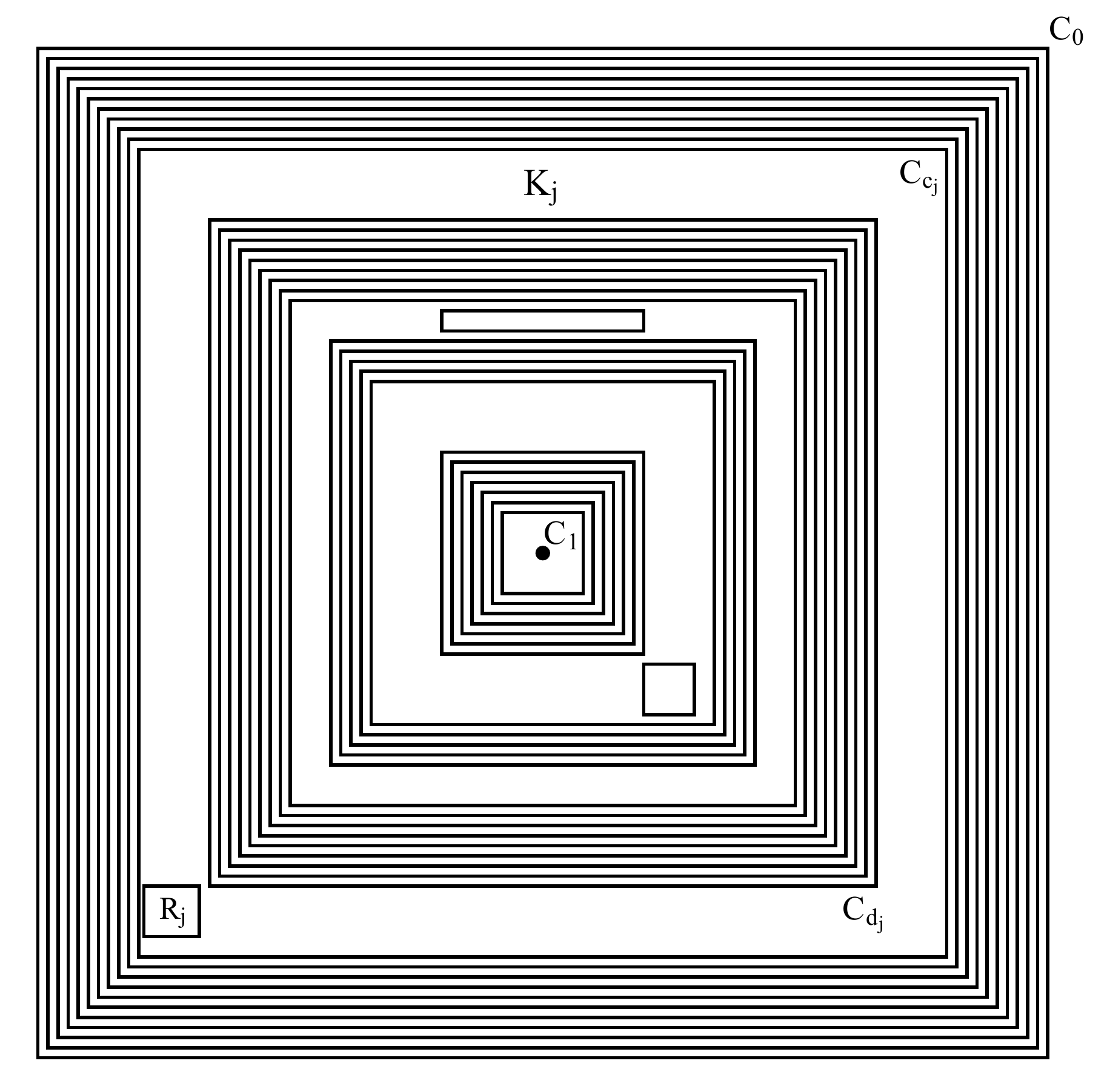}
\caption{\label{symfig} The domain of $f$ in the case $n=2$; $f$ maps $K_j\backslash \int(R_j)$ to $a_j$ and maps $R_j$ into $A_j$ by $f_j$.}
\end{figure}

Certainly, each projection $\rho_{k}\circ f:I^n\to Y_k$ is continuous for each $k\in\bbn$ and therefore $f$ is continuous. Moreover, it is easy to see that $r_j\circ f\simeq  f_j$ for all $j\in\bbn$. While this proves the surjectivity of $\psi$, we must further analyze this construct to ensure that it defines a splitting of $\psi$. In particular, we note that homotopy class of $f$ is independent of our choice of the $n$-cubes $R_j$. Indeed, for each $j\in\bbn$, let $S_j$ be another $n$-cube in $K_j$. Let $f'$ be the map which agrees with $\ell$ on $I^n\backslash \bigcup_{j\in\bbn}S_j$ and $f_{j}'|_{S_j}=f_j\circ L_{S_j,I^n}$. Since $K_j$ is path connected, we may find a continuously parameterized family of $n$-cubes $T_{j}(t)\subseteq K_j$, $t\in\ui$ where $T_j(0)=R_j$ and $T_j(1)=S_j$ (note that the sizes of the cubes may change). Define $H:I^n\times I\to Y$ so that $H(\bfx,t)=\ell(\bfx)$ if $\bfx\notin T_j(t)$ and so that $H(\bfx,t)=f_{j}\circ L_{T_{j}(t),I^n}(\bfx)$ if $\bfx\in T_j(t)$. Again, the finite projections $\rho_k\circ H$ into the spaces $Y_k$ are continuous and therefore $H$ is a continuous homotopy from $f$ to $f'$. In short, the homotopy $H$ simultaneously slides the domain of $f_j$ from $R_j$ to $S_j$ within $K_j$. It is also possible to make the same observation using at most $n$ applications of Lemma \ref{shrinkingcubelemma}.

The function $\kappa(([f_j])_{j\in\bbn})=[f]$ defines a set-theoretic section to $\psi $. We check that $\kappa$ is a homomorphism. Consider another family of $n$-loops $g_j\in \Omega(A_j,a_j)$ and a map $g$ constructed, as above, so that $\kappa(([g_j])_{j\in\bbn})=[g]$. According to the previous paragraph we may choose the domains in a convenient manner: For all $j\in\bbn$, choose $R_j=\left[\frac{c_j}{2},\frac{d_j}{2}\right]^n$ as the domain of $f|_{R_j}\equiv f_j$ and notice that each $R_j$ lies in $[0,1/2]\times I^{n-1}$. Choose $S_{j}=\left[\frac{1-d_j}{2},\frac{1-c_j}{2}\right]\times \left[\frac{c_j}{2},\frac{d_j}{2}\right]^{n-1}$ as the domain of $g|_{S_j}\equiv g_j$ and notice that each $S_j$ lies in $[1/2,1]\times I^{n-1}$. 

Consider the concatenation $f\cdot g$. We apply a homotopy similar to the standard homotopy used to prove the general property $\beta\ast (f\cdot g)\simeq (\beta\ast f)\cdot (\beta\ast g)$ of the $\pi_1$-action on $\pi_n$. In particular, define $G:I^n\times I\to X$ as \[G(s_1,s_2,\dots,s_n,t)=\begin{cases}
f((2-t)s_1,s_2,\dots ,s_n) & \text{ if }s_1\in [0,1/2]\\
g((2-t)s_1+t-1,s_2,\dots ,s_n) & \text{ if }s_1\in [1/2,1]
\end{cases}\]
Since $f$ agrees with $\ell$ on $[1/2,1]\times I^{n-1}$ and $g$ agrees with $\ell^{-}$ on $[0,1/2]\times I^{n-1}$, $G$ is well-defined. Let $h\in\Omega^n(Y,y_0)$ be the $n$-loop, $h(\bfx)=G(\bfx,1)$. Then $h$ agrees with $\ell$ on $I^n\backslash \bigcup_{j\in \bbn}R_j\cup S_j$, and we have $h|_{R_j}\equiv f_j$ and $h|_{S_j}\equiv g_j$ for all $j\in\bbn$.
\begin{figure}[h]
\centering \includegraphics[height=1.6in]{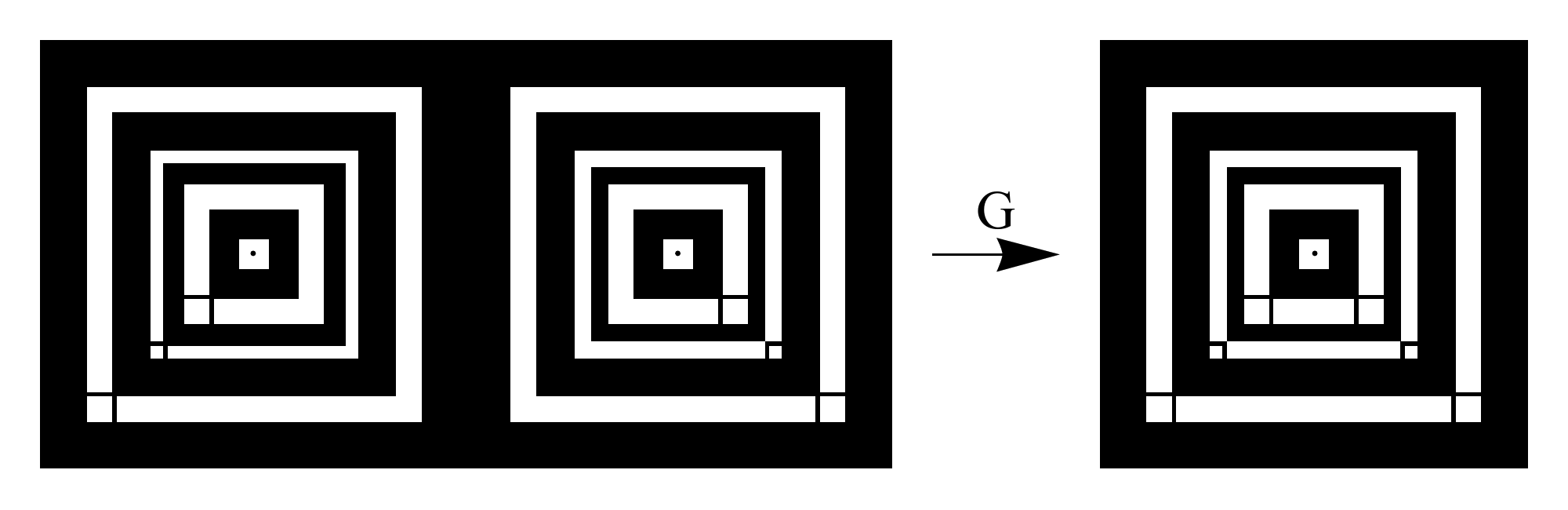}
\caption{\label{symfig2} The homotopy $G$ from $f\cdot g$ to $h$ excises the middle portion of the concatenation $f\cdot g$ (left rectangle). The squares $R_j$ appear along the diagonal of the domain of $f$ and the squares $S_j$ appear along the lower off-diagonal of the domain of $g$. The homotopy $H$, performed after $G$, stretches the left edge of $S_j$ to meet the right edge of $R_j$ to obtain $L$.}
\end{figure}

Now we perform an infinite $n$-cube shuffle homotopy to $h$. Let $S_{j}'=\left[\frac{d_j}{2},\frac{1-c_j}{2}\right]\times \left[\frac{c_j}{2},\frac{d_j}{2}\right]^{n-1}$. Note that $S_j\subseteq S_{j}'\subseteq K_j$ and that $R_j$ and $S_{j}'$ share a face. Here, we may simply take $T_j(t)=\left[(t)\frac{d_j}{2}+(1-t)\frac{1-d_j}{2},\frac{1-c_j}{2}\right]\times \left[\frac{c_j}{2},\frac{d_j}{2}\right]^{n-1}$, $0\leq t\leq 1$ as a parameterized collection of $n$-cubes from $S_j$ to $S_j '$ in $K_j$ which do not meet $\int(R_j)$. Define $H':I^n\times I\to Y$ to be the map which
\begin{itemize}
\item is the constant homotopy (of $h$) on $I^n\backslash \bigcup_{j\in\bbn}\int(K_j)$ and each $R_j$, $j\in\bbn$,
\item maps $T_j(t)\times\{t\}$ to $A_j$ by $g_j$ for all $0\leq t\leq 1$,
\item maps $K_j\times I\backslash \left(\bigcup_{t\in\ui}(T_j(t)\times\{t\})\cup (R_j\times I)\right)$ to $x_j$
\end{itemize}
As before, the projections $\rho_k\circ H'$ are continuous and so $H'$ is continuous. Let $L\in\Omega^n(Y,y_0)$ be the $n$-loop $L(\bfx)=H'(\bfx,1)$. Notice that $L$ agrees with $h$ on $I^n\backslash \bigcup_{j\in \bbn}(R_j\cup S_j)$ and on each $n$-cube $R_j\cup S_j'$, we have $L|_{R_j\cup S_j'}\equiv f_j\cdot g_j$. Hence, $L$ is constructed according to our definition of $\kappa$, that is, $[L]=\kappa(([f_j\cdot g_j])_{j\in \bbn})$. All together, we have:
\begin{eqnarray*}
\kappa(([f_j])_{j\in \bbn})+\kappa(([g_j])_{j\in \bbn}) &=& [f\cdot g]\\
&=& [L]\\
&=& \kappa(([f_j\cdot g_j])_{j\in \bbn})\\
&=&\kappa(([f_j])_{j\in \bbn}+([g_j])_{j\in \bbn})
\end{eqnarray*}
We conclude that $\kappa$ is a group homomorphism, which is a section to $\psi $.
\end{proof}

\begin{corollary}
Suppose $Y=\shadj(X,x_j,A_j,a_j)$ is path-connected and $x_0\in Y$. If $\pi_n(A_j,a_j)\neq 0$ for infinitely many $j\in\bbn$, then $|\pi_n(Y,y_0)|=2^{\aleph_0}$.
\end{corollary}

\subsection{Indeterminate adjunction spaces}\label{subsectionintermediateadjunctionspaces}

Some relevant spaces, such as generalized covering spaces of shrinking adjunction spaces, will not quite be shrinking adjunction spaces. We define the following general term to provide a context for this situation.

\begin{definition}[indeterminate adjunction space]\label{indeterminatedef}
We say that $Y$ is an \textit{indeterminate adjunction space} with core space $X$, attachment spaces $\{A_j\}_{j\in J}$, and attachment points $\{x_j\}_{j\in J}$ if $X$ is a closed subspace of $Y$ and if $Y\backslash X$ is the disjoint union of open sets $N_j$, $j\in J$ where $A_j=\overline{N_j}$ and $A_j\cap X=\{x_j\}$ for all $j\in \bbn$.
\end{definition}

If $Y$ is an indeterminate adjunction space, then the underlying set of $Y$ is the same as the ordinary adjunction space (with the weak topology). However, the homeomorphism type of $Y$ is not uniquely determined. Despite this ambiguity, it is clear that, given any finite set $F\subseteq J$, there is a retraction $\rho_{F}:Y\to X\cup\bigcup_{j\in F}A_j$ that collapses $A_j$, $j\in J\backslash F$ to $x_j$. Therefore, if $\tau$ is the topology of $Y$, then $\tau$ may be as fine as the weak topology  $\tau_{weak}$ (this is the upper limit) or as coarse as the inverse limit topology $\tau_{lim}$ (this is the lower limit). Indeed, we have $ \tau_{\lim}\subseteq\tau\subseteq\tau_{weak}$, and it is possible that $\tau$ is equal to neither (see Example \ref{coneintermediateexample}). We show that even if we are confronted with an indeterminate adjunction space $Y$, a subspace of $Y$ that is a Peano continuum must be a true shrinking adjunction space.

\begin{proposition} \label{indeterminateproppeano}
Let $Y$ be the indeterminate adjunction space with core space $X$, attachment spaces $\{A_j\}_{j\in \bbn}$, and attachment points $\{x_j\}_{j\in \bbn}$. If $Z\subseteq Y$ is a subspace that is a Peano continuum and meets $X$ and each open set $A_j\backslash\{a_j\}$, then $Z$ is homeomorphic to the shrinking adjunction space obtained by attaching the spaces $\{Z\cap A_j\}_{j\in\bbn}$ to $Z\cap X$ along the points $\{x_j\}_{j\in \bbn}$.
\end{proposition}

\begin{proof}
Let $Y_{lim}$ denote the underlying set of $Y$ equipped with the inverse limit (shrinking adjunction) topology and note that the identity $i: Y\to Y_{lim}$ is continuous. Since $Z$ is path-connected and $Z\cap (A_j\backslash\{a_j\})$ for all $j\in\bbn$, we have $x_j=a_j\in Z\cap X$ for all $j\in\bbn$. The collapsing maps $r:Y\to X$ and $r_j:Y\to A_j$ are retractions with respect to the topology of the indeterminate adjunction space so $r(Z)=Z\cap X$ and $r_j(Z)=Z\cap A_j$ are all Peano continua. Let $Z_{lim}$ denote the shrinking adjunction space obtained by attaching the spaces $\{Z\cap A_j\}_{j\in\bbn}$ to $Z\cap X$ along the points $\{x_j\}_{j\in \bbn}$. Even though the topology of $Y$ is indeterminate, the function $\rho_{k}\circ i:Y\to Y_k$, which collapses $A_j$, $j>k$ to $x_j$ is continuous. Therefore, the restriction $(\rho_{k}\circ i)|_{Z}:Z\to Y_{k}\cap Z$ is continuous for all $k\in\bbn$. We conclude that the identity function $i|_{Z}:Z\to Z_{lim}=\varprojlim_{k}(Y_k\cap Z)$ is continuous. Since $Z$ is compact and $Z_{lim}$ is Hausdorff, $i|_{Z}$ is a homeomorphism.
\end{proof}

\section{Whitney Covers and Infinite Homotopy Factorization}\label{sectionwhitneycover}

\subsection{General homotopy factorization methods}\label{subsectiongeneralfactorization}
We employ the classical Whitney Covering Lemma \cite{Whitney}, which guarantees the existence of highly structured locally finite covers of arbitrary open subsets of n-dimensional space. Our statement of the Whitney Covering Lemma below is not the strongest possible version of this well-known result; the weaker statement below is more streamlined for our application. In particular, we do not require all of the geometric features of the construction and we restrict our focus to open sets $U$ in $\bbr^n$ which lie in the open unit cube so that $\partial U\subseteq I^n$. We refer to \cite[Appendix J]{Grafakos} for the classical statement, which implies the following.

\begin{lemma}[Whitney Covering]\label{whitneycoveringlemma}
Let $U$ be a non-empty, proper open subset of $(0,1)^n$. Then there exists a collection $\scrc$ of $n$-cubes such that
\begin{enumerate}
\item $\bigcup \scrc=U$
\item $\int(C)\cap \int(C')$ when $C\neq C'$ in $\scrc$,
\item for given $C\in\scrc$, we have $C\cap C'=\emptyset$ for all but finitely many cubes $C'\in \scrc$.
\item Every neighborhood of the boundary $\partial U$ in $I^n$ contains all but finitely many $C\in\scrc$.
\end{enumerate}
\end{lemma}

We refer to $\scrc$ from Lemma \ref{whitneycoveringlemma} as a \textit{Whitney cover of} $U$ and the individual cubes $C\in\scrc$ as \textit{Whitney cubes} (See Figure \ref{fig3}). Notice that every Whitney cover of $U$ is an $n$-domain.

\begin{figure}[h]
\centering \includegraphics[height=1.7in]{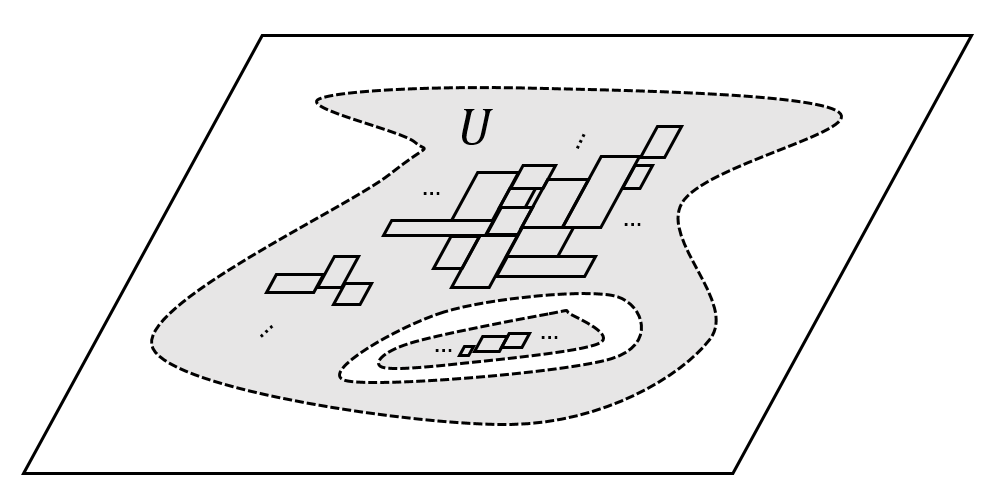}
\caption{\label{fig3}A Whitney covering of an open set $U\subseteq (0,1)^2$, which need not be connected nor simply connected. The diameters of the $2$-cubes shrink to $0$ as they approach any boundary point of $U$.}
\end{figure}

\begin{remark}\label{subdivisionremark}
It is important to realize that a Whitney cover $\scrc$ of $U$ will not be a ``cubical decomposition" of $U$, that is, the intersection of two $n$-cubes of $\scrc$ may not be a face of either cube. While a carefully constructed triangulation of $U$ could overcome this technical issue, we avoid going back and forth between the cubical and simplicial settings by using the following construction, which takes advantage of the fact that the cubes in $\scrc$ are oriented with the standard unit vectors and have pairwise disjoint interiors: Let $\scrc_n=\scrc$. If $\scrc_{k}$ is defined, let $\scrc_{k-1}$ be the set of all $k-1$-cubes which are the intersection of two $k$-cubes from $\scrc_{k}$. Notice that $\bigcup \scrc_{k-1}=\bigcup \{bd(K)\mid K\in\scrc_k\}$ where $bd(K)$ is the face-boundary of $K$, i.e. the union of all $(k-1)$-faces of the $k$-cube $K$. However, $\bigcup \scrc_{k-1}$ will usually be strictly larger than the union of the $(k-1)$-faces of the original $n$-cubes $C\in\scrc$. For instance, $\scrc_0$ contains the vertices ($0$-faces) of the $n$-cubes $C\in\scrc$ but also the $0$-faces of any $k$-cube formed as the intersection of higher dimensional cubes from $\bigcup_{0<m\leq n}\scrc_{m}$.
\end{remark}


\begin{lemma}[$n$-loop factorization]\label{factorizationlemma}
Let $n\geq 1$ and $X$ be sequentially $(n-1)$-connected at $x_0\in X$. Suppose $f:I^n\to X$ is a map and $U$ is an open subset of $(0,1)^n$ such that $f(\partial U)=x_0$. Then for every Whitney cover $\scrc$ of $U$, there exists a map $g:I^n\to X$ such that
\begin{itemize}
\item $f$ is homotopic to $g$ by a homotopy constant on $I^n\backslash U$,
\item for each $C\in\scrc$, $g(\partial C)=x_0$,
\item for each $C\in\scrc$, there is an $n$-cube $K\subseteq \int(C)$ such that $g|_{K}\equiv f|_{C}$.
\end{itemize}
\end{lemma}

\begin{proof}
Let $A=I^n\backslash U$ and consider a Whitney cover $\scrc$ of the (not necessarily connected) open set $U$ by $n$-cubes. Recall the definition of $\scrc_k$ for $0\leq k\leq n$ from Remark \ref{subdivisionremark}. We define a homotopy $H:I^n\times I\to X$ by recursively by extending the definition of $f$ on faces of the cubes $K\times I$, $K\in\scrc_k$. By Condition 4. in Lemma \ref{whitneycoveringlemma} and the continuity of $f$, the set $\{f|_{C}\mid C\in\scrc\}$ clusters at $x_0$ (when each $C\in\scrc$ is identified with $I^n$). Hence, $\{f|_{K}\mid K\in\scrc_{k}\}$ clusters at $x_0$ for all $0\leq k\leq n$. First, set $Y_0=I^n\times \{0\}\cup (A\times I)\cup   \left(\bigcup_{C\in\scrc}\partial C\right)\times \{1\}$.
Define $H_0:Y_0\to X$ so that
\begin{itemize}
\item $H_0(\bfx,0)=f(\bfx)$ if $\bfx\in I^n\times\{0\}\cup A\times I$,
\item $H_0(\bfx,1)=x_0$ if $\bfx\in \bigcup_{C\in\scrc}\partial C$.
\end{itemize}
Since $H_0(\partial U\times I)=x_0$, it is clear that $H_0$ is continuous on $Y_0$. Recursively, define $$Y_k=Y_{k-1}\cup \bigcup_{K\in\scrc_{k-1}}K\times I$$ for $1\leq k\leq n$ and notice that $Y_n=Y_0\cup \bigcup_{C\in\scrc}(\partial C)\times I$. We recursively construct maps $H_k:Y_k\to X$ as follows. Recall that any neighborhood of $x_0$ in $X$ contains all but finitely many of the points $f(v)$, $v\in\scrc_0$. Since $X$ is sequentially $0$-connected at $x_0\in X$, there exists a countable system of paths $\alpha_v:I\to X$, $v\in\scrc_0$ from $f(v)$ to $x_0$ that clusters at $x_0$. Define $H_1(v,t)=\alpha_v(t)$ for $v\in\scrc_0$. This extends the definition of $H_0$ to $H_1$ by extending the definition on $\bigcup_{v\in\scrc_0}\{v\}\times I$ in a well-defined fashion. Additionally, the continuity of $H_1$ follows directly from the continuity of $H_0$ and the fact that $\{\alpha_v\}_{v\in\scrc_0}$ clusters at $x_0$. For any edge $K\in\scrc_1$, the map $H_1$ restricts to a $1$-loop $g_K:K\times \{0,1\}\cup bd (K)\times I\to X$. Since the collections $\{f|_{K}\mid K\in \scrc_1\}$ and $\{\alpha_{v}\mid v\in\scrc_0\}$ both cluster at $x_0$, it follows that the collection of $1$-loops $\{g_K\mid K\in\scrc_1\}$ clusters at $x_0$.

Suppose, for $1\leq k\leq n-1$, we have constructed a continuous extension $H_{k}:Y_{k}\to X$ of $H_{k-1}$ with the property that if $g_K$ is the restriction of $H_{k-1}$ to $S_K=K\times \{0,1\}\cup bd( K)\times I$ for $K\in \scrc_k$, then the collection $\{g_K\}_{K\in\scrc_{k}}$ clusters at $x_0$ when we identify each $S_K$ with $\partial I^{k+1}$ in the canonical way. Since $X$ is assumed to be sequentially $k$-connected at $x_0$, there exists a collections of maps $G_K:K\times I\to X$ such that $G_K|_{S_K}=g_K$ and such that $\{G_K\mid K\in\scrc_k\}$ clusters at $x_0$ when each $K\times I$ is identified with $I^{k+1}$ in the canonical way. For each $K\in\scrc_k$, define the extension  $H_{k+1}$ of $H_k$ on $K\times I$ so that $H_{k+1}|_{K\times I}=G_k$ for $K\in\scrc_{k}$. This is clearly well-defined and is continuous based on the continuity of $H_k$ and the fact that $\{G_K\mid K\in\scrc_k\}$ clusters at $x_0$.

We proceed until arriving at a continuous extension $H_n:Y_n\to X$. In particular, we have constructed $H_n$ so that for every neighborhood $V$ of $x_0$, we have $H_n(C\times \{0\}\cup (\partial C)\times I)\subseteq V$ for all but finitely many $n$-cubes $C\in\scrc$. For each $C\in\scrc$, fix a piecewise linear retraction $r_C:(C\times I)\to C\times \{0\}\cup (\partial C)\times I$ as in Remark \ref{retractionremark}. We extend $H_n$ to a homotopy $H:I^n\times I\to X$ by defining $H(\bfx,t)=H_n(r_C(\bfx))$ if $\bfx\in C$. See Figure \ref{fig4} for an illustration of the construction of $H$. Since $H(C\times I)=H_n(C\times \{0\}\cup bd(C)\times I)$ for all $C\in\scrc$, we have that for a given neighborhood $V$ of $x_0$ in $X$, $H$ maps all but finitely many of the $(n+1)$-cubes $C\times I$ into $V$. The continuity of $H$ is an immediate consequence of this observation. Let $g:(I^n,\partial I^n)\to (X,x_0)$ be the map $g(\bfx)=H(\bfx,1)$ and $K\subseteq C$ be the $n$-cube that $r_C$ maps $K\times \{1\}$ homeomorphically onto $C\times\{0\}$. Then $g(\partial C)=x_0$ and $g|_{K}\equiv f|_{C}$ for all $C\in\scrc$.
\end{proof}

\begin{remark}\label{imageremark}
In some applications we will apply the above construction of $H$ in slightly more structured scenarios. For example, suppose there is a sequentially $(n-1)$-connected based subspace $(A,x_0)$ of $(X,x_0)$, and that after some choice of Whitney cover $\scrc$, we know that $f(\partial C)$ lies in $A$ for all $C\in\scrc$. Then we may construct $H$ through the same process but choosing all of the maps $g_K$ and $G_K$ to have image in $A$. The resulting homotopy $H$ then has the additional property that $H(C\times I)\subseteq A\cup f(C)$ for all $C\in\scrc$ and thus $\im(H)\subseteq A\cup \im(f)$.
\end{remark}

\begin{figure}[p]
\centering
\begin{subfigure}{.5\textwidth}
  \centering
  \includegraphics[width=.9\linewidth]{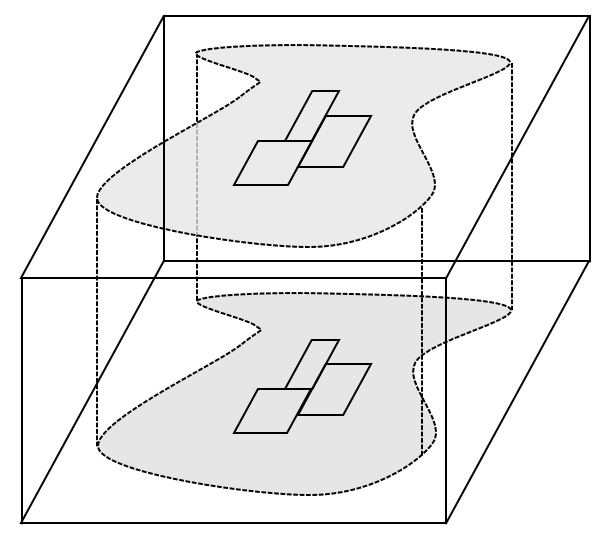}
  \caption{Step 0}
\end{subfigure}%
\begin{subfigure}{.5\textwidth}
  \centering
  \includegraphics[width=.9\linewidth]{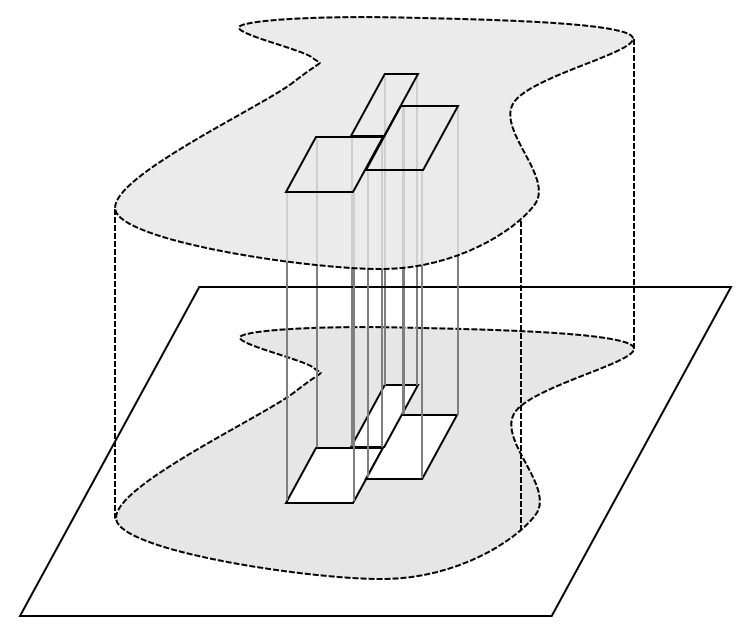}
  \caption{Step 1}
\end{subfigure}

\begin{subfigure}{.5\textwidth}
  \centering
  \includegraphics[width=.9\linewidth]{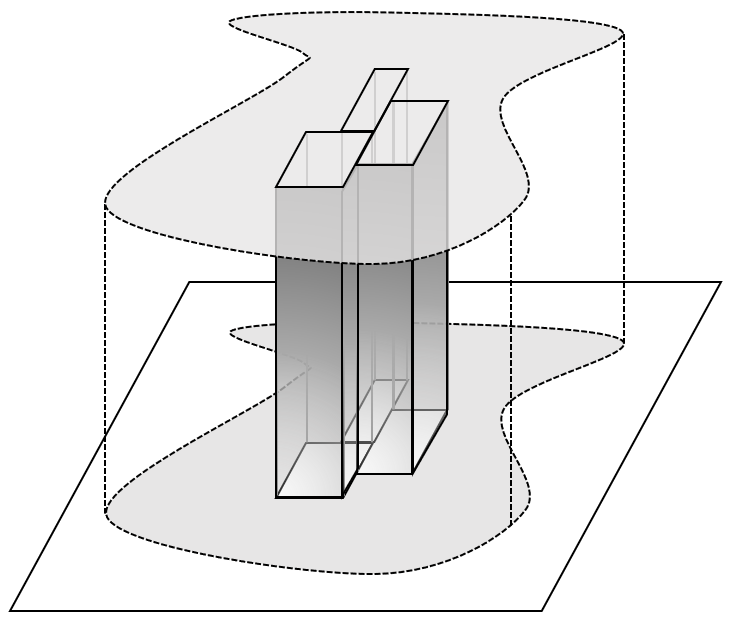}
  \caption{Step 2}
\end{subfigure}%
\begin{subfigure}{.5\textwidth}
  \centering
  \includegraphics[width=.9\linewidth]{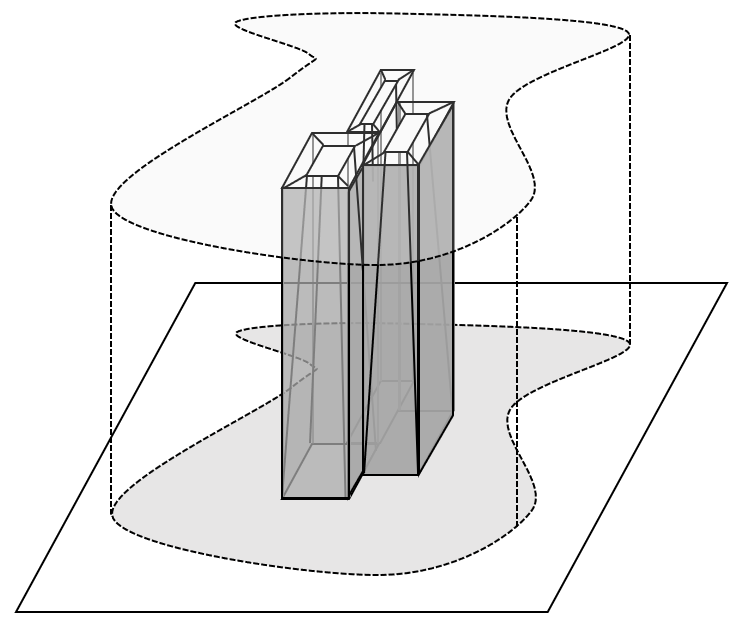}
  \caption{Step 3}
\end{subfigure}
\caption{The four steps in the construction of the homotopy $H$ from the proof of Lemma \ref{factorizationlemma} for $n=2$. In all four sub-figures, the bottom square is the domain of $f$ and the dashed domain is $U\times I$. In Step $0$, the cubes at height $1$ are mapped to $x_0$. In Step $1$, the previous step is extended to the $1$-faces in $\scrc_1$ using the sequentially $0$-connected property. In Step $2$, we use the sequentially $1$-connected property to define $H$ on the rectangles created by the vertical $1$-faces. After Step 2, $H$ is defined on the ``empty box without a top" over each Whitney cube. In Step $3$, $H$ is defined on each $3$-cube $C\times I$ $,C\in\scrc$ by retracting it onto the empty box from Step 2 and applying the value of $H$ from Step 2. }
\label{fig4}
\end{figure}

Often, we will apply Lemma \ref{factorizationlemma} to $f\in \Omega(X,x_0)$ and $U=I^n\backslash f^{-1}(x_0)$, in which case, the conclusion ensures that $f$ is homotopic rel. $\partial I^n$ to a $\scrc$-concatenation $\prod_{C\in\scrc}g_C$ for any Whitney covering $\scrc$ of $U$. The following corollary, which will be applied to shrinking adjunction spaces, is the result of applying Lemma \ref{factorizationlemma} finitely many times and forming the vertical concatenation of the resulting homotopies. 

\begin{corollary}\label{finitefactorizationcorollary}
Let $n\geq 1$ and $X$ be sequentially $(n-1)$-connected at points $x_0,x_1,x_2,\dots, x_k\in X$. Suppose $f:I^n\to X$ is a map and $U_0,U_1,U_2,\dots ,U_k$ are pairwise-disjoint open sets in $(0,1)^n$ such that $f(\partial U_j)=x_j$ for each $1\leq j\leq k$. Then for any Whitney covers $\scrc_j$ of $U_j$, $1\leq j\leq k$, there exists a map $g:I^n\to X$ such that
\begin{itemize}
\item $f$ is homotopic to $g$ by a homotopy that is constant on $\ui^n\backslash \bigcup_{j=1}^{k}U_j$,
\item $g(\partial C)=x_j$ whenever $C\in\scrc_j$.
\item for each $C\in\scrc_j$, there is an $n$-cube $K\subseteq \int(C)$ such that $g|_{K}\equiv f|_{C}$.
\end{itemize}
\end{corollary}

It is natural to ask if there is a version of Lemma \ref{factorizationlemma} where the sequentially $n$-connected property is replaced with the weaker ``$n$-tame" property. We show that the answer is positive but only if we are willing to replace interiors of the $n$-cubes $C\in \scrc$ with open manifolds constructed as finite unions of such cubes. Indeed, the sequentially $(n-1)$-connected property allows us to recursively construct the maps $H_k$ on \textit{all} faces of cubes in $\scrc$. If we only know that $X$ is $(n-1)$-tame, we may perform the same construction only for a cofinite subset of faces in each dimension $ k\leq n$. We give a slightly shorter proof since most of the construction is the same as in the proof of Lemma \ref{factorizationlemma}. 

\begin{lemma}\label{manifoldsinsidelemma}
Let $n\geq 1$ and $X$ be a path-connected space that is $(n-1)$-tame at $x_0\in X$. Suppose $f\in X^{I^n}$ and $U$ is an open set in $(0,1)^n$ such that $f(\partial U)=x_0$. Then there is a map $g\in \Omega^n(X,x_0)$ and pairwise-disjoint, connected, open $n$-manifolds $M_1,M_2,M_3,\dots,\subseteq U$ such that
\begin{itemize}
\item $f$ is homotopic rel. to $g$ by a homotopy constant on $I^n\backslash U$,
\item there exists $k\in\bbn$ such that $\overline{M_i}$ is a finite union of $n$-cubes for all $1\leq i\leq k$ and $\overline{M_i}$ is a $n$-cube for $i>k$,
\item $U=\bigcup_{i\in\bbn}\overline{M_i}$,
\item $g(\partial M_i)=x_0$ for all $i\in\bbn$.
\end{itemize}
\end{lemma}

\begin{proof}
Let $\scrc$ be a Whitney cover of $U$ and set $A=I^n\backslash U$. Define $Y_0=I^n\times \{0\}\cup A\times I\cup \bigcup_{C\in\scrc}\partial C\times \{1\}$. Define $Y_0$ and $H_0$ as in the Proof of Lemma \ref{factorizationlemma}. Note that $\{f(v)\mid v\in\scrc_0\}$ clusters at $x_0$. Since $X$ is path connected and $0$-tame at $x_0$, $X$ is sequentially $0$-connected at $x_0$. Setting $\scrd_0= \scrc_0$, there exists a set of paths $\alpha_v:I\to X$, $v\in\scrd_0$ from $f(v)$ to $x_0$. Let $Y_1=Y_0\cup \bigcup_{v\in\scrd_0} \{v\}\times I$ and define $H_1:Y_1\to X$ as an extension of $H_0$ by defining $H(v,t)=\alpha_v(t)$. Let $\scrb_1$ be the set of $K\in\scrc_1$ whose endpoints lie in $\scrd_0$ and note that $\scrb_1$ is cofinite in $\scrc_1$. For each $K\in \scrb_1$, let $g_K:K\times \{0,1\}\cup bd(K)\times I\to X$ be the restriction of $H_1$ and notice that $\{g_K\mid K\in\scrb_1\}$ clusters at $x_0$.

Suppose, for $1\leq k\leq n-1$, we have constructed cofinite sets $\scrb_k\subseteq \scrc_k$ and a continuous extension $H_{k}:Y_{k}\to X$ of $H_{k-1}$ with the property that if $g_K$ is the restriction of $H_{k-1}$ to $S_K=K\times \{0,1\}\cup bd(K)\times I$ for $K\in \scrb_k$, then the collection $\{g_K\}_{K\in\scrb_{k}}$ clusters at $x_0$ when we identify each $S_k$ with $\partial I^{k+1}$ in the canonical way. Since $X$ is assumed to be $k$-tame at $x_0$, there exists a a cofinite subset $\scrd_{k}\subseteq \scrb_{k}$ and collections of maps $G_K:K\times I\to X$, $K\in\scrd_k$ such that $G_K|_{S_K}=g_K$ and such that $\{G_K\mid K\in\scrd_k\}$ clusters at $x_0$ when each $K\times I$ is identified with $I^{k+1}$ in the canonical way. Let $$Y_{k+1}=Y_{k}\cup \bigcup_{K\in\scrd_{k}}K\times I.$$Define the extension $H_{k+1}$ of $H_k$ on $K\times I$, $K\in\scrd_k$ so that $H_{k+1}|_{K\times I}=G_k$ for $K\in\scrc_{k}$. This is clearly well-defined and is continuous based on the continuity of $H_k$ and the fact that $\{G_K\mid K\in\scrc_k\}$ clusters at $x_0$. Let $\scrb_{k+1}$ be the cofinite subset of $\scrc_{k+1}$ consisting of all the $(k+1)$-cubes $K$ such that $bd(K)$ is the union of elements from $\scrd_{k}$.

We proceed until we arrive at an extension $H_n:Y_n\to X$ and the cofinite subset $\scrb_n$ of $\scrc$ consisting of all $n$-cubes $C\in\scrc$ such that $\partial C$ is the union of $(n-1)$-cubes in $\scrd_{n-1}$. Let $Z=f^{-1}(x_0)\cup \bigcup_{B\in\scrb_n}B$ and define $G:I^n\times \{0\}\cup Z\times I$ so that $G(\bfx,t)=x_0$ if $(\bfx,t)\in f^{-1}(x_0)\times I$ and if $B\in \scrb_n$, then $G|_{B\times I}$ is the composition of the retraction $B\times I\to B\times \{0\}\cup bd (B)\times I$ and the restriction of $H_n$ to $B\times I\to B\times \{0\}\cup bd (B)\times I$. Continuity of $G$ requires the same observations used in the proof of Lemma \ref{factorizationlemma}. Notice that $W=I^n\backslash Z$ is open and has closure equal to the union of the finitely many $n$-cubes in $\scrc\backslash \scrb_n$. Therefore, $W$ is the union of finitely many open $n$-manifolds $M_1,M_2,M_3,\dots, M_k$ and the closure of each is a finite union of $n$-cubes. In particular $(M_i,\partial M_i)$ has the homotopy extension property for all $1\leq i\leq k$. Hence, we may extend $G$ to a map $H:I^n\times I\to X$ by choosing a retraction $r_i:\overline{M_i}\times I\to \overline{M_i}\times \{0\}\times \partial M_i\times I$ and define $H(\bfx,t)=H(r_i(\bfx,t))$ for $(\bfx,t)\in\overline{M_i}\times I$. Let $g\in I^n\to X$ be the map $g(\bfx)=H(\bfx,1)$. If we write $\scrb_n=\{B_{m+1},B_{m+2},B_{m+3},\dots\}$ and set $M_i=\int(B_i)$ for $i>m$, we complete the definition of $M_i$, $i\in\bbn$ in a manner that satisfies the conclusion of the Lemma.
\end{proof}

If we add the $n$-tame condition to Lemma \ref{manifoldsinsidelemma}, then the cofinal sequence of maps $g|_{\overline{M_i}}$ (from the conclusion of the Lemma), which are $n$-loops based at $x_0$ must have a cofinal subsequence that is sequentially null-homotopic.

\begin{corollary}\label{manifoldsinsidecorollary}
Let $n\geq 1$ and $X$ be a path-connected space that is $n$-tame at $x_0\in X$. Suppose $f:I^n\to X$ is a map and $U$ is an open set in $(0,1)^n$ such that $f(\partial U)=x_0$. Then there is a map $g:I^n\to X$ and finitely many pairwise-disjoint, connected, open $n$-manifolds $M_1,M_2,M_3,\dots,M_k\subseteq U$ such that
\begin{itemize}
\item $f$ is homotopic to $g$ by a homotopy constant on $I^n\backslash U$,
\item $\overline{M_i}$ is a finite union of $n$-cubes for all $1\leq i\leq k$,
\item $g(\overline{U}\backslash \bigcup_{i=1}^{k}M_i)=x_0$.
\end{itemize}
\end{corollary}

\subsection{Homotopy factorization in shrinking adjunction spaces}\label{subsectionshadjfactorization}

Here, we apply the methods from the previous section to $n$-loops in shrinking adjunction spaces. Notice that this first definition is intentionally given for the far more general \textit{indeterminate} adjunction spaces.

\begin{definition}
Let $Y$ be an indeterminate adjunction space with core $X$, attachment spaces $\{A_j\}_{j\in J}$, and attachment points $\{x_j\}_{j\in J}$. Let $y_0\in Y$. We say that an $n$-loop $f\in \Omega^n(Y,y_0)$ is in \textit{factored form with respect to this decomposition of }$Y$ if for every $j\in J$, there is a (possibly empty) $n$-domain $\scrc_j$ in $U_j=f^{-1}(A_j\backslash\{a_j\})$ such that $f\left(\ov{U_j}\backslash \bigcup_{C\in\scrc_j}\int(C)\right)=a_j$.

If a choice of decomposition for $Y$ (e.g. $Y=\shadj(X,x_j,A_j,a_j)$ if $Y$ is a shrinking adjunction space) is clear from context, we may simply say that $f$ is in \textit{factored form}. We refer to the collection $\{\scrc_j\mid j\in J\}$ of $n$-domains as a \textit{factorization} of $f$ and to the $n$-loops $f|_C\in \Omega^n(A_j,a_j)$, $C\in\scrc_j$ as the \textit{$A_j$-factors} of $f$.
\end{definition}

An $n$-loop $f\in \Omega^n(Y,y_0)$ is in factored form in a shrinking adjunction space $Y=\shadj(X,x_j,A_j,a_j)$ with factorization $\{\scrc_j\mid j\in\bbn\}$ if and only if for each $j\in\bbn$ either $r_j\circ g$ is constant (in which case $\scrc_j=\emptyset$) or $r_j\circ g=\prod_{C\in\scrc_j}g|_{C}$ is a $\scrc_j$-concatenation.

\begin{lemma}[Existence of factored forms]\label{factoredformsequencelemma1}
Let $n\geq 1$, $Y=\shadj(X,x_j,A_j,a_j)$ where $(A_j,a_j)$ is sequentially $(n-1)$-connected for all $j\in \bbn$, and $y_0\in X$. Every $n$-loop $f\in\Omega^n(Y,y_0)$ is homotopic to some $g\in \Omega^n(Y,y_0)$ that is in factored form. Moreover, the homotopy from $f$ to $g$ may be chosen to be the constant homotopy on $f^{-1}(X)$ and to have image in $\im(f)\cup \bigcup\{A_j\mid \im(f)\cap A_{j}\backslash\{a_j\}\neq\emptyset\}$.
\end{lemma}

\begin{proof}
Let $U_j=f^{-1}(A_j\backslash \{a_j\})=(r_j\circ f)^{-1}(A_j\backslash \{a_j\})$ and choose a Whitney cover $\scrc_{j}$ of $U_j$. Apply Lemma \ref{factorizationlemma} to $r_j\circ f$ to find a sequence $\{g_{j,C}\in \Omega^n(A_j,a_j)\mid C\in \scrc_j\}$ that clusters at $a_j$ and such that $r_j\circ f\simeq \prod_{C\in\scrc_j}g_{C,j}$ by a homotopy $H_j:I\times I\to A_j$ that is the constant homotopy on $(r_j\circ f)^{-1}(a_j)=I^n\backslash U_j$. Define $g\in \Omega^n(Y,y_0)$ to be the function, which agrees with $f$ on $f^{-1}(X)$ and such that $g|_{C}=g_{C,j}$ for all $j\in\bbn$, $C\in\scrc_j$. Corollary \ref{finitefactorizationcorollary} ensures that this construction may be achieved continuously for any finite number collection of $j\in \bbn$. Recalling the inverse limit definition $Y=\varprojlim_{k}Y_k$, it follows that the projection $\rho_{k}\circ g:I^n\to Y_k$ is continuous for all $k\in\bbn$. Thus $g$ is continuous with respect to the inverse limit topology on $Y$. Similarly, define $H:I^n\times I\to Y$ so that $H(\bfx,t)=f(\bfx)$ when $f(\bfx)\in X$ and $H(\bfx,t)=H_j(\bfx,t)$ if $\bfx\in U_j$. As before, Corollary \ref{finitefactorizationcorollary} ensures the projections $\rho_k\circ H$ are continuous for all $k\in\bbn$. We conclude that $H$ is continuous.
\end{proof}

The next lemma shows that our construction of factored-form homotopy-representatives may be carried out in a way that preserves convergence to a point.

\begin{lemma}[Sequential factorization]\label{factoredformlemma2}
Let $n\geq 1$, $Y=\shadj(X,x_j,A_j,a_j)$ where $(A_j,a_j)$ is sequentially $(n-1)$-connected for all $j\in \bbn$, and $y_0\in X$. Let $\{f_m\}_{m\in\bbn}$ be a sequence in $\Omega^n(Y,y_0)$ that converges to $y_0$. Then $\{f_m\}_{m\in\bbn}$ is sequentially homotopic to a sequence $\{g_m\}_{m\in\bbn}$ in $\Omega^n(Y,y_0)$ such that, for every $m\in\bbn$, $g_m$ is in factored form. Moreover, the sequential homotopy $\{H_m\}_{m\in\bbn}$ from $\{f_m\}_{m\in\bbn}$ to $\{g_m\}_{m\in\bbn}$ may be chosen to have the following properties: for all $m\in\bbn$,
\begin{itemize}
\item $H_m$ is the constant homotopy on $f_{m}^{-1}(X)$,
\item $\im(H_m)\subseteq \im(f_m)\cup \bigcup\{A_j\mid \im(f_m)\cap A_{j}\backslash\{a_j\}\neq\emptyset\}$.
\end{itemize}
\end{lemma}

\begin{proof}
Since $\{f_m\}_{m\in\bbn}$ converges to $y_0$, we may form the infinite concatenation $f=\prod_{m\in\bbn}f_m\in \Omega^n(Y,y_0)$. Let $K_m=\left[\frac{m-1}{m},\frac{m}{m+1}\right]\times I^{n-1}$ and note that $K=\bigcup_{m\in \bbn}\partial K_m\cup \{1\}\times I^n\subseteq f^{-1}(y_0)$. By Lemma \ref{factoredformsequencelemma1}, $f$ is homotopic to an $n$-loop $g\in \Omega^n(Y,y_0)$ in factored form by a homotopy $H$, which is the constant homotopy on $f^{-1}(X)$ and to have image in $\im(f)\cup \bigcup\{A_j\mid \im(f)\cap A_{j}\backslash\{a_j\}\neq\emptyset\}$. Let $H_m=H\circ L_{I^{n+1},K\times I}$. Since $K\subseteq f^{-1}(y_0)\subseteq f^{-1}(X)$, we have $H_m(\partial I^n\times I)=y_0$ for all $m\in\bbn$. Define $g_m\in \Omega^n(Y,y_0)$ by $g_m(\bfx)=H_m(\bfx,1)$ and note that $g=\prod_{m\in\bbn}g_m$. Since $g$ is in factored form, each restriction $g_m$ must also be in factored form. Additionally, for each $m\in\bbn$, $H$ is the constant homotopy on $f^{-1}(X)\cap K_m=f|_{K_m}^{-1}(X)$. Therefore, $H_m$ is the constant homotopy on $f_{m}^{-1}(X)$. Moreover, the construction of $H$ from Lemma \ref{factoredformsequencelemma1} makes it clear that $\im(H_m)\subseteq  \im(f_m)\cup \bigcup\{A_j\mid \im(f_m)\cap A_{j}\backslash\{a_j\}\neq\emptyset\}$.
\end{proof}

\begin{theorem}\label{seqnconnectedtheorem1}
Let $n\geq 0$, $Y=\shadj(X,x_j,A_j,a_j)$ where $(A_j,a_j)$ is sequentially $n$-connected for all $j\in\bbn$, and $y_0\in X$. Every sequence $\{f_m\}_{m\in\bbn}$ in $\Omega^n(Y,y_0)$ that converges to $y_0$ is sequentially homotopic to $\{r\circ f_m\}_{m\in\bbn}$ where $r:Y\to X$ is the canonical retraction. Moreover, if $X$ is sequentially $n$-connected at $y_0$, then $Y$ is sequentially $n$-connected at $y_0$.
\end{theorem} 

\begin{proof}
The case $n=0$ is proved in Lemma \ref{adjunctionlemmazero}. Suppose $n\geq 1$ and that $\{f_m\}_{m\in\bbn}$ is a sequence in $\Omega^n(Y,y_0)$ that converges to $y_0$. Since $\{f_m\}_{m\in\bbn}$ converges to $y_0$ and the retraction $r:Y\to X$ is continuous, $\{r\circ f_m\}_{m\in\bbn}$ converges to $y_0$. By Lemma \ref{factoredformlemma2}, we may assume that $f_m$ is in factored form for all $m\in\bbn$. We will construct a sequential homotopy $\{f_m\}_{m\in\bbn}\simeq \{r\circ f_m\}_{m\in\bbn}$ by constructing a single homotopy $H$ from $f=\prod_{m\in\bbn}f_m$ to $g=\prod_{m\in\bbn}r\circ f_m$. Note that $f$ is in factored form since each $f_m$ is. Set $U_j=f^{-1}(A_j\backslash\{a_j\})=\bigcup_{m\in\bbn}f_{m}^{-1}(A_j\backslash\{a_j\})$. Choose a factorization of $f$, i.e. Whitney covers $\scrc_j$ of $U_j$ such that $f(\partial C)=a_j$ for all $C\in\scrc_j$. The continuity of $f$ ensures that $\{f|_{C}\mid C\in\scrc_j\}$ clusters at $a_j$. Therefore, since $A_j$ assumed to be sequentially $n$-connected at $a_j$, there are maps $H_{j,C}:C\times I\to A_j$, $C\in\scrc_{j}$ such that $H_{j,C}(\bfx,0)=f_m(\bfx)$ and $H_{j,C}(\partial C\times I\cup C\times \{1\})=a_j$ and so that the collection $\{H_{j,C}\mid C\in\scrc_{m,j}\}$ clusters at $a_j$.

Define $H:I^n\times I\to Y$ to be the constant homotopy on $f^{-1}(X)$ and so that $H|_{C}=H_{j,C}$ whenever $C\in\scrc_{j}$. Let $s_k:Y\to X\vee A_k$ be the retraction collapsing $A_j$, $j\neq k$ to $a_j$. Notice that $s_k\circ H:I^n\times I\to X\vee A_k$ is continuous for every $k\in\bbn$ since every neighborhood of $a_k$ in $X\vee A_k$ contains $\im(H_{k,C})$ for all but finitely many $C\in\scrc_{m,k}$. It follows that $\rho_k\circ H:I^n\times I\to Y_k$ is continuous for all $k\in\bbn$. Since $Y=\varprojlim_{k}Y_k$, we conclude that $H$ is continuous. By construction, $H$ is a homotopy from $f$ to $g$. Therefore, if $K_m=\left[\frac{m-1}{m},\frac{m}{m+1}\right]\times I^n$, we let $H_m=H|_{K_m\times I}\circ L_{I^n\times I,K_m\times I}$. Now $\{H_m\}_{m\in\bbn}$ is a sequential homotopy from $\{f_m\}_{m\in\bbn}$ to $ \{r\circ f_m\}_{m\in\bbn}$.

For the second statement, we add the assumption that $X$ is sequentially $n$-connected at $y_0$. Now if $\{f_m\}_{m\in\bbn}\subseteq \Omega^n(Y,y_0)$ converges to $y_0$, then by the first part, $\{f_m\}_{m\in\bbn}\simeq \{r\circ f_m\}_{m\in\bbn}$. Since $\{r\circ f_m\}_{m\in\bbn}$ is a sequence of $n$-loops in $X$ that converges to $y_0$, $\{r\circ f_m\}_{m\in\bbn}$ must be sequentially null-homotopic. Thus $\{f_m\}_{m\in\bbn}$ is sequentially null-homotopic.
\end{proof}

\begin{corollary}\label{retractioniso}
Let $n\geq 2$, $Y=\shadj(X,x_j,A_j,a_j)$ where $X$ is path connected and $(A_j,a_j)$ is sequentially $n$-connected for all $j\in\bbn$, and $y_0\in X$. Then the retraction $r:Y\to X$ induces an isomorphism $r_{\#}:\pi_m(Y,y_0)\to \pi_m(X,y_0)$ for all $0\leq m\leq n$.
\end{corollary}

\begin{definition}
Let $Y=\shadj(X,x_j,A_j,a_j)$, $y_0\in Y$, and $f,g\in\Omega^n(Y,y_0)$ be in factored form. We say that a factorization $\{\scrc_j\}_{j\in\bbn}$ of $f$ and a factorization $\{\scrd_j\}_{j\in\bbn}$ of $g$ are equivalent if for each $j\in\bbn$, there exists bijection $\kappa_j:\scrc_j\to\scrd_j$ such that for all $C\in\scrc_j$, $f|_{C}\equiv g|_{\kappa_j(C)}$.
\end{definition}

\begin{remark}\label{equivalentremark}
One should not jump to the tempting conclusion that distinct $n$-loops that admit equivalent factored forms in $Y=\shadj(X,x_j,A_j,a_j)$ are homotopic in $Y$. Indeed, the open sets $f^{-1}(A_j\backslash\{a_j\})$ and $g^{-1}(A_j\backslash\{a_j\})$ may have nothing in common and they may each have infinitely many components. Such a homotopy is only easy to construct (using Theorem \ref{shuffletheorem}) if $f^{-1}(A_j\backslash\{a_j\})=g^{-1}(A_j\backslash\{a_j\})$ and if this set is an open $n$-cube. Below, we will define $n$-loops enjoying such a simple factored-form to be in ``single-factor form." It requires a substantial effort to identify cases where we may homotope an ordinary factored form into single-factor form. To this end, we will use the next lemmas to carefully maintain control over the size and structure of the images of the $n$-loops involved. 
\end{remark}

\begin{remark}[Iterated shrinking adjunction spaces]\label{iteratedremark}
In the next lemma, we consider iterated shrinking adjunction spaces. In particular, let $Y=\shadj(X,x_j,A_j,a_j)$ and $Z=\shadj(Y,y_i,B_i,b_i)$ be a shrinking adjunction space with core space $Y$ (see Figure \ref{fig8}). By adding one-point spaces to $\{B_i\}_{i\in\bbn}$ (if necessary), we may assume that $T=\{i\in\bbn\mid y_i\in X\}$ is infinite and $T_j=\{i\in\bbn\mid y_i\in A_j\backslash\{a_j\}\}$ is infinite for all $j\in\bbn$, that is, there are infinitely many spaces $B_i$ attached to $X$ and each $A_j$. Moreover, we may then add one-point spaces to $\{A_j\}_{j\in\bbn}$ so that $\{x_j\mid j\in\bbn\}=\{y_i\mid i\in T\}$

Define $A_{j}^{\ast}=A_j\cup \bigcup_{i\in T_j}B_i$. We may identify $A_{j}^{\ast}$ as the shrinking adjunction space \[A_{j}^{\ast}=\shadj\left(A_j,\{y_i\}_{i\in T_j}, \{B_i\}_{i\in T_j},\{b_i\}_{i\in T_j}\right).\]Additionally, we may view $X^{\ast}=X\cup \bigcup_{i\in T}B_i$ as the shrinking adjunction space\[X^{\ast}=\shadj\left(X,\{y_i\}_{i\in T}, \{B_i\}_{i\in T},\{b_i\}_{i\in T}\right).\]
Finally, we point out the existence of the decomposition $Z=\shadj(X^{\ast},x_j,A_{j}^{\ast},a_j)$.
\end{remark}

\begin{figure}[H]
\centering \includegraphics[height=2.5in]{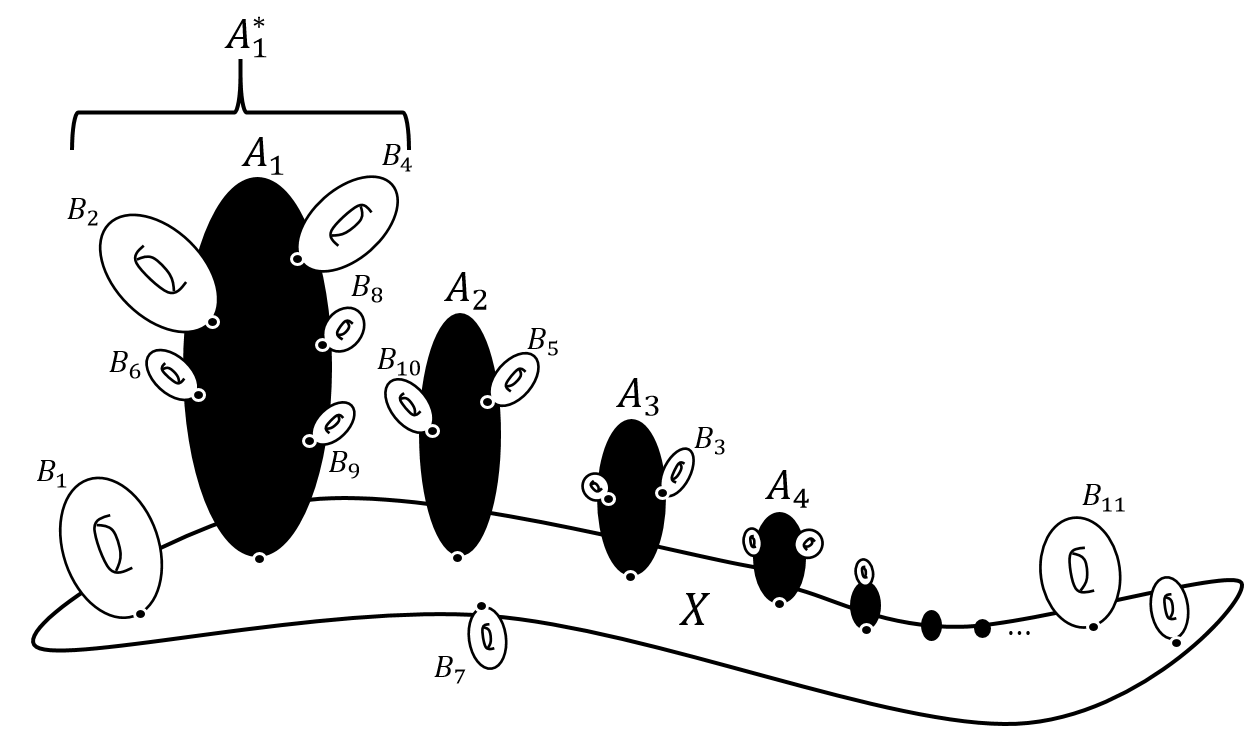}
\caption{\label{fig8} A shrinking adjunction space $Z$ with core space $Y=\shadj(X,x_j,A_j,a_j)$ and attachment spaces $\{B_i\}_{i\in\bbn}$.}
\end{figure}

\begin{lemma}\label{relativelemma}
Consider $Y=\shadj(X,x_j,A_j,a_j)$ and $Z=\shadj(Y,y_i,B_i,b_i)$ and the corresponding shrinking adjunction spaces $X^{\ast}$ and $A_{j}^{\ast}$ as described in Remark \ref{iteratedremark}. Let $z_0\in X$. Suppose $n\geq 1$ and that $(X,z_0)$ is sequentially $(n-1)$-connected and $(A_j,a_j)$ is sequentially $(n-1)$-connected for all $j\in\bbn$. If $f\in \Omega^n(Z,z_0)$ is in factored form with respect to $Z=\shadj(Y,y_i,B_i,b_i)$ with factorization $\{\scrc_i\}_{i\in\bbn}$, then there exists a map $g\in \Omega^n(Z,z_0)$ such that
\begin{itemize}
\item $f$ is homotopic to $g$ by a homotopy that is constant on $f^{-1}(X^{\ast})$ and which has image in $Y\cup \im(f)$,
\item $g$ admits a factorization $\{\scrr_j\}_{j\in\bbn}$ with respect to $Z=\shadj(X^{\ast},x_j,A_{j}^{\ast},a_j)$,
\item $g$ admits a factorization $\{\scrd_i\}_{i\in\bbn}$ with respect to $Z=\shadj(Y,y_i,B_i,b_i)$ that is equivalent to $\{\scrc_i\}_{i\in\bbn}$.
\end{itemize}
Moreover, for every $D\in \scrd_i$ with $i\in T_j$, there exists $R\in\scrc_{j}$ such that $D\subseteq R$.
\end{lemma}

\begin{proof}
Recall that $T=\{i\in\bbn\mid y_i\in X\}$ and each $T_j=\{i\in\bbn\mid y_i\in A_j\backslash\{a_j\}\}$ are infinite. Since $f$ is in factored form with respect to $Z=\shadj(Y,y_i,B_i,b_i)$, for each $i\in \bbn$, there a Whitney cover $\scrc_i$ of $V_i=f^{-1}(B_i\backslash\{b_i\})$ such that $f(\overline{V_i}\backslash\bigcup_{C\in\scrc_i}\int(C))=b_i$. In particular, $f(\partial C)=b_i$ for all $C\in\scrc_i$. Let $U_j=f^{-1}(A_{j}^{\ast}\backslash\{a_j\})$ and note that $V_i\subseteq U_j$ whenever $i\in T_j$. For each $j\in \bbn$, choose a Whitney cover $\scrr_j$ of $U_j$. 

For the moment, fix a pair $(j,i)$ with $i\in T_j$. For every $C\in\scrc_i$, the interior of $C$ meets the interior of some $R\in \scrr_j$ in an open $n$-cube and so we may choose $n$-cube $C'\subseteq \int(C)\cap \int(R)$. Let $\scrc_i'$ denote the sub-$n$-domain $\{C'\mid C\in\scrc_i\}$ of $\scrc_i$. Making these choices for all $i\in T_j$, ensures that given any $C\in \scrc_i$, we have $C'\subseteq \int(R)$ for some $R\in\scrr_j$. 

Define $f'\in \Omega^n(Z,z_0)$ so that $f'$ agrees with $f$ on $f^{-1}(Y\cup \bigcup_{i\in T}B_i)$ and so that for each $C\in\scrc_i$, $f'(C\backslash \int(C'))=a_j$ and $f'|_{C'}=f|_{C}\circ L_{C',C}$. Intuitively, $f'$ is the same as $f$ except that the domain of each of the $A_j$-factors has been made small enough to fit in some element of $\scrr=\bigcup_{j\in\bbn}\scrr_{j}$. In particular $f'$ is in factored form and the factorization $\scrc_i'$ is equivalent to $\scrc_i$. Since $C'\subseteq C$ for all $C\in\scrc_i$, there exists a cube-shrinking homotopy $H:I^n\times I\to Z$ from $f$ to $f'$ as described in Lemma \ref{shrinkingcubelemma}. Specifically, $H$ is the constant homotopy on $f^{-1}(Y\cup \bigcup_{i\in T}B_i)$, whenever $C\in\scrc_i$, $H|_{C\times I}:C\times I\to B_i$ continuously shrinks the domain of $f$ from $C$ to $C'$, and if $K$ is the convex hull of $C\times\{0\}\cup C'\times\{1\}$ in $C\times I$, then $H((C\times I)\backslash \int(K))=a_j$. Hence, $\im(H)=\im(f)$.

Observing the similarities between $f$ and $f'$, we have $U_j=(f')^{-1}(A_{j}^{\ast}\backslash \{a_j\})$. Therefore, $\scrr_j$ is a Whitney cover of $(f')^{-1}(A_{j}^{\ast}\backslash \{a_j\})$ with the property that for all $i\in T_j$ and $C'\in \scrc_i'$, we have $C'\subseteq \int(R)$ for some $R\in\scrr_j$. Additionally, since 
\[(f')^{-1}\left(\bigcup_{i\in T_j}B_i\backslash\{b_i\}\right)=\bigcup_{i\in T_j}\bigcup_{C'\in\scrc_i'}\int(C')\subseteq \bigcup_{R\in\scrr_j}\int(R),\]
it follows that $f'(\partial R)\subseteq A_j$ for all $R\in \scrr_j$. 

Since $A_j$ is sequentially $(n-1)$-connected at $a_j$ and $f'(\partial R)\subseteq A_j$ for all $R\in\scrr_j$, we apply the modification of the construction of the homotopy $H$ from Lemma \ref{factorizationlemma} that is described in Remark \ref{imageremark}. In particular, there is a homotopy $H:I^n\times I\to Z$ from $f'$ to a map $g\in \Omega^n(Z,z_0)$ such that
\begin{itemize}
\item $f'$ is homotopic to $g$ by a homotopy constant on $(f')^{-1}(X^{\ast})=f^{-1}(X^{\ast})$,
\item for each $R\in\scrr_j$, $g(\partial R)=a_j$ and there is a unique $n$-cube $K_R\subseteq \int(R)$ such that $g|_{K_R}\equiv (f')|_{R}$ and $g(R\backslash \int(K_R))\subseteq A_j$.
\end{itemize}
The construction of $H$ is given by considering the retraction $\zeta_j:Z\to A_{j}^{\ast}$ and applying the construction in Lemma \ref{factorizationlemma} to $\zeta_j\circ f'$ to obtain a homotopy $H_j$, which is the constant homotopy on $I^n\backslash U_j$ and which maps $U_j\times I$ into $A_j$. In particular, the recursive steps used to construct $H_j$ may all be carried out in $A_j$ since $f'(\partial R)\subseteq A_j$ for all $R\in\scrr_m$ and $\{f'|_{R}\mid R\in\scrr_j\}$ clusters at $a_j$. Take $H$ to be the constant homotopy on $(f')^{-1}(X^{\ast})$ and $H(\bfx,t)=H_j(\bfx,t)$ when $\bfx\in U_j$. If $s_j:Z\to X^{\ast}\vee A_{j}^{\ast}$ is the canonical retraction, then $s_j\circ H:I^n\times I\to X^{\ast}\vee A_{j}^{\ast}$ is continuous for all $j$. Hence, by viewing $Z$ as the shrinking adjunction space $Z=\shadj(X^{\ast},x_j,A_{j}^{\ast},a_j)$, it follows that $H$ is continuous. 

The composition homotopy $f\simeq f'\simeq g$ has image in $Y\cup \im(f)$. The sets $\{\scrr_j\}_{j\in\bbn}$ form a factorization of $g$ with respect to $Z=\shadj(X^{\ast},x_j,A_{j}^{\ast},a_j)$. By the second bullet point in the description of $g$, for each $C'$, there is a unique $n$-cube $C''\subseteq K_R\subseteq \int(R)$ such that $g|_{C''}\equiv (f')|_{C'}\equiv f|_{C}$. The $n$-domain $\scrd_i=\{C''\mid \scrc_i\}$ forms a factorization of $g$ that is equivalent to the original factorization $\{\scrc_i\}_{i\in\bbn}$ of $f$.
\end{proof}

Perhaps the most important thing to notice about Lemma \ref{relativelemma} is that no assumptions need to be made about the attachment spaces $B_i$. Indeed, the argument uses the assumed factorization of $f$ without calling upon the spaces $B_i$ in any other way. The next lemma is a strengthened version of Lemma \ref{relativelemma} for the case when $Y$ is a shrinking wedge. No assumptions on the spaces $B_i$ are required for this lemma either.

\begin{lemma}\label{bigtimelemma}
Let $X=\{z_0\}$ and $(A_j,a_j)$ be sequentially $(n-1)$-connected for all $j\in\bbn$. Let $Y=\sw_{j\in\bbn}(A_j,a_j)$ be the shrinking wedge with basepoint $z_0$ and consider the shrinking adjunction space $Z=\shadj(Y,y_i,B_i,b_i)$. Define subspaces $X^{\ast}=\sw_{i\in T}B_i$ and $A_{j}^{\ast}$ of $Z$ as in Remark \ref{iteratedremark}. If $f\in\Omega^n(Z,z_0)$ is in factored form with respect to $Z=\shadj(Y,y_i,B_i,b_i)$ (with factorization $\{\scrc_i\}_{i\in\bbn}$), then $f$ is homotopic to an infinite concatenation $\prod_{i\in T}h_i\cdot\prod_{j\in\bbn}g_j$ such that
\begin{itemize}
\item for all $i\in T$, $h_i\in \Omega^n(B_i,b_i)$
\item for all $j\in\bbn$, $g_j\in \Omega(A_{j}^{\ast},a_j)$ is in factored form in \[A_{j}^{\ast}=\shadj(A_j,\{y_i\}_{i\in T_j},\{B_i\}_{i\in T_j},\{b_i\}_{i\in T_j})\] with a factorization equivalent to the factorization $\{\scrc_i\}_{i\in T_j}$ of $r_j\circ f\in \Omega(A_{j}^{\ast},a_j)$.
\end{itemize}
Moreover, we may choose the homotopy $f\simeq \prod_{i\in T}h_i\cdot\prod_{j\in\bbn}g_j$ to be constant on $f^{-1}(z_0)$ and to have image in $Y\cup \im(f)$.
\end{lemma}

\begin{proof}
Lemma \ref{relativelemma} gives a map $g\in \Omega^n(Z,z_0)$ such that
\begin{enumerate}
\item $f$ is homotopic to $g$ by a homotopy that is constant on $f^{-1}(X^{\ast})$ and with image in $Y\cup \im(f)$,
\item $g$ admits a factorization $\{\scrr_j\}_{j\in\bbn}$ with respect to $Z=\shadj(X^{\ast},x_j,A_{j}^{\ast},a_j)$,
\item $g$ admits a factorization $\{\scrd_i\}_{i\in\bbn}$ with respect to $Z=\shadj(Y,y_i,B_i,b_i)$ that is equivalent to $\{\scrc_i\}_{i\in\bbn}$.
\end{enumerate}
In particular, Condition 2. gives:
\begin{enumerate}[label=(\roman*)]
\item $g|_{D}\in \Omega(B_i,z_0)$ whenever $i\in T$ and $D\in \scrd_i$, 
\item $g|_{R}\in \Omega^{n}(A_{j}^{\ast},a_j)$ for all $R\in \scrr_j$,
\item if $\scrs=\bigcup_{i\in T}\scrd_i\cup\bigcup_{j\in\bbn}\scrr_j$, then $g(I^n\backslash \bigcup\scrs)=z_0$
\end{enumerate}
In other words, $g=\prod_{S\in\scrs}g|_{S}$. The final statement in Lemma \ref{relativelemma} ensures that
\begin{itemize}[label=$(\ast)$]
\item whenever $i\in T_j$ and $D\in\scrd_i$, we have $D\subseteq R$ for some $R\in\scrr_j$.
\end{itemize}
Since $g$ is an $\scrs$-concatenation, we apply an infinite shuffle (Theorem \ref{shuffletheorem}) to collect all $g|_{S}$, $S\in \scrd_i$, $i\in T_j$ into a single $n$-loop $h_i$ and all $g_{S}$, $S\in\scrr_j$ into a single $n$-loop $g_j$. We choose a rearrangement so that $f\simeq \prod_{i\in T}h_i\cdot\prod_{j\in\bbn}g_j$ as in the statement of the lemma. Shuffle homotopies have image within the $n$-loop to which they are applied. In this case, the shuffle will only rearrange the $n$-domains $\scrd_i$, $i\in T_j$ and $\scrr_j$. Therefore, observation $(\ast)$ ensures that there exists a factorization for $ \prod_{i\in T}h_i\cdot\prod_{j\in\bbn}g_j$ with respect to $Z=\shadj(Y,y_i,B_i,b_i)$ that is equivalent to the factorization $\{\scrc_i\}_{i\in\bbn}$ of $f$.
\end{proof}

\begin{remark}\label{bigtimesetupremark}
Recall that since the sequentially $(n-1)$-connected property is preserved by shrinking wedges of based spaces, whenever we assume attachment spaces $(B_i,b_i)$ are sequentially $(n-1)$-connected, we may replace all spaces attached at a fixed point $y\in Y$ with the finite or infinite shrinking wedge $\bigcup\{B_i\mid b_i=y\}$. Hence, we may assume that $i\mapsto b_i$ is injective. Applying this assumption to Lemma \ref{bigtimelemma} would leave at most one space $B_i$ attached at the wedgepoint $z_0$. Assuming $B_1$ is a sole attachment space at $z_0$, the conclusion of the lemma becomes $f\simeq h_1\cdot \prod_{j\in\bbn}g_j$ for a map $h_1\in \Omega^n(B_1,b_1)$.
\end{remark}

Recalling Remark \ref{equivalentremark}, it is desirable to have a version of factored form with a single-factor in each attachment space.

\begin{definition}\label{fullyfactoredef}
Let $Y=\shadj(X,x_j,A_j,a_j)$ and $y_0\in X$. We say that $f\in \Omega^n(Y,y_0)$ is \textit{in single-factor form with respect to the decomposition $\shadj(X,x_j,A_j,a_j)$} if there exists a factorization $\{\scrc_j\}_{j\in\bbn}$ of $f$ such that each $\scrc_j$ contains at most one $n$-cube.
\end{definition}

Note that a map $f$ is in single-factor form if and only if there exists $J\subseteq \bbn$ and $n$-domain $\scrr=\{R_j\mid j\in J\}$ such that $f(I^n\backslash \bigcup\scrr)\subseteq X$ and such that $f(R_j)\subseteq A_j$, $f(\partial R_j)=a_j$ for all $j\in J$.

\begin{lemma}\label{singlefactorlemma}
Let $Y=\shadj(X,x_j,A_j,a_j)$ and $y_0\in X$. If $f\in \Omega^n(Y,y_0)$ is in single-factor form  with factorization $\scrr=\{R_j\mid j\in \bbn\}$ and $f|_{R_j}$ is null-homotopic in $A_j$ for all $j\in \bbn$, then $f$ is homotopic rel. $\partial I^n$ to $r\circ f\in \Omega^n(X,y_0)$
\end{lemma}

\begin{proof}
For each $j\in\bbn$, choose a null-homotopy $H_j:R_j\times I\to A_j$ with $H_j(\bfx,0)=f|_{R_j}(\bfx)$ and $H_j(\partial I^n\times I\cup I\times \{1\})=a_j$. Define a homotopy $H:I^n\times I\to Y$ to be constant homotopy on $f^{-1}(X)$ and so that $H(\bfx,t)=H_j(\bfx,t)$ whenever $\bfx\in R_j$. Each projection $\rho_k\circ H:I^n\times I\to Y_k$ is continuous by the usual Pasting Lemma and so $H$ is continuous. By construction, $H$ is a homotopy from $f$ to $r\circ f$.
\end{proof}

Suppose an $n$-loop $f$ in $Y=\shadj(X,x_j,A_j,a_j)$ admits a factorization $\{\scrc_j\}_{j\in\bbn}$ where each $\scrc_j$ consists of two $n$-cubes $R_{j},S_j$. Unless the core $X$ consists of a single point, there may be no obvious way to perform a $n$-cube shuffle or some other kind of \textit{continuous} deformation that merges each $R_j$ and $S_j$ into a single $n$-cube with the goal of arriving at single-factored form. Even when $X$ consists of a single point (so that $Y=\sw_{\bbn}A_j$) we must have Theorem \ref{shuffletheorem} at our disposal. We focus on this important case in the next section.

\subsection{Homotopy groups of shrinking wedges}\label{shrinkingwedgesection}

Recall that a shrinking wedge $\sw_{j\in\bbn}A_j$ of based spaces is a shrinking adjunction space where $X=\{b_0\}$ is a single point space. We treat this special case first since (1) our approach improves the best-known results and (2) since infinite $n$-loop factorization in shrinking wedges will be required in the following sections. The next two corollaries are direct applications of Lemma \ref{factoredformsequencelemma1} and Theorem \ref{seqnconnectedtheorem1} respectively.

\begin{corollary}\label{wedgefactorizatoincorollary}
Let $n\geq 1$, $(A_j,a_j)$ be sequentially $(n-1)$-connected for all $j\in\bbn$, and $Y=\sw_{j\in\bbn}A_j$ with wedge point $y_0$. Every map $f\in \Omega^n(Y,y_0)$ is homotopic to a map $g\in \Omega^n(Y,y_0)$ that is in factored form by a homotopy that is constant on $f^{-1}(y_0)$ and which as image in $\bigcup\{A_j\mid \im(f)\cap A_j\neq \emptyset\}$.
\end{corollary}

\begin{corollary}\label{swissequentiallynconnected}
If $n\geq 0$ and $\{(A_j,a_j)\mid j\in J\}$ is a countable collection of sequentially $n$-connected based spaces, then $\sw_{j\in\bbn}A_j$ is sequentially $n$-connected at the wedge point.
\end{corollary}

To prove the next lemma, apply an infinitary $n$-cube shuffle (Theorem \ref{shuffletheorem}) to the map $g$ in Corollary \ref{wedgefactorizatoincorollary}. Specifically, if $\scrc$ is a factorization of $g$, then we shuffle the cubes $C\in\scrc$ for which $g|_{C}\in\Omega(A_j,a_j)$ to a disjoint family of $n$-cubes in $\left[\frac{j-1}{j},\frac{j}{j+1}\right]\times I^{n-1}$. Doing so gives a homotopy from $g$ to a true infinite concatenation $\prod_{j\in\bbn}g_j$ where $g_j$ is an $n$-loop in $A_j$ based at $a_j=y_0$. Such a map $\prod_{j\in\bbn}g_j$ is single-factor form with respect to the decomposition $Y=\shadj(\{y_0\},\{y_0\}_{j\in\bbn},A_j,a_j)$. Recall that a shuffle homotopy applied to a $\scrc$-concatenation $g=\prod_{C\in \scrc}g_C$ has image within $\im(g)$.

\begin{lemma}\label{wedgeinfiniteconcatformlemma}
Suppose $n\geq 1$, the spaces $A_j$, $j\in\bbn$ are sequentially $(n-1)$-connected at $a_j\in A_j$, and let $Y=\sw_{j\in\bbn}A_j$ with wedge point $b_0$. Every $n$-loop $f\in \Omega^n(Y,b_0)$ is homotopic to an infinite concatenation $\prod_{j\in\bbn}g_j$ where $g_j\in \Omega^n(A_j,a_j)$ by a homotopy with image in $\bigcup\{A_j\mid \im(f)\cap A_j\neq \emptyset\}$. Moreover, if $f$ was already in factored form, then the homotopy may be chosen to have image in $\im(f)$.
\end{lemma}

\begin{remark}
Although we do not make use of them in the current paper, there are applications of Lemma \ref{manifoldsinsidelemma} and Corollary \ref{manifoldsinsidecorollary} that apply the same argument used in the proof of Theorem \ref{factoredformsequencelemma1}. For example: Suppose $n\geq 2$ and $Y=\sw_{j\in\bbn}(A_j,a_j)$ where $A_j$ is $n$-tame at $a_j$. Then for every map $f\in \Omega^n(Y,y_0)$, there exists a map $g\in \Omega^n(Y,y_0)$ and connected, open $n$-manifolds $M_{j,1},M_{j,2},\dots,M_{j,k_j}\subseteq (0,1)^n$, $j\in\bbn$ such that
\begin{itemize}
\item $f$ is homotopic rel. $\partial I^n$ to $g$,
\item the collection $\{M_{j,i}\mid j\in\bbn,1\leq i\leq k_j\}$ is pairwise-disjoint,
\item $\overline{M_{j,i}}$ is a finite union of $n$-cubes for all $j\in\bbn$, $1\leq i\leq k_j$,
\item $g(\bigcup_{i=1}^{k_j}\overline{M_{j,i}})\subseteq A_j$ for all $j\in\bbn$,
\item $g(I^n\backslash \bigcup_{j\in\bbn}\bigcup_{i=1}^{k_j}M_{j,i})=y_0$.
\end{itemize}
Moreover, this homotopy may be chosen to be the constant homotopy on $f^{-1}(y_0)$ and to have image in $\bigcup\{A_j\mid \im(f)\cap A_j\neq \emptyset\}$.
\end{remark}

At this point, we can generalize the main result in \cite{EK00higher}. We remind the reader that the utility of this generalization is the fact that, unlike the local contractability conditions imposed on the spaces $A_j$ in \cite{EK00higher}, the sequentially $(n-1)$-connected property is closed under many more constructions (Recall Section \ref{subsectionsequentialprops}).

\begin{theorem}\label{wedgeshapeinjectivitythm}
Let $n\geq 2$ and $Y=\sw_{j\in\bbn}A_j$ be the shrinking wedge of a sequence $(A_j,a_j)$, $j\in\bbn$ of based spaces. Then the canonical homomorphism \[\phi:\pi_n\left(Y,b_0\right)\to \prod_{j\in\bbn}\pi_n(A_j,a_j),\quad \phi([f])=([r_1\circ f],[r_2\circ f],[r_3\circ f],\dots )\] is a split epimorphism. Moreover, if $(A_j,a_j)$ is sequentially $(n-1)$-connected for all $j\in\bbn$, then $\phi$ is an isomorphism.
\end{theorem}

\begin{proof}
The first statement is a special case of Theorem \ref{splittheorem}; however, this result was is also proved in \cite[Theorem 2.1]{Kawamurasuspensions} and \cite[Lemma 5.3]{Brazasncubeshuffle}. Suppose $(A_j,a_j)$ is sequentially $(n-1)$-connected for all $j\in\bbn$ and $f\in \Omega^n(Y,b_0)$ is a map such that $r_j\circ f$ is null-homotopic for all $j\in\bbn$. By Lemma \ref{wedgeinfiniteconcatformlemma}, we have $f\simeq \prod_{j\in \bbn}g_j$ where $g_j\in \Omega^n(A_j,a_j)$. Since $g_j\simeq r_j\circ g\simeq r_j\circ f$ is null-homotopic, we may choose a null-homotopy $H_j:I^n\times I\to A_j$ of $g_j$, i.e. with $H_j(\bfx,0)=g_j(\bfx)$ and $H_j(\partial I^n\times I\cup I^n\times\{1\})=a_j$. The infinite horizontal concatenation $\prod_{j\in\bbn}H_j$ now gives a null-homotopy of $\prod_{j\in\bbn}g_j$. We conclude that $f$ is null-homotopic in $Y$. This proves the injectivity of $\phi$.
\end{proof}

When all but finitely many $A_j$ are one-point spaces, Corollary \ref{swissequentiallynconnected} and Theorem \ref{wedgeshapeinjectivitythm} apply to finite wedges of based spaces. Since any compact subset of an infinite wedge $\bigvee_{j\in\bbn}A_j$ with the weak topology lies in a finite sub-wedge, we have the following corollary.

\begin{corollary}\label{wedgecomputationcor}
If $n\geq 2$ and the based spaces $(A_j,a_j)$, $j\in \bbn$ are sequentially $(n-1)$-connected, then $\bigvee_{j\in\bbn}A_j$ is sequentially $(n-1)$-connected at the wedge point $b_0$ and $\pi_n\left(\bigvee_{j\in\bbn}A_j,b_0\right)\cong \bigoplus_{j\in\bbn}\pi_n(A_j,a_j)$.
\end{corollary}

\begin{example}
All of the cases proved in \cite{Edaonedim} are included in Corollary \ref{wedgecomputationcor}. Two important examples (where $n\geq 2$) are the following:
\begin{enumerate}
\item when $A_j=S^n$ for all $j\in\bbn$, i.e. $Y$ is the $n$-dimensional Hawaiian earring, we have
 \[\pi_n(\bbh_n)\cong \prod_{j\in\bbn}\pi_n(S^n)\cong \bbz^{\bbn}.\]
\item when $A_j=C\bbh_n$ for all $j\in\bbn$, $\pi_n\left(\sw_{j\in\bbn}C\bbh_n\right)\cong \prod_{j\in\bbn}\pi_n(C\bbh_n)$ is trivial (recall \ref{coneexample}). We note that the authors of \cite{EK00higher} cannot include this example directly in their main result because $C\bbh_n$ does not meet the required local contractability condition. Rather they call upon singular homology and the Hurewicz theorem. In contrast, our approach follows directly from our use of sequentially $(n-1)$-connected spaces.
\end{enumerate}
\end{example}

\begin{example}
If $\{(A_j,a_j)\}_{j\in\bbn}$ is a sequence of sequentially $(n-1)$-connected and $\pi_n$-residual based spaces, then $\sw_{j\in\bbn}A_j$ is $\pi_n$-residual at the wedge point $x_0$. Indeed, if $\{f_k\}_{k\in\bbn}$ is a sequence of null-homotopic $n$-loops in $\Omega^n(\sw_{j\in\bbn}A_j,x_0)$, we may find $j_1\leq j_2\leq j_3\leq \cdots$ such that $\{j_k\}\to\infty$ and $\im(f_k)\subseteq \sw_{j\geq j_k}A_j$. Applying Lemmas \ref{singlefactorlemma} and \ref{wedgeinfiniteconcatformlemma}, we see that each $f_k$ contracts by a null-homotopy $H_k$ with image in $\sw_{j\geq j_k}A_j$. Since $\{H_k\}_{k\in\bbn}$ converges to $x_0$, we can conclude that $\{f_k\}_{k\in\bbn}$ is sequentially null-homotopic. Hence, $\bbh_n$ is $\pi_n$-residual. Moreover, combined with Corollary \ref{swissequentiallynconnected}, this observation shows that the property conjoined property ``sequentially $(n-1)$-connected and $\pi_n$-residual" of based spaces is closed under forming shrinking wedges.
\end{example}

Generalized universal coverings of shrinking adjunction spaces will generally be indeterminate adjunction spaces with uncountably many attachment spaces. In the next example, we analyze a situation motivated by this phenomenon.

\begin{example}[An intermediate wedge of cones]\label{coneintermediateexample}
Consider the cone $C\bbh_n$ where the basepoint $x_0$ is the image of $(b_0,0)$ in $C\bbh_n$. For each $j\in\bbn$, let $B_j$ be a copy of $\bbh_n$ and $A_j=C B_j$ with basepoint $a_j=x_0\in B_j$. We give the wedge $Y=\bigvee_{j\in\bbn}(A_j,a_j)$ (with wedgepoint $b_0$) a topology that lies strictly between the weak topology and inverse limit topology. If $V\subseteq C\bbh_n$, then we write $V_j$ to denote the corresponding subset of $A_j$. A set $U\subseteq Y$ is open if and only if 
\begin{enumerate}
\item $U\cap A_j$ is open in $A_j$ for all $j$,
\item if $b_0\in U$, then there exists an open neighborhood $V$ of $x_0$ in $C\bbh_n$  such that $\bigcup_{j\in\bbn}V_j\subseteq U$.
\end{enumerate}
If $A_j=C\bbh_n$ has the metric induced from an embedding in $\bbr_{n+2}$, then the topology of $Y$ is induced by the quotient metric inherited from the metric spaces $\{A_j\}_{j\in\bbn}$. It is in this sense that all of the attached cones are ``of the same size." This topology is coarse enough that a compact subset $K\subseteq Y$ can meet $A_{j}\backslash\{a_j\}$ for infinitely many $j$. For instance, let $\mu_{j,k}:S^n\to B_j$ be the inclusion into the $k$-th sphere and define $f:(\bbh_n,b_0)\to (Y,b_0)$ by $f\circ\ell_j=\mu_{j,j}$. Now $f\circ \ell_j$ is null-homotopic in $Y$ for all $j\in\bbn$ since it may be contracted in the cone $A_j$. However, the image of a pointed null-homotopy $H$ of $f$ can only meet the vertex $v_j$ of $A_j$ for finitely many $j$, say $j\in\{1,2,\dots ,k\}$. There is a canonical retraction $R:Y\backslash\{v_1,v_2,\dots ,v_{k}\}\to Z_k$ where $Z_k=\ui\cup\bigcup_{j<k}A_j\cup \bigcup_{j\geq k}B_j$. Thus, $R\circ H$ is a null-homotopy of $R\circ f$ in $Z_k$. However, $R\circ f$ is not null-homotopic since, for all $j$, the $j$-th sphere of $B_j$ is a retract of $Z_k$. This contradiction proves that $Y$ is not $\pi_n$-residual at $b_0$.

Although, it follows from our last observation that $\pi_n(Y,b_0)=\varinjlim_{k}\pi_n(Z_k,b_0)$, further analysis is required to obtain a canonical characterization of $\pi_n(Y,b_0)$ that one might consider a ``computation." We give a brief exposition of this argument and leave the details as an exercise. Let $B_{j,i}=\bbh_n$ denote the $i$-th sphere of $B_j$. If $f \in \Omega^n( Z_k,b_0)$, then using the compactness of $Im(f)$, there exists an unbounded, non-decreasing sequence $s:=m_k\leq m_{k+1}\leq m_{k+2}\leq $ in $\bbn$ and an $n$-loop $g$ such that $f\simeq g$ and \[Im(g)\subseteq W_{k,s}=\bigcup_{j<k}A_j\cup \bigcup_{j\geq k}\bigcup_{i\geq m_j}B_{j,i}.\] The space $W_{k,s}$ is homeomorphic to the sequentially $(n-1)$-connected shrinking wedge $\bigvee_{j<k}C\bbh_n\vee \sw_{j\geq k}\bbh_n$ and therefore, $\pi_n(W_{k,s},b_0)\cong \pi_n(\bbh_n)^{\bbn}\cong \bbz^{\bbn\times\bbn}\cong \bbz^{\bbn}$ is isomorphic to the Baer-Specker group. By defining a natural direct ordering on the set of pairs $(k,s)$ as described above, it follows that $\pi_n(Y,b_0)\cong \varinjlim_{(k,s)}\pi_n(W_{k,s},b_0)$. Hence, $\pi_n(Y,b_0)$ is isomorphic to a direct limit of copies of the Baer-Specker group. In particular, every finitely generated subgroup of $\pi_n(Y,b_0)$ is free abelian.
\end{example}

\section{Shrinking adjunction spaces with dendrite core}\label{sectiondendritecore}

The goal of this section is to prove an analogue of Lemma \ref{wedgeinfiniteconcatformlemma} for shrinking adjunction spaces $Y=\shadj(X,x_j,A_j,a_j)=\varprojlim_{k}Y_k$ where $X$ is a dendrite. Specifically we will seek to show that any $n$-loop is homotopic to another that is in single-factor form. Since generalized covering spaces of shrinking adjunction spaces will generally be indeterminate adjunction spaces, we must pay close attention to the size the homotopies we construct. In the end, we will show that if $f$ is already in factored form, then the homotopy to single-factored form may be constructed within $X\cup \im(f)$. The technical results from Section \ref{subsectionshadjfactorization} will be applied recursively to achieve this. 

As before, the maps $\rho_k:Y\to Y_k$, $r_j:Y\to A_j$, and $r:Y\to X$ will denote the canonical retractions. As noted in Example \ref{dendriteexample}, dendrites are sequentially $(n-1)$-connected at all of their points. Therefore, by Theorem \ref{seqnconnectedtheorem1}, $Y$ is sequentially $(n-1)$-connected at all points. Therefore, we may apply the results of Section \ref{subsectionshadjfactorization} to any shrinking adjunction space of this form. In particular, if $K$ is a compact, connected subspace of $X$, then $K$ is a dendrite and $K\cup \bigcup\{A_j\mid x_j\in K\}$ is a shrinking adjunction space with dendrite core. Hence, we may apply the results of Section \ref{subsectionshadjfactorization} to subspaces of this form as well.

\subsection{Homotopy factorization for an arc core}

In this section, we assume $X=[0,1]$. We briefly recall the notation established in Remark \ref{iteratedremark}. Let $\{x_j\}$ be a sequence of distinct points in $X$, and $A_j$ be a Peano continuum that is sequentially $(n-1)$-connected at $a_j\in A_j$. Set $Y=\shadj([0,1],x_j,A_j,a_j)$. Let $\{y_i\}_{i\in\bbn}$ be a sequence of distinct points in $Y$ and $\{(B_i,b_i)\}_{i\in\bbn}$ be any sequence of based spaces. Set $Z=\shadj(Y,y_i,B_i,b_i)$. Recall that $T=\{i\in\bbn\mid y_i\in [0,1]\}$ and $T_j=\{i\in\bbn\mid y_i\in A_j\backslash\{a_j\}\}$. We also reuse the following notation established in Section \ref{subsectionshadjfactorization} for the respective subspaces of $Z$.
\begin{itemize}
\item $A_{j}^{\ast}=\shadj\left(A_j,\{y_i\}_{i\in T_j},\{B_i\}_{i\in T_j}, \{b_i\}_{i\in T_j}\right)$ is the shrinking adjunction space with core $A_j$ and attachment spaces $B_i$ $i\in T_j$. Note that there is no $B_i$ attached to $a_j$ in $A_{j}^{\ast}$.
\item $X^{\ast}=\shadj\left([0,1],\{y_i\}_{i\in T},\{B_i\}_{i\in T}, \{b_i\}_{i\in T}\right)$ has arc core $[0,1]$ with attached spaces $B_i$, $i\in T$.
\end{itemize}

\begin{remark}[Simplifying the structure of $Z$]\label{simplifyingremark}
By adding one-point spaces to the sequences $\{A_j\}$ and $\{B_i\}$, if necessary, we may assume that
\begin{itemize}
\item $T$ and each $T_j$ are infinite,
\item $\{y_i\mid i\in T_j\}$ is dense in $A_j$,
\item $\{y_i\mid i\in T\}=\{x_j\mid j\in\bbn\}$ is dense in $[0,1]$ and contains $\{0,1\}$.
\end{itemize}
Moreover, let $\scrd=\{\frac{k}{2^n}\mid n\in\bbn,1\leq k\leq 2^{n}-1\}$ denote the set of dyadic rationals in $(0,1)$. Since $\{x_j\mid j\in\bbn\}$ is a countable dense set in $\ui$ containing $0$ and $1$, there is an order-preserving bijection $\phi:\{x_j\mid j\in\bbn\}\to \{0,1\}\cup \scrd$ where both domain and codomain are considered as suborders of $\ui$. Since $\ui$ is the metric completion of both subsets, this bijection extends continuously to a homeomorphism $\ui\to\ui$. Therefore, without changing the homeomorphism type of $Y$ or $Z$, we may assume that $x_1=0$, $x_2=1$, and $\{x_j\mid j\geq 3\}=\scrd$. Since reordering the set $\{x_j\mid j\in\bbn\}$ does not change the homeomorphism type of $Y$, we may assume that $\{x_j\}_{j\geq 3}$ enumerates $\scrd$ in the standard pattern:
\[x_3=1/2,\,x_4=1/4,\,x_5=3/4,\,x_6=1/8,\,x_7=3/8,\,\dots\]
For convenience, we change the indices of our attachment spaces to match the ordering of $\scrd$. For each dyadic rational $d\in\scrd\cup \{0,1\} $, we have $d=x_j=y_i$ for unique $j$ and $i$. We write:
\begin{itemize}
\item $A_d$ for $A_j$ when $d=x_j$,
\item $A_{d}^{\ast}$ for $A_{j}^{\ast}$ when $d=x_j$,
\item $B_d$ for $B_i$ when $d=y_i$.
\end{itemize}
In summary, $Z$ may be regarded as the shrinking adjunction space with core $[0,1]$ and where at each $d\in\scrd\cup\{0,1\}$, there are two spaces attached: the attachment space $B_d$ and the shrinking adjunction space $A_{d}^{\ast}$ (with core $A_d$ and attachment spaces $B_i$, $y_i\in A_d$).
\end{remark}

For $s,t\in\ui$ with $s\neq t$, let $\lambda_{[s,t]}:[0,1]\to [0,1]$ denote the linear path from $s$ to $t$. When $s<t$, we let
\begin{enumerate}
\item[] $\mcy[s,t]=[s,t]\cup \bigcup\{A_d\mid d\in (s,t)\cap\scrd\}$,
\item[] $\mcz[s,t]=[s,t]\cup \bigcup\{A_{d}^{\ast}\vee B_d\mid d\in (s,t)\cap \scrd\}$.
\end{enumerate}
each with the subspace topology inherited from $Z$. Note that $\mcy[0,1]=(Y\backslash (A_0\cup A_1))\cup\{0,1\}$ and $\mcz[0,1]=Z\backslash (A_{0}^{\ast}\cup B_0\cup A_{1}^{\ast}\cup B_1)\cup\{0,1\}$ are precisely the spaces $Y$ and $Z$ with the attachment spaces at the endpoints $0$ and $1$ removed. Similarly, $\mcy[s,t]$ and $\mcz[s,t]$ include $[s,t]$ and all corresponding attachment spaces attached at points in $(s,t)$. 

\begin{remark}
The space $\mcy[s,t]$ is the shrinking adjunction space with core $[s,t]$ and attachment spaces $A_d$, $d\in (s,t)$. The space $\mcz[s,t]$ has many relevant shrinking adjunction space decompositions. However, we focus on the decomposition with core space $\mcy[s,t]$ and attachment spaces $B_i$, $b_i\in \mcy[s,t]$ (recall that infinitely many $B_i$ are attached to each $A_j\backslash\{a_j\})$. The importance of this choice is that for each $u\in\scrd\cap (s,t)$, the core space $\mcy[s,t]$ has a decomposition as a finite wedge $\mcy[s,u]\vee A_u\vee \mcy[u,t]$ and so Lemma \ref{bigtimelemma} applies.
\end{remark}

\begin{terminology}
In what follows, we will use the abbreviated term \textit{factored form} to mean that an $n$-loop $f\in \Omega^n(Z,t)$ based at $t\in[0,1]$ is in factored form with respect to the decomposition $Z=\shadj(Y,y_i,B_i,b_i)$. We will often apply this abbreviated terminology to $n$-loops that have image in proper subspaces of $Z$ such as $\mcz[u,v]$. 
\end{terminology}

We begin by dealing with the endpoints $0$ and $1$ separately. 

\begin{lemma}\label{zeroonelemma}
Suppose $f\in \Omega^{n}(Z,0)$ is in factored form. Then $f\in \Omega^{n}(Z,0)$ is homotopic rel. $\partial I^n$ to a product \[(\alpha_0\cdot\beta_0 ) \cdot(\lambda_{[0,1/2]}\ast g_{1/2})\cdot (\lambda_{[0,1]}\ast (\alpha_1\cdot \beta_1))\] 
where $\alpha_t\in \Omega^n(B_t,t)$, $\beta_t\in \Omega^n(A_{t}^{\ast},t)$ for $t\in\{0,1\}$, and $g_{1/2}\in \Omega^n(\mcz[0,1],1/2)$. Moreover, we may choose $g_{1/2}$ to be in factored form and we may choose the homotopy to have image in $Y \cup \im(f)$.
\end{lemma}

\begin{proof}
First, we view $Z$ as the attachment space with finite wedge core $Y=A_0\vee ((Y\backslash A_0)\cup\{0\})$ and attachment spaces $\{B_i\}_{i\in \bbn}$. Note that $B_0$ is the only space attached at the wedgepoint $0$. By Lemma \ref{bigtimelemma}, $f$ is homotopic to a map $\alpha_0\cdot \beta_0\cdot  k_0$ (by a homotopy in $Y\cup \im(f)$) where
\begin{itemize}
\item $\alpha_0\in \Omega(A_{0}^{\ast},0)$ is in factored form,
\item $\beta_0\in \Omega(B_0,0)$,
\item $k_0\in \Omega^n(\mcz[0,1]\cup A_{1}^{\ast}\cup B_1,0)$ is in factored form.
\end{itemize}
The path conjugate $\lambda_{[1,0]}\ast k_0\in \Omega^n(\mcz[0,1]\cup A_{1}^{\ast}\cup B_1,1)$ is based at $1$ and is in factored form.

Next, we view $\mcz[0,1]\cup A_{1}^{\ast}\cup B_1$ as the shrinking adjunction space with core $\mcy[0,1]\vee A_1$ and attachment spaces $B_i$, $b_i\in \mcy[0,1]\vee A_1$. Applying Lemma \ref{bigtimelemma} to $\lambda_{[1,0]}\ast k_0$ gives $\lambda_{[1,0]}\ast k_0\simeq \alpha_1\cdot \beta_1\cdot k_1$ (by a homotopy in $Y\cup \im(k_0)\subseteq Y\cap \im(f)$) where 
\begin{itemize}
\item $\alpha_1\in \Omega(A_{1}^{\ast},1)$ is in factored form,
\item $\beta_1\in \Omega(B_1,1)$,
\item $k_1\in \Omega^n(\mcz[0,1],1)$ is in factored form.
\end{itemize}
Finally, the path-conjugate $g_{1/2}=\lambda_{[1/2,1]}\ast k_1$ is an $n$-loop in $\mcz[0,1]$ that is in factored form. Every homotopy in the following composition has image in $Y\cup \im(f)$, including all commuting homotopies and path-conjugate cancellations involved. Therefore, the composition itself has image in $Y\cup \im(f)$.
\begin{eqnarray*}
f &\simeq & \alpha_0\cdot \beta_0\cdot  k_0\\
&\simeq & \alpha_0\cdot \beta_0\cdot (\lambda_{[0,1]}\ast (\alpha_1\cdot \beta_1\cdot k_1))\\
& \simeq & \alpha_0\cdot \beta_0\cdot (\lambda_{[0,1]}\ast k_1)\cdot (\lambda_{[0,1]}\ast (\alpha_1\cdot \beta_1))\\
& \simeq & \alpha_0\cdot \beta_0\cdot (\lambda_{[0,1]}\ast (\lambda_{[1,1/2]}\ast g_{1/2}))\cdot (\lambda_{[0,1]}\ast (\alpha_1\cdot \beta_1))\\
& \simeq & (\alpha_0\cdot \beta_0)\cdot (\lambda_{[0,1/2]}\ast g_{1/2})\cdot (\lambda_{[0,1]}\ast (\alpha_1\cdot \beta_1)).
\end{eqnarray*}
\end{proof}

In the next lemma we begin with a factored form $n$-loop $g_{1/2}\in \Omega(\mcz[0,1],1/2)$ (as obtained from Lemma \ref{zeroonelemma}) and factor this map ``at $1/2$" in a way that will support recursive application.

\begin{lemma}\label{recursionlemma}
Suppose $s<u<v<w<t$ are dyadic rationals in $[0,1]$ and $g_v\in \Omega^{n}(\mcz[s,t],v)$ is in factored form. Then $g_v$ is homotopic rel. $\partial I^n$ to a concatenation \[(\lambda_{[v,u]}\ast g_{u})\cdot \alpha_v\cdot \beta_v\cdot (\lambda_{[u,w]}\ast g_{w})\]
where $\alpha_v\in \Omega^n(A_{v}^{\ast},v)$, $\beta_v\in \Omega^n(B_v,v)$, $g_{u}\in \Omega^n(\mcz[s,v],u)$, and $g_{w}\in \Omega^n(\mcz[v,t],w)$. Moreover, we may choose $g_{u},g_{w}$ to be in factored form and we may choose the homotopy to have image in $\mcy[s,t]\cup \im(g_v)$.
\end{lemma}

\begin{proof}
View $\mcz[s,t]$ as the shrinking adjunction space with finite wedge core $\mcy[s,t]=\mcy[s,v]\vee A_v\vee \mcy[v,t]$ and attachment spaces $B_i$, $b_i\in \mcy[s,t]$. By Lemma \ref{bigtimelemma}, $g_v$ is homotopic in $\mcy[s,t]\cup \im(g_v)$ to a concatenation $g_{[s,v]}\cdot \alpha_{v}\cdot\beta_{v}\cdot g_{[v,t]}$ where
\begin{itemize}
\item $\alpha_v\in \Omega(A_{v}^{\ast},v)$ is in factored form,
\item $\beta_v\in \Omega(B_v,v)$,
\item $g_{[s,v]}\in \Omega^n(\mcz[s,v],v)$ is in factored form,
\item $g_{[v,t]}\in \Omega^n(\mcz[v,t],v)$ is in factored form.
\end{itemize}
The path-conjugates $g_{v}=\lambda_{[u,v]}\ast g_{[s,v]}$ and $g_{w}=\lambda_{[w,v]}\ast g_{[v,t]}$ also are in factored form. Since the paths $\lambda_{[u,v]}$ and $\lambda_{[w,v]}$ have image in $[s,t]$, the canonical homotopy 
\[g_{[s,v]}\cdot \alpha_{v}\cdot\beta_{v}\cdot g_{[v,t]}\simeq (\lambda_{[v,u]}\ast g_{u})\cdot \alpha_v\cdot \beta_v\cdot (\lambda_{[v,w]}\ast g_{w})\] also has image in $\mcy[s,t]\cup \im(g_v)$.
\end{proof}

Since we plan to recursively apply Lemma \ref{iteratedremark}, it will be helpful to have a notational mechanism for referring to the domains of $n$-loops in the construction. First, let $D[0,1]=I^n$. Suppose $D[s,t]$ is defined where $s,t$ are dyadic rationals in $\ui$. Let $s<u<v<w<t$ be a partition of $[s,t]$ is into four segments of equal length. Certainly, $u,v,w$ are dyadic rationals. We define four $n$-cubes within $D[s,t]$ as follows.

Although the four cubes to be defined do not depend on the map $g_{v}$ itself, the definition is entirely determined by the structure of the concatenation \[(\lambda_{[v,u]}\ast g_{u})\cdot \alpha_v\cdot \beta_v\cdot (\lambda_{[u,w]}\ast g_{w}).\] Indeed, if we regard this concatenation as a map with domain $D[s,t]$, then we will take $D[s,v],R_v,S_v,D[v,t]$ to be the $n$-cubes in $D[s,t]$ which are the respective domains of $g_u,\alpha_v,\beta_v,g_w$ (See Figure \ref{domainfig3}). Formally, these are defined as follows:
\begin{itemize}
\item $D[s,v]=L_{I^n,D[s,t]}\circ L_{I^n,[0,1/4]\times I^{n-1}}([1/3,2/3]^n)$,
\item $R_v=L_{I^n,D[s,t]}([1/4,1/2]\times I^{n-1})$,
\item $S_v=L_{I^n,D[s,t]}([1/2,3/4]\times I^{n-1})$,
\item $D[v,t]=L_{I^n,D[s,t]}\circ L_{I^n,[1/4,1]\times I^{n-1}}([1/3,2/3]^n)$.
\end{itemize}
Note that $diam(D[s,v])=diam(D[v,t])<\frac{1}{2}diam(D[s,t])$.
\begin{figure}[H]
\centering \includegraphics[height=2.2in]{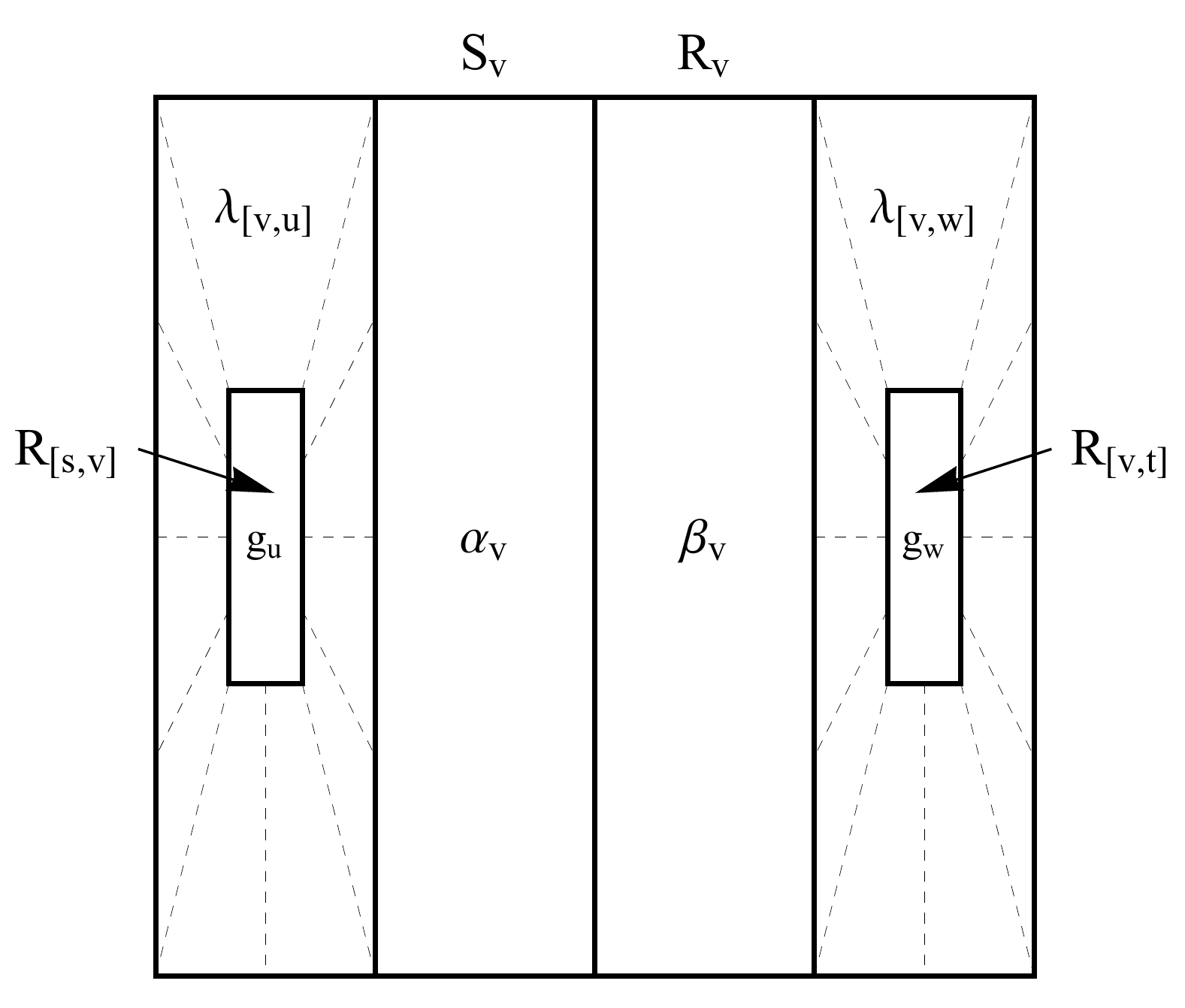}
\caption{\label{domainfig3} A $2$-cube $D[s,t]$, which represents the domain of a map $g_v\in \Omega^n(\mcz[s,t],v)$ where $v=\frac{s+t}{2}$. The domains of the maps $g_u,\alpha_v,\beta_v,g_w$ shown within are $D[s,v],R_v,S_v,D[v,t]$ respectively.}
\end{figure}

For every $p\in\bbn$ and $1\leq j\leq 2^{p-1}$, let $I(p,j)=[\frac{j-1}{2^{p-1}},\frac{j}{2^{p-1}}]$. Note that $I(p,j)$ is subdivided into two segments of equal length by $I(p+1,2j-1)\cup I(p+1,2j)$. Recall that $D[0,1]=I^n$. Given $1\leq j\leq 2^{p-1}$ and $n$-cube $D[\frac{j-1}{2^{p-1}},\frac{j}{2^{p-1}}]$ in $I^n$, we consider the four pre-determined $n$-cubes contained within:
\[D\left[\frac{j-1}{2^{p-1}},\frac{2j-1}{2^{p}}\right],\,R_{\frac{2j-1}{2^{p}}},\,S_{\frac{2j-1}{2^{p}}},\, D\left[\frac{2j-1}{2^{p}},\frac{j}{2^{p-1}}\right]\]
From this recursive definition, we obtain unique $n$-cubes $R_v,S_v$ for each dyadic rational $v\in\scrd$ and an $n$-cube $D[\frac{j-1}{2^{p-1}},\frac{j}{2^{p-1}}]$ for each interval $[\frac{j-1}{2^{p-1}},\frac{j}{2^{p-1}}]$. The $n$-cubes $R_v,S_v$, $v\in\scrd$ are not nested and together form an $n$-domain $\scrr=\{R_v\mid v\in\scrd\}\cup \{S_v\mid v\in\scrd\}$. On the other hand, cubes of the form $D[\frac{j-1}{2^{p-1}},\frac{j}{2^{p-1}}]$ are nested according to inclusion: $D[\frac{j-1}{2^{p-1}},\frac{j}{2^{p-1}}]\subseteq D[\frac{i-1}{2^{q-1}},\frac{i}{2^{q-1}}]$ is proper if and only if $p>q$ and $[\frac{j-1}{2^{p-1}},\frac{j}{2^{p-1}}]\subseteq [\frac{i-1}{2^{q-1}},\frac{i}{2^{q-1}}]$. Moreover, whenever $1\leq j\leq 2^{p-1}$, we have $diam\left(D[\frac{j-1}{2^{p-1}},\frac{j}{2^{p-1}}]\right)\leq \frac{diam(D[0,1])}{2^{p-1}}= \frac{\sqrt{n}}{2^{p-1}}$.

Let \[\scrg_p=\left\{D\left[\frac{j-1}{2^{p-1}},\frac{j}{2^{p-1}}\right]\Big|\, 1\leq j\leq 2^{p-1}\right\}\] so that $C_p=\bigcup\scrg_p$ is a disjoint union $2^{p-1}$-many $n$-cubes of diameter at most $\frac{\sqrt{n}}{2^{p-1}}$. Since \[C_1\supseteq C_2\supseteq C_3\supseteq \cdots\]
where the diameters of the $n$-cubes in $C_p$ approach zero, it must be the case that the intersection $C=\bigcap_{p\in\bbn}C_p=\bigcap_{p\in\bbn}\int(C_p)$ is totally disconnected. Moreover, it is easy to see that $C$ is perfect and therefore homeomorphic to a Cantor set. Recall that this is precisely the set-up from Remark \ref{verticalinfconcatremark} for constructing an infinite vertical concatenations of homotopies.

\begin{lemma}\label{wirefactorizationlemma}
Suppose $g_{1/2}\in \Omega^{n}(\mcz[0,1],1/2)$ is in factored form. Then $g_{1/2}$ is homotopic in $\mcy[0,1]\cup \im(g_{1/2})$ to a map $h_{\infty}\in \Omega^n(\mcz[0,1],1/2)$ such that
\begin{itemize}
\item $h_{\infty}$ is in factored form,
\item $h_{\infty}$ is in single-factor form with the decomposition of $Z$ as the shrinking adjunction space with core $[0,1]$ and attachment spaces $B_d$, $A_{d}^{\ast}$, $d\in \scrd$.
\end{itemize}
\end{lemma}

\begin{remark}
To clarify the statement of Lemma \ref{wirefactorizationlemma}, we take a moment to describe the conclusion to be proved; it asserts that there is an $n$-domain $\scrr=\{R_d\mid\in\scrd\}\cup\{S_d\mid d\in \scrd\}$ such that
\begin{itemize}
\item for each $d\in\scrd$, $h_{\infty}|_{R_d}\in \Omega^n(B_d,b_d)$,
\item for each $d\in\scrd$, $h_{\infty}|_{S_d}\in \Omega^n(A_{d}^{\ast},d)$ is in factored form (with respect to the $B_i$ attached to $A_{d}^{\ast}$),
\item $h_{\infty}\left(I^n\backslash \left(\bigcup_{d\in\scrd}\int(R_d)\cup \int(S_d)\right)\right)\subseteq [0,1]$.
\end{itemize}
\end{remark}

\begin{proof}[Proof of Lemma \ref{wirefactorizationlemma}]
Recall the definition of the $n$-cubes $D[\frac{j-1}{2^{p-1}},\frac{j}{2^{p-1}}]$, $R_{\frac{2j-1}{2^p}}$, $S_{\frac{2j-1}{2^p}}$ for each $p\in\bbn$, $1\leq j\leq 2^{p-1}$. Set $h_1=g_{1/2}\in \Omega^{n}(\mcz[0,1],1/2)$. We apply Lemma \ref{recursionlemma} recursively starting with $v=1/2$ and $u=1/4$, $w=3/4$. For this first step, we have $h_1\simeq h_2$ where 
\[h_2=(\lambda_{[1/2,1/4]}\ast g_{1/4})\cdot \alpha_{1/2}\cdot \beta_{1/2}\cdot (\lambda_{[1/2,3/4]}\ast g_{3/4})\]
where $\alpha_{1/2}\in \Omega^n(A_{1/2}^{\ast},1/2)$ and $\beta_{1/2}\in \Omega^n(B_{1/2},{1/2})$, $g_{1/4}\in \Omega^n(\mcz[0,1/2],1/4)$, and $g_{3/4}\in \Omega^n(\mcz[1/2,1],3/4)$. If $H_1$ is the homotopy from $h_1$ to $h_2$, then we may assume that $\im(H_1)\subseteq \mcy[s,t]\cup \im(g_{1/2})$.

Suppose that the map $h_{p}\in \Omega^n(\mcz[0,1],1/2)$ has been defined so that, for all $1\leq j\leq 2^{p-1}$, if we let $g_{\frac{2j-1}{2^p}}$ be the restriction of $h_p$ to $D[\frac{j-1}{2^{p-1}},\frac{j}{2^{p-1}}]$, then $g_{\frac{2j-1}{2^p}}$ is an $n$-loop in $\mcz[\frac{j-1}{2^{p-1}},\frac{j}{2^{p-1}}]$ based at the midpoint $\frac{2j-1}{2^p}$ and which is in factored form.

We construct $h_{p+1}$ and the homotopy $H_p$ from $h_p$ to $h_{p+1}$ as follows: Define $h_{p+1}$ to agree with $h_p$ on $I^n\backslash \bigcup\{int(D[\frac{j-1}{2^{p-1}},\frac{j}{2^{p-1}}])\mid 1\leq j\leq 2^{p-1}\}$. For each $1\leq j\leq 2^{p-1}$, we apply Lemma \ref{iteratedremark} to $g_{\frac{2j-1}{2^p}}$ to obtain a homotopy $H_{p,j}:D[\frac{j-1}{2^{p-1}},\frac{j}{2^{p-1}}]\times I\to \mcz[\frac{j-1}{2^{p-1}},\frac{j}{2^{p-1}}]$ from $g_{\frac{2j-1}{2^p}}$ to a concatenation \[\Gamma_{p,j}=(\lambda_{[\frac{j-1}{2^{p-1}},\frac{4j-3}{2^{p+1}}]}\ast g_{\frac{4j-3}{2^{p+1}}})\cdot \alpha_{\frac{2j-1}{2^p}}\cdot \beta_{\frac{2j-1}{2^p}}\cdot (\lambda_{[\frac{j-1}{2^{p-1}},\frac{4j-1}{2^{p+1}}]}\ast g_{\frac{4j-1}{2^{p+1}}})\]
viewed as an $n$-loop with domain $D[\frac{j-1}{2^{p-1}},\frac{j}{2^{p-1}}]$ and 
where 
\begin{itemize}
\item $\alpha_{\frac{2j-1}{2^p}}\in \Omega^n(A_{\frac{2j-1}{2^p}}^{\ast},\frac{2j-1}{2^p})$,
\item $\beta_{\frac{2j-1}{2^p}}\in \Omega^n(B_{\frac{2j-1}{2^p}},\frac{2j-1}{2^p})$,
\item $g_{\frac{4j-3}{2^{p+1}}}\in \Omega^n(\mcz[\frac{j-1}{2^{p-1}},\frac{2j-1}{2^p}],\frac{4j-3}{2^{p+1}})$,
\item $g_{\frac{4j-1}{2^{p+1}}}\in \Omega^n(\mcz[\frac{2j-1}{2^p},\frac{j}{2^{p-1}}],\frac{4j-1}{2^{p+1}})$. 
\end{itemize}
Moreover, since $f$ is in factored form, then we may choose $g_{\frac{4j-3}{2^{p+1}}},g_{\frac{4j-1}{2^{p+1}}}$ to be in factored form and we may choose the homotopy $H_{p,j}$ to have image in $\mcy[\frac{j-1}{2^{p-1}},\frac{j}{2^{p-1}}]\cup \im\left(g_{\frac{2j-1}{2^p}}\right)$.

Now, we define the restriction of $h_{p+1}$ to $D[\frac{j-1}{2^{p-1}},\frac{j}{2^{p-1}}]$ to agree with $\Gamma_{p,j}$. We take the homotopy $H_p$ from $h_p$ to $h_{p+1}$ to be the constant homotopy on $I^n\backslash \bigcup\{int(D[\frac{j-1}{2^{p-1}},\frac{j}{2^{p-1}}])\mid 1\leq j\leq 2^{p-1}\}$ and, for each $1\leq j\leqq 2^{p-1}$ to agree with $H_{p,j}$ on $D[\frac{j-1}{2^{p-1}},\frac{j}{2^{p-1}}]\times I$. By the standard Pasting Lemma, $H_p$ is continuous. Note that this definition ensures that 
\begin{itemize}
\item[$(\ast)$] $H_q(D[\frac{j-1}{2^{p-1}},\frac{j}{2^{p-1}}]\times I)\subseteq \mcz[\frac{j-1}{2^{p-1}},\frac{j}{2^{p-1}}]$ for all $q\geq p$.
\end{itemize}
To define the infinite vertical concatenation of the homotopies $\{H_p\}_{p\in\bbn}$, we first define the limit map $h_{\infty}\in \Omega^n(\mcz[0,1],1/2)$. For each $\bfx\in I^n\backslash C$, let $h_{\infty}(\bfx)$ be the limit of the eventually constant sequence $\{h_p(\bfx)\}_{p\in\bbn}$. In particular $h_{\infty}$ agrees with $h_p$ on $I^n\backslash C_p$. If $\bfx \in C$, then for each $p\in\bbn$, there is a unique $1\leq i_p\leq 2^{p-1}$ such that $\bfx\in K_p=D[\frac{i_p-1}{2^{p-1}},\frac{i_p}{2^{p-1}}]$. By construction, these intervals $I_p=[\frac{i_p-1}{2^{p-1}},\frac{i_p}{2^{p-1}}]$ are nested as $I_1\supseteq I_2\supseteq I_3\supseteq \cdots$ and there is a unique real number $t_{\bfx}\in \ui$ such that $\bigcap_{p\in\bbn}I_p=\{t_{\bfx}\}$. We define $h_{\infty}(\bfx)=t_{\bfx}$. The binary decimal representation of any $t\in \ui$ shows that $t$ is the intersection of such a sequence $\{I_p\}_{p\in\bbn}$ (this sequence is unique if $t\in \ui\backslash\scrd$ and there are two if $t\in\scrd$). Since $\mcz[0,1]$ only has spaces attached at dyadic rationals, our choice of definition of the spaces $\mcz[s,t]$ to not include the attachment spaces at the endpoints $s,t$ also ensures that \[\bigcap_{p\in\bbn}\mcz\left[\frac{i_p-1}{2^{p-1}},\frac{i_p}{2^{p-1}}\right]=\{t_{\bfx}\}.\]In this way, $h_{\infty}$ maps $C$ onto $\ui$ in the same manner that the classical Ternary Cantor function maps the Ternary Cantor set onto $\ui$.

Since $\mcz[\frac{j-1}{2^{p-1}},\frac{j}{2^{p-1}}]$ is closed in $\mcz[0,1]$, it follows from $(\ast)$ that $h_{\infty}(D[\frac{j-1}{2^{p-1}},\frac{j}{2^{p-1}}])\subseteq \mcz[\frac{j-1}{2^{p-1}},\frac{j}{2^{p-1}}]$ for all $p\in\bbn$ and $1\leq j\leq 2^{p-1}$.

We check the four conditions for the continuity of $h_{\infty}$ and the infinite vertical concatenation $H_{\infty}$ as stated in Remark \ref{verticalinfconcatremark}.
\begin{enumerate}
\item By construction, $H_p$ is constant on $I^n\backslash \int(C_p)$,
\item By construction, if $\bfx\in I^n\backslash C$, then $h_{\infty}(\bfx)$ is be the eventual value of the limit of the eventually constant sequence $\{h_p(\bfx)\}_{p\in\bbn}$.
\item We check that $h_{\infty}|_{C}:C\to \ui$ is continuous. Let $U$ be an open neighborhood of $h_{\infty}(\bfx)=t_{\bfx}$ and recall that we have $\bfx\in \bigcap_{p\in\bbn}\int(K_p)$ for $K_p=D[\frac{i_p-1}{2^{p-1}},\frac{i_p}{2^{p-1}}]$, $p\in\bbn$. Since $\bigcap_{p\in\bbn}I_p=\{t_{\bfx}\}$, we may find $p\in\bbn$ such that $I_{p}=[\frac{i_p-1}{2^{p-1}},\frac{i_p}{2^{p-1}}]\subseteq U$. As we have already observed,  $h_{\infty}(K_p)\subseteq \mcz [\frac{i_p-1}{2^{p-1}},\frac{i_p}{2^{p-1}}]$. Thus $\bfx\in \int(K_p)$ and $h_{\infty}(C\cap \int(K_p))\subseteq I_p\subseteq U$. 

\item Suppose $\{\bfx_p\}\to \bfx$ where $\bfx_p\in C_p $. Since $\bfx\in C$, we have $h_{\infty}(\bfx)=t_{\bfx}$ as before. Let $U$ be an open neighborhood of $t_{\bfx}$ in $\mcz[0,1]$. We check that $H_p(\{\bfx_p\}\times I)\subseteq U$ for all but finitely many $p$. Suppose to obtain a contradiction that there are natural numbers $p_1<p_2<p_3<\cdots$ and $s_{i}\in \ui$ such that $H_p(\bfx_{p_i},s_i)\notin U$. For each $i\in\bbn$, we have $\bfx_{p_i}\in D[\frac{j_{i}-1}{2^{p_i}},\frac{j_{i}}{2^{p_i-1}}]$ for some $1\leq j_{i}\leq 2^{p_i-1}$ and $\{t_{\bfx}\}= \bigcap_{i\in\bbn}[\frac{j_{i}-1}{2^{p_i}},\frac{j_{i}}{2^{p_i-1}}]$. Find $i_0\in\bbn$ such that $[\frac{j_{i_0}-1}{2^{p_i-1}},\frac{j_{i_0}}{2^{p_i-1}}]\subseteq U\cap \ui$. Since $H_{p_i}(D[\frac{j_i-1}{2^{p_i-1}},\frac{j_i}{2^{p_i-1}}])\subseteq \mcz[\frac{j_{i}-1}{2^{p_i-1}},\frac{j_{i}}{2^{p_i-1}}]$, the points $H_{p_i}(\bfx_{p_i},s_i)$, $i\geq i_0$ must lie an attachment space, that is, there exists a dyadic rational $d_i\in [\frac{j_{i}-1}{2^{p_i-1}},\frac{j_{i}}{2^{p_i-1}}]$ such that $H_{p_i}(\bfx_{p_i},s_i)\in (A_{d_i}^{\ast}\cup B_{d_i})\backslash\{d_i\}$. Then $\{d_i\}_{i\in\bbn}\to t_{\bfx}$. However, $\mcz[0,1]$ has shrinking adjunction decomposition with core $[0,1]$ and attachment spaces $A_{d}^{\ast}\vee B_{d}$, $d\in\scrd$. Therefore, by Lemma \ref{opensetlemma}, we must have $(A_{d_i}^{\ast}\cup B_{d_i})\subseteq U$ for all but finitely many $i\geq i_0$; a contradiction.
\end{enumerate}

For every $d\in\scrd$, we have $h_{\infty}|_{R_d}\equiv \alpha_{d}\in \Omega^n(A_{d}^{\ast},d)$ and $h_{\infty}|_{S_d}\equiv \beta_{d}\in \Omega^n(B_d,d)$. Moreover $h_{\infty}\left(I^n\backslash \left(\bigcup_{d\in\scrd}\int(R_d)\cup \int(S_d)\right)\right)\subseteq [0,1]$. Since $h_{\infty}$ agrees with $h_p$ on $S_d$ and $h_p$ is in factored form, we conclude that $h_{\infty}$. 

Finally, each map $h_p$ and homotopy $H_p$ was constructed so that $\im(h_p)\subseteq \im(H_p)\subseteq Y\cup \im(g_{1/2})$. When $\bfx\notin C$, $h_{\infty}(\bfx)\in \bigcup_{p\in\bbn} h_{p}(I^n\backslash C_p)$ and $h_{\infty}(C)=[0,1]$. Hence, $\im(h_{\infty})\subseteq [0,1]\cup\bigcup_{p\in\bbn} \im(h_p)\subseteq Y\cup \im(g_{1/2})$. This gives the bound $\im(H_{\infty})=\im(h_{\infty})\cup\bigcup_{p\in\bbn} \im(H_p)\subseteq Y\cup \im(g_{1/2})$.
\end{proof}

The next theorem is the combination of Lemmas \ref{zeroonelemma} and \ref{wirefactorizationlemma} and should be considered the main result of this section. Although we allow $A_{j}^{\ast}$ to maintain its established meaning, we release all other fixed notation in order to state the most general version of this result.

\begin{theorem}\label{arctheorem}
Let $Y=\shadj([0,1],x_j,A_j,a_j)$ where each $(A_j,a_j)$ is sequentially $(n-1)$-connected, $Z=\shadj(Y,y_i,B_i,b_i)$, and $z_0\in [0,1]$. If $f\in \Omega^{n}(Z,z_0)$ is in factored form, then $f$ is homotopic in $Y\cup \im(f)$ to a map $g\in \Omega^n(Z,z_0)$ such that
\begin{itemize}
\item $g$ is in factored form with respect to the decomposition $Z=\shadj(Y,y_i,B_i,b_i)$,
\item $g$ is in single-factor form with respect to the decomposition of $Z$ with core $[0,1]$ and attachment spaces $B_i$, $y_i\in [0,1]$ and $A_{j}^{\ast}$, $j\in \bbn$.
\end{itemize}
\end{theorem}

\begin{proof}
By performing the reduction at the beginning of the section, we may assume that $Z$ has the form as described in Remark \ref{shrinkingwedgesection} and we may employ the relevant established notation. Since $f$ is in factored form $\lambda_{[0,z_0]}\ast f\in \Omega^n(Z,0)$. Together, Lemmas \ref{zeroonelemma} and \ref{wirefactorizationlemma} show that $\lambda_{[0,z_0]}\ast f$ is homotopic to a map $h\in \Omega^n(Z,0)$, which satisfies the two conditions in the statement of the lemma. Moreover if $H$ is the homotopy from $\lambda_{[0,z_0]}\ast f$ to $h$, then we may choose $H$ so that $\im(H)\subseteq Y\cup \im(\lambda_{[0,z_0]}\ast f)=Y\cup \im(f)$. The path-conjugate $g=\lambda_{[z_0,0]}\ast h\in \Omega^n(Z,z_0)$ also satisfies the two conditions in the statement and the homotopy $H$ extends to a homotopy $ \lambda_{[z_0,0]}\ast (\lambda_{[0,z_0]}\ast f)\simeq g$ with image in $\im(H)\cup [0,1]\subseteq Y\cup \im(f)$. Finally, the homotopy $f\simeq \lambda_{[z_0,0]}\ast (\lambda_{[0,z_0]}\ast f)$ excising the path-conjugates has image in $[0,1]\cup \im(f)$. The vertical composition of these homotopies gives the desired homotopy $f\simeq g$.
\end{proof}

The above results assume that $f$ is already in factored form. We must assume the attachments spaces $(B_i,b_i)$ are sequentially $(n-1)$-connected ensures that any $n$-loop in $Z$ can be homotoped to such a form. In the next corollary, we consider the case from above where each space $A_j$ is a one-point space.

\begin{corollary}\label{arcorollary}
Let $Y=\shadj([0,1],y_i,B_i,b_i)$ where each $(B_i,b_i)$ is sequentially $(n-1)$-connected and let $y_0\in [0,1]$. Every $n$-loop $f\in\Omega^n(Y,y_0)$ is homotopic rel. $\partial I^n$ to an $n$-loop $g$ that is in single-factor form. Moreover, the canonical homomorphism $\phi:\pi_n(Y,y_0)\to \prod_{i\in\bbn}\pi_n(B_i,b_i)$ is an isomorphism.
\end{corollary}

\begin{proof}
By Lemma \ref{factoredformsequencelemma1}, we may assume that $f$ is in factored form. By Theorem \ref{arctheorem}, $f$ is homotopic to a map $g$ in single-factor form. For the second statement, recall that $\phi$ is surjective by Theorem \ref{splittheorem}. Suppose $[g]\in\pi_n(Y,y_0)$ where $\phi([g])=([r_i\circ g])_{j\in \bbn}$ is trivial where $r_i:Y\to B_i$ is the canonical retraction. By the first part, we may assume $g$ is in single-factor form. Let $\{\scrc_i\}_{i\in\bbn}$ be a factorization of $g$ with $\scrc_i=\emptyset$ or $\scrc_i=\{R_i\}$. Since $r_i\circ g\simeq g|_{R_i}$ is null-homotopic in $B_i$ for all $i\in\bbn$, it follows from Lemma \ref{singlefactorlemma} that $g$ is null-homotopic. Thus $\phi$ is injective.
\end{proof}

\subsection{Homotopy factorization for a dendrite core}

Here, we apply the results of the previous subsection to prove analogous results for shrinking adjunction spaces with dendrite core (see Figure \ref{fig6}). Recall from Example \ref{dendriteexample} that a dendrite is a locally arc-wise connected and uniquely arc-wise connected compact metric space. Every dendrite is sequentially $n$-connected for all $n\geq 0$.

\begin{figure}[H]
\centering \includegraphics[height=1.7in]{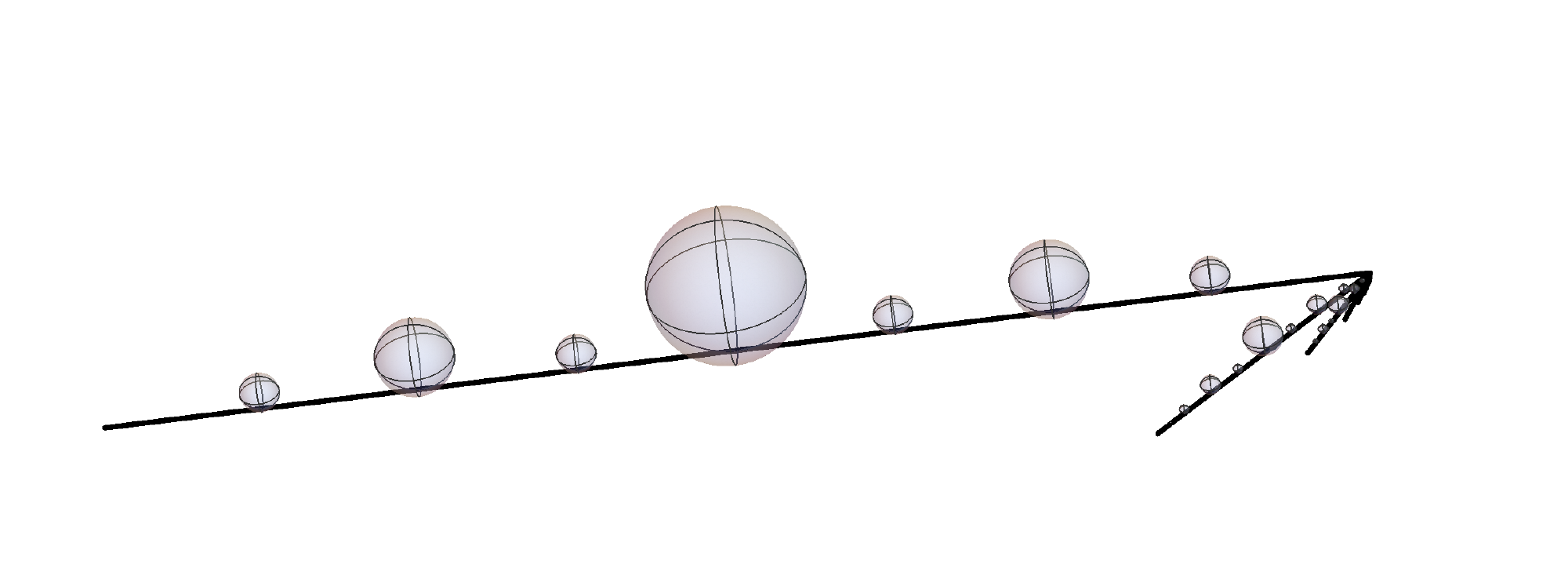}
\caption{\label{fig6} A shrinking adjunction space obtained by attaching $2$-spheres to a shrinking wedge of arcs.}
\end{figure}

Most of our effort will go toward proving the following theorem.

\begin{theorem}\label{dendritetheorem}
Let $Z=\shadj(Y,y_i,B_i,b_i)$ where $Y$ is a dendrite and $z_0\in Y$. If $f\in \Omega^{n}(Z,z_0)$ is in factored form, then $f$ is homotopic in $Y\cup \im(f)$ to an $n$-loop $f_{\infty}\in \Omega^n(Z,z_0)$ that is in single-factor form.
\end{theorem}

First, we recall some well-known characterizations of dendrites \cite[Chapter 10]{Nadler}. Fix a non-degenerate dendrite $Y$, that is, a dendrite which contains at least one arc. A point $y\in Y$ is a \textit{cut-point} if $Y\backslash\{y\}$ has more than one connected component. If $Y\backslash\{y\}$ has at least three components, then we call $y$ a \textit{branch-point}. If $y\in Y$ is not a cut-point, then we call it an \textit{end-point} or simply an \textit{end} of $Y$. Let $Cut(Y)$, $Br(Y)$, and $End(Y)=Y\backslash Cut(Y)$ denote the subsets of cut-points, branch-points, and end-points respectively. These sets are related to each other and $Y$ in the following ways:
\begin{itemize}
\item $Br(Y)\subseteq Cut(Y)$ and $Cut(Y)\backslash Br(Y)$ is empty or a disjoint union of open arcs.
\item $End(Y)$ may be uncountable but $Br(Y)$ must be countable. 
\item $End(Y)\cup Br(Y)$ is closed in $Y$.
\end{itemize}

\begin{remark}[Viewing dendrites as shrinking adjunction spaces]\label{dendriteretractionremark}
If $A\subseteq Y$ is closed and connected, then $A$ is a dendrite. Let $\scrw$ be the set of connected components of $Y\backslash A$. For each $W\in\scrw$, $\ov{W}$ is a dendrite and there is a point $w\in A$ such that $\overline{W}\cap A=\{w\}$. If $\scrw=\{W_1,W_2,W_3,\dots\}$ is infinite, then $\lim_{j\to\infty}diam(\overline{W_j})= 0$. Since $Y$ is compact, we have shrinking adjunction decomposition $Y=\shadj(A,w_j,\overline{W_j},w_j)$. We refer to this decomposition of $Y$ as the \textit{$A$-decomposition} of $Y$.

There exists a deformation retraction $R_{A}:Y\times I\to Y$ of $Y$ onto $A$. Such a retraction may be constructed by setting $R_{A}(a,t)=a$ when $a\in A$ and defining the restriction of $R_{A}|_{\overline{W}\times I}:\overline{W}\times I\to \overline{W}$, $W\in\scrw$ to be any choice of basepoint-preserving contraction of $\overline{W}$ to $w_j\in \ov{W}\cap A$.
\end{remark}

\begin{example}[Wazewski universal dendrite]\label{universaldendriteexample}
It is instructive to consider the above decomposition of a dendrite in the ``worst case" scenario, that is, when $Y$ is the \textit{Wazewski universal dendrite} (see \cite[Chapter K]{Whyburn} and \cite[10.37]{Nadler}). This dendrite may be constructed directly in the plane inductively by starting with a homeomorphic copy of $\sw_{\bbn}(I,0)$, attaching smaller copies of this space in a dense fashion, and taking the closure to obtain uncountably many ends. Every dendrite is homeomorphic to a subspace (and, therefore, is a deformation retract) of the universal dendrite.
\end{example}

\begin{figure}[H]
\centering \includegraphics[height=1.4in]{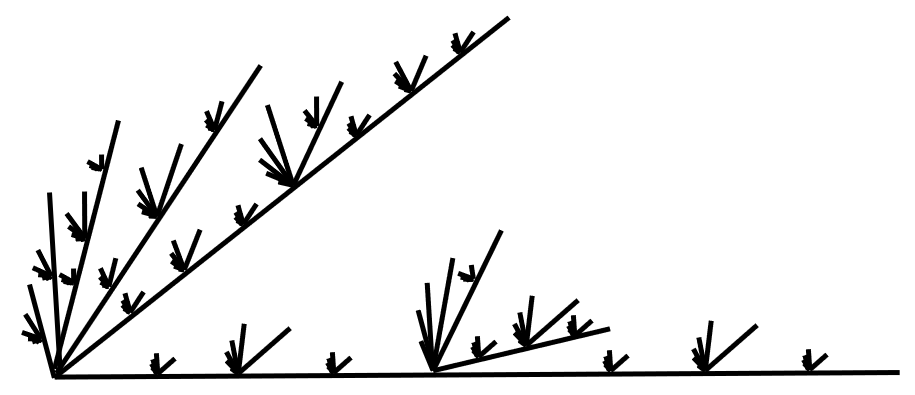}
\caption{\label{fig7}The Wazewski universal dendrite}
\end{figure}

\begin{remark}[Approximation by finite trees]\label{approximationremark}
We will also use a choice of a sequence of arcs in $Y$ to characterize $Y$ as an inverse limit of finite trees. By Theorem 10.27 and Corollary 10.28 of \cite{Nadler}, any non-degenerate dendrite $Y$ may be written as the union of $End(Y)$ and an infinite tree $T=\bigcup_{p=1}^{\infty}X_p$ where
\begin{itemize}
\item for all $p\in\bbn$, $X_p$ is an arc,
\item for all $p\in\bbn$, $T_p=\bigcup_{i=1}^{p}X_i$ is a finite tree such that  $X_{p+1}\cap T_p=\{x_{p+1}\}$ is an end-point of $X_{p+1}$,
\item $\ds\lim_{p\to\infty}diam(X_p)=0$.
\end{itemize}
There is a retraction $\delta_{p+1,p}:T_{p+1}\to T_p$ that collapses $X_{p+1}$ to $x_{p+1}$. We may identify $Y$ with the inverse limit $\varprojlim_{p}(T_p,\delta_{p+1,p})$. Certainly, the choice of approximating trees $T_p$ for $Y$ is not unique.
\end{remark}

\begin{example}
While our proof of Theorem \ref{dendritetheorem} does not depend on our choice of arcs $\{X_p\}_{p\in\bbn}$, it is instructive to realize that (1) the weak topology on the tree $T$ with respect to the arcs $X_p$ will not always agree with the subspace topology inherited from $Y$ and (2) the end of the arc $X_p$ not attached to $T_{p-1}$ will only sometimes end up as an element of $End(Y)$. For instance, if $Y=[0,1]$, then the natural choice is to take $X_p=\left[\frac{p-1}{p},\frac{p}{p+1}\right]$ so that $T=[0,1)$. If $Y=\sw_{j\in\bbn}(I,0)$, then there is a choice of $\{X_p\}_{p\in\bbn}$ that give $T=\sw_{j\in\bbn}([0,1),0)$ and another choice for which $T=Y$.
\end{example} 

In the next lemma, we reduce our proof of Theorem \ref{dendritetheorem} to a special case to which we can readily apply Theorem \ref{arctheorem}. Although it is possible to prove Theorem \ref{dendritetheorem} without this reduction, we find that it allows us to avoid extra casework.

\begin{lemma}\label{universalreduction}
Suppose the conclusion of Theorem \ref{dendritetheorem} holds for shrinking adjunction spaces $Z=\shadj(Y,y_i,B_i,b_i)$ where the set $\{y_i\mid i\in\bbn\}$ of attachment points enumerates a dense in $Y\backslash End(Y)$ and such that there is an inverse limit decomposition $Y=\varprojlim_{p}(T_p,\delta_{p+1,p})$ using arcs $X_p\subseteq Y$ for which $\{y_i\mid i\in\bbn\}$ meets each arc $X_p$ in infinitely many points. Then Theorem \ref{dendritetheorem} holds for all dendrites.
\end{lemma}

\begin{proof}
Supposing the hypothesis, let $Y'$ be any dendrite, $z_0\in Y'$, and $Z'=\shadj(Y',y_i,B_i,b_i)$ where each $(B_i,b_i)$ is sequentially $(n-1)$-connected. Let $f\in\Omega^n(Z',z_0)$ be in factored form. As usual, by replacing all spaces $B_i$ attached at a fixed point of $Y$ with their (finite or infinite shrinking) wedge and implicitly using \ref{swissequentiallynconnected}, we may assume that $i\mapsto y_i$ is injective.

For each $i\in\bbn$, attach an arc $K_i$ of diameter $\frac{1}{i}$ to $Y'$ by identifying on one end of $K_i$ with $y_i$. The resulting space $Y$ is a dendrite for which $\{y_i\mid i\in\bbn\}\cap End(Y)=\emptyset$. Finds arcs $X_p$, $p\in\bbn$ in $Y$ such that $Y=\varprojlim_{p}(T_p,\delta_{p+1,p})$ as in Remark \ref{approximationremark}. Since $Y$ is hereditarily separable, we may find a sequence of distinct points $w_2,w_4,w_6,\dots$ in $End(Y)\backslash (Y\cup \{y_i\mid i\in \bbn\})$ such that $\scra=\{w_{2j}\mid j\in\bbn\}\cup \{y_i\mid i\in\bbn\}$ is dense in $Y\backslash End(Y)$ and such that $\scra$ meets $X_p$ in infinitely many points for all $p\in\bbn$. We combine these sequences by defining $w_{2i-1}=y_i$, $i\in\bbn$. For all $i\in\bbn$, let $A_{2i-1}=B_i$, $a_{2i-1}=b_i$, let $A_{2i}$ be a one-point space $\{a_{2i}\}$. Consider the shrinking adjunction space $Z=\shadj(Y,w_j,A_j,a_j)$ and note that $Z\backslash Z'=Y\backslash Y'$.

We now regard $f$ as an $n$-loop in $Z$. By assumption, $f$ is homotopic to a map $g\in \Omega^n(Z,z_0)$, which is in single-factor form, by a homotopy $H:I^n\times I\to Z$ with image in $Y\cap \im(f)$. Since $\bigcup_{j\in\bbn}A_j\subseteq Z'\cup Y$, the canonical retraction $Y\to Y'$ extends to a retraction $r:Z\to Z'$. Set $G=r\circ H$ and $g'=r\circ g\in \Omega^n(Z',z_0)$. Now $G:I^n\times I\to Z'$ is a homotopy from $r\circ f=f$ to $g'$. Since $\im(H)\subseteq Y\cap \im(f)$ and $r(\im(f))=\im(f)$, $\im(G)=r(\im(H))\subseteq Y'\cap \im(f)$. Moreover, since $\bigcup_{i\in\bbn}B_i\subseteq Z'$, the maps $g'$ and $g$ agree on $\bigcup_{i\in\bbn}g^{-1}(B_i\backslash\{b_i\})$. Therefore, if we take $\{\scrc_i\}_{i\in\bbn}$ to be a factorization of $g$ where each $\scrc_i$ contains at most one $n$-cube, then we may define a corresponding factorization $\{\scrd_j\}_{j\in\bbn}$ of $g'$ by $\scrd_{2i-1}=\scrc_i$ and $\scrd_{2i}=\emptyset$. We conclude that $g'$ is in single-factor form. This verifies that the conclusion of Theorem \ref{dendritetheorem} holds for $Y'$.
\end{proof}

In light of Lemma \ref{universalreduction}, we will, from now one, assume that $Z=\shadj(Y,y_i,B_i,b_i)$ where $Y=\varprojlim_{p}(T_p,\delta_{p+1,p})$ is a dendrite and $\{y_i\}_{i\in\bbn}$ is an enumeration of subset of $Y\backslash End(Y)$ that meets each $X_p$ in infinitely many points. Fix a basepoint $z_0=x_1\in X_1\backslash\{y_i\mid i\in\bbn\}$. 

\begin{remark}
As in the previous section, if $R$ as an $n$-cube and $z'\in Z'\subseteq Z$, we simply say that an $n$-loop $f:(R,\partial R)\to (Z',z')$ is in ``factored form" if it is in factored form with respect to $Z=\shadj(Y,y_i,B_i,b_i)$ when viewed as a map $f:(R,\partial R)\to (Z,z')$. 
\end{remark}

Using the inverse limit description of $Y$, we construct a sequence of sub-dendrites of $Y$. The set $\chi=\{X_p\mid p\in\bbn\}$ of arcs inherits a canonical partial ordering: $X_p\leq X_q$ if and only if there is a unique finite sequence $p=p_0< p_1<p_2<\cdots <p_n=q$ where $x_{p_j}\in X_{p_{j-1}}$ for all $1\leq j\leq n$, that is, if $X_{p_1}$ is attached to $X_p$, $X_{p_2}$ is attached to $X_{p_1}$, and so on. Set $X_{\geq p}=\bigcup_{q\geq p}X_q$. Then $Y_p=\ov{X_{\geq p}}$ is a dendrite and $Y_p\subseteq Y_q$ whenever $X_{p}\leq X_q$.

If $\bfs:=X_{p_1}<X_{p_2}<X_{p_3}<\cdots$ is an infinite increasing sequence in $\chi$, then $\bigcap_{j\in\bbn}Y_{p_j}$ consists of a single point $e_{\bfs}\in End(Y)$. In terms of the inverse limit characterization $Y=\varprojlim_{p}T_p$, $e_{\bfs}$ is identified with the unique sequence $(t_1,t_2,t_3,\dots)\in \prod_{p}T_p $ satisfying $t_{p_{j-1}}=x_{p_j}$ for all $j\in\bbn$. We refer to any end of the form $e_{\bfs}\in End(Y)$ as an \textit{infinite end} of $Y$. For any open neighborhood $U$ of $e_{\bfs}$, we have $Y_{p_j}\subseteq U$ for all but finitely many $j$.

Set $J_0=\{1\}$. For $p\geq 1$, let $J_p=\{q>p\mid X_q\geq X_p\text{ and }x_q\in X_p\}$, i.e. the indices $q$ for which the arc $X_q$ is directly attached to $X_p$ in the tree $T_q$. Note that the sets $J_p$, $p\geq 0$ form a partition of $\bbn$ and that, for $p\geq 1$, $J_p$ may be finite or infinite. Using the $X_p$-decomposition of $Y_p$ from Remark \ref{dendriteretractionremark}, each dendrite $Y_p$ has shrinking adjunction space decomposition \[Y_p=\shadj(X_p,\{x_q\}_{q\in J_p},\{Y_q\}_{q\in J_p},\{x_q\}_{q\in J_p}).\] Let $K_p=\{i\in\bbn\mid y_i\in Y_p\backslash\{x_p\}\}$ and $Z_p=Y_p\cup \{B_i\mid i\in K_p\}$. Notice that for each $p\in\bbn$, $Z_p$ has two relevant decompositions:
\begin{enumerate}
\item $\shadj(Y_p,\{y_i\}_{i\in K_p},\{B_i\}_{i\in K_p},\{b_i\}_{i\in K_p})$,
\item the decomposition with core $X_p$ and attachment spaces $\{Z_q\mid q\in J_p\}$ and $\{B_i\mid y_i\in X_{p}\backslash \{x_p\}\}$ (we have chosen $z_0\neq y_i$ for all $i$ to avoid a special case for $p=1$).
\end{enumerate}
Additionally, if $\bfs:=X_{p_1}<X_{p_2}<X_{p_3}<\cdots$ is an infinite increasing sequence in $\chi$, then $\bigcap_{j\in\bbn}Z_{p_j}=\{e_{\bfs}\}$. Since $X_p$ is an arc, the spaces $Z_p$ satisfy the hypotheses of Theorem \ref{arctheorem} and therefore permit its recursive application. To match the notation in the previous section, we let $X_{p}^{\ast}=X_p\cup \bigcup\{B_i\mid y_i\in X_{p}\backslash \{x_p\}\}$.

\begin{lemma}\label{opensetend}
Suppose $\bfs:=X_{p_1}<X_{p_2}<X_{p_3}<\cdots$ is an infinite increasing sequence in $\chi$ corresponding infinite end $e_{\bfs}\in End(Y)$. If $U$ is an open neighborhood of $e_{\bfs}$ in $Z$, then $Z_{p_j}\subseteq U$ for all but finitely many $j$.
\end{lemma}

\begin{proof}
As noted earlier, we have $Y_{p_j}\subseteq U\cap Y$ for all but finitely many $j$. Suppose to obtain a contradiction, that we can find points $ Z_{p_j}\backslash U\neq\emptyset$ for infinitely many $j$. Replacing $\bfs$ with the subsequence of $X_{p_j}$ for which $ Z_{p_j}\backslash U\neq\emptyset$, we may assume that for every $j\in\bbn$, there exists a point $z_j\in Z_{p_j}\backslash U$. Then there exists $i_j\in\bbn$ such that $z_j\in B_{i_j}\backslash \{b_{i_j}\}$ and $b_{i_j}=y_{i_j}\in Y_{p_j}$. However, since $\{y_{i_j}\}_{j\in\bbn}\to e_{\bfs}$, Lemma \ref{opensetlemma} ensures that we must have $B_{i_j}\subseteq U$ for all but finitely many $j$; a contradiction.
\end{proof}

\begin{proof}[Proof of Theorem \ref{dendritetheorem}]
Suppose $Z=\shadj(Y,y_i,B_i,y_i)$ is a shrinking adjunction space where $Y$ is a non-degenerate dendrite with inverse limit structure $\varprojlim_{p}(T_p(Y),\delta_{p+1,p})$ given by arcs $X_p$, $p\in\bbn$. Suppose $z_0\in X_1$ and let $f\in \Omega^n(Z,z_0)$ be in factored form.

Recall that $Y=Y_1$ and $Z=Z_1$. Set $f_1=f$. Applying Theorem \ref{arctheorem}, we see that $f_1$ is homotopic to a map $f_2\in \Omega^n(Z,z_0)$ such that $f_2$ is in factored form and $f_2$ is in single-factor form with the decomposition of $Z_1$ into core $X_1$ and attachment spaces $\{Z_q\mid q\in J_1\}$ and $\{B_i\mid y_i\in X_{1}\backslash \{z_0\}\}$. We may choose the homotopy $H_1$ from $f_1$ to $f_2$ to have image in $Y\cup \im(f_1)$. Thus $(f_2)\subseteq Y\cup \im(f)$. Let $S_2=J_1$. Then there is an $n$-domain $\scrr_2=\{R_q\mid q\in S_2\}$ such that $(f_2)|_{R_q}\in \Omega^n(Z_q,x_q)$, $q\in J_1$ and $f_2(I^n\backslash \bigcup_{q\in J_1}\int(R_q))\subseteq X_{1}^{\ast}$. Although it is implicit in the proof of Theorem \ref{arctheorem}, we may apply Lemma \ref{shrinkingcubelemma}, if necessary, to ensure that $diam(R_q)<\frac{1}{3}$ for all $q\in S_2$. Since $X_2$ must be attached to $X_1$, we have $2\in S_2$.

Suppose, for $p\geq 2$, we have constructed $f_p\in \Omega^n(Z,z_0)$ with $\im(f_p)\subseteq Y\cup \im(f)$, set $S_{p}\subseteq \bbn$ with $p\in S_p$, and an $n$-domain $\scrr_p=\{R_q\mid q\in S_p\}$ such that $(f_p)|_{R_q}\in \Omega^n(Z_q,x_q)$, $q\in S_p$ and $diam(R_q)<\frac{1}{3^q}$. Since $p\in S_p$, we will construct $f_{p+1}$ from $f_p$ by only altering values in $R_p$. Recall the two shrinking adjunction decompositions of $Z_p$. Theorem \ref{arctheorem} gives that $(f_p)|_{R_p}:(R_p,\partial R_p)\to (Z_p,x_p)$ is homotopic rel. $\partial R_p$ to a map $g_p:(R_p,\partial R_p)\to (Z_p,x_p)$ such that $g_p$ is in factored form and $g_p$ is in single-factor form with the decomposition of $Z_p$ into core $X_p$ and attachment spaces $\{Z_q\mid q\in J_p\}$ and $\{B_i\mid y_i\in X_{p}\backslash \{x_p\}\}$. Let $G_p:R_p\times I\to Z_p$ be the homotopy from $(f_p)|_{R_q}$ to $g_p$ and recall that we may choose $G_p$ to have image in $Y_p\cup f_p(R_p)$. Let $f_{p+1}:(I^n,\partial I^n)\to (Z,z_0)$ be the map, which agrees with $f_p$ on $I^n\backslash \int(R_p)$ and so that $(f_{p+1})|_{R_p}=g_p$. Define homotopy $H_p:I^n\times I\to Z$ from $f_p$ to $f_{p+1}$ to be the constant homotopy on $I^n\backslash \int(R_p)$ and to agree with $G_p$ on $R_p\times I$. Note that $\im(H_p)\subseteq Y\cup \im(f_p)$. 

The single-factor form of $g_p$ ensures that there is an $n$-domain $\{R_q\mid q\in J_{p}\}$ of $n$-cubes in $R_p$ such that $g|_{R_q}\in \Omega^n(Z_q,x_q)$, $q\in J_p$ and $(g_p)(R_p\backslash \bigcup_{q\in J_p}\int(R_q))\subseteq X_{p}^{\ast}$. Applying Lemma \ref{shrinkingcubelemma}, if necessary, we may assume that $diam(R_q)<\frac{1}{3^p}$ for all $q\in J_p$. Set $S_{p+1}=(S_p\backslash\{p\})\cup J_{p}$ and let $\scrr_{p+1}=\{R_q\mid q\in S_{p+1}\}$. Then $(f_{p+1})|_{R_q}\in \Omega^n(Z_q,x_q)$ for all $q\in S_{p+1}$. Since $X_{p+1}$ is attached to some $X_q$, $q\leq p$, we have $p+1\in \bigcup_{q=1}^{p}J_q$ and thus $p+1\in S_{p+1}$. This completes the induction.

Consider the inductively constructed sequence of maps $\{f_p\}_{p\in\bbn}$ in $\Omega^n(Z,z_0)$ and homotopies $H_p$ from $f_p$ to $f_{p+1}$ where, by construction, $\im(H_p)\subseteq Y\cup \im(f)$ for all $p\in\bbn$. We will construct the infinite vertical concatenation of the homotopies $\{H_p\}_{p\in\bbn}$. First, we construct the limit map $f_{\infty}$. Let $C_1=I^n$. Although $H_p$, $p\geq 2$ is the constant homotopy on $I^n\backslash \int(R_p)$, we let $C_p=\bigcup\{R_q\mid q\in S_p\}$. Since $R_p\subseteq C_p$, $H_p$ is the constant homotopy on $I^n\backslash \int(C_p)$. Our construction ensures that $C_{1}\supseteq C_{2}\supseteq C_3\supseteq \cdots$ where $C=\bigcap_{p\in\bbn}C_p=\bigcap_{p\in\bbn}\int(C_p)$ is compact and totally disconnected. It is possible that $C=\emptyset$, for instance, if $\im(f)$ meets at most finitely many $X_p$. It is also possible that $C$ has isolated points. If $\bfx\in I^n\backslash C$, then we let $f_{\infty}(\bfx)$ be the limit of the eventually constant sequence $\{f_p(\bfx)\}_{p\in\bbn}$. To define $f_{\infty}$ on $C$, recall the partially ordered set $\chi$ of arcs $X_p$.

If $\bfx\in C$, then for every $p\in\bbn$, there is a unique $q_p\in S_p$ such that $\bfx\in R_{q_p}$ and such that the intersection of the non-increasing sequence of $n$-cubes $R_{q_1}\supseteq R_{q_2}\supseteq R_{q_3}\supseteq $ is $\bigcap_{p\in\bbn}\int(R_{q_p})=\{\bfx\}$. In $\chi$, this corresponds to a non-decreasing infinite sequence $\bfs(\bfx):=X_{q_1}\leq X_{q_2}\leq X_{q_3}\leq \cdots $ where either $X_{q_{p+1}}=X_{q_p}$ or $X_{q_{p+1}}>X_{q_p}$ are consecutive terms in the partial order $\chi$. The latter case means that $X_{q_{p+1}}$ is attached to $X_{q_p}$ at $x_{q_{p+1}}$. The sequence $s(\bfx)$ may only remain constant at a single value of $X_{q}$ for finitely many terms. Hence, $\bigcap_{p\in\bbn}Y_{q_p}=\{e_{\bfs}\}$ for a unique infinite end $e_{\bfs(\bfx)}\in End(Y)$. Define $f_{\infty}(\bfx)=e_{\bfs(\bfx)}$. 

For fixed $p\in\bbn$, $Z_p$ is closed in $Z$, $e_{\bfs(\bfx)}\in \bigcap_{q\geq p}f_{q}(R_{p})$, and we have $f_{q}(R_p)\subseteq Z_p$ for all $q\geq p$. It follows that for all $p\in\bbn$, $f_{\infty}(R_p)\subseteq Z_p$. Given any $i\in\bbn$, we have ensured by Lemma \ref{universalreduction} that $y_i\notin End(Y)$. Hence, for any given $i\in\bbn$, the preimages $f_{p}^{-1}(B_i\backslash\{b_i\})$ stabilize as $p\to \infty$. In particular, the construction ensures that there is an $n$-domain $\scrs=\{D_i\mid \im(f)\cap B_i\backslash\{b_i\}\neq\emptyset\}$ such that $(f_{\infty})|_{D_i}\in \Omega^n(B_i,b_i)$ for all $D_i\in\scrs$ and $f_{\infty}(I^n\backslash \bigcup\scrs)\subseteq Y$. Therefore, once we show that $f_{\infty}$ is continuous, it follows that $f_{\infty}$ is in single-factor form with respect to $Z=\shadj(Y,y_i,B_i,b_i)$. We check the four criterion in Remark \ref{verticalinfconcatremark} to verify that $f_{\infty}$ is continuous and that the infinite vertical concatenation $H_{\infty}$ from $f=f_1$ to $f_{\infty}$ is continuous. Note that if $C=\emptyset$, $f_{\infty}=f_{p}$ and $H_p$ is the constant homotopy for sufficiently large $p$ and this is trivial. Suppose $C\neq \emptyset$.

\begin{enumerate}
\item By construction, for every $p\in\bbn$, $H_p$ is the constant homotopy on $I^n\backslash \int(C_{p})$,
\item By construction, if $\bfx\in I^n\backslash  C$, then $f_{\infty}(\bfx)$ is the value of the eventually constant sequence $\{f_p(\bfx)\}_{p\in\bbn}$,
\item To check that $f_{\infty}|_{C}:C\to Z$ is continuous, pick $\bfx\in C$ and let $U$ be an open neighborhood of $f_{\infty}(\bfx)=e_{\bfs(\bfx)}$ in $Z$ where, as before, $\bfs(\bfx)$ is the sequence $X_{q_1}\leq X_{q_2}\leq X_{q_3}\leq \cdots $ in $\chi$. By Lemma \ref{opensetend}, we may find $m$ such that $Z_{q_m}\subseteq U$. Now $C\cap \int(R_{q_{m}})$ is an open neighborhood of $\bfx$ in $C$ and, since $f_{\infty}(R_{q_m})\subseteq Z_{q_m}$, we have $f_{\infty}(C\cap \int(R_{q_{P}}))\subseteq U$.
\item Suppose $\bfx_p\in C_p$ and $\{\bfx_p\}\to\bfx$. Let $U$ be an open neighborhood of $f_{\infty}(\bfx)$. We show that $H_p(\{\bfx_p\}\times I)\subseteq U$ all but finitely many $p$. Since $\bfx\in C$, we have $f_{\infty}(\bfx)=e_{\bfs(\bfx)}$ corresponding the sequence $\bfs(\bfx)=X_{q_1}\leq X_{q_2}\leq X_{q_3}\leq \cdots $. As in Criterion 3., find $m\in\bbn$ such that $f_{\infty}(R_{q_m})\subseteq Z_{q_m}\subseteq U$. Since the $n$-cubes $R_{q_p}$ are nested, it follows that $\bfx_p\in C_{q_m}$ for all $p\geq q_m$. Additionally, $H_p(R_{q_m}\times I)\subseteq Z_{q_m}\subseteq U$ for all $p\geq q_m$. Thus $H_p(\{\bfx_p\}\times I)\subseteq U$ for all $p\geq q_m$.
\end{enumerate}
\end{proof}

By replacing Theorem \ref{arctheorem} with its dendrite analogue (Theorem \ref{dendritetheorem}), the proof of Corollary \ref{arcorollary} gives the following conclusion to this section.

\begin{corollary}\label{dendritecorollary}
Let $Z=\shadj(Y,y_i,B_i,b_i)$ where $Y$ is a dendrite and $(B_i,b_i)$ is sequentially $(n-1)$-connected for every $i\in\bbn$. Let $z_0\in Y$. Then every $f\in\Omega^n(Y,y_0)$ is homotopic rel. $\partial I^n$ to some $n$-loop $g\in\Omega^n(Y,y_0)$ that is in single-factor form. Moreover,
\begin{itemize}
\item the canonical homomorphism $\phi:\pi_n(Z,z_0)\to \prod_{i\in\bbn}\pi_n(B_i,b_i)$ is an isomorphism,
\item the quotient map $Z\to \sw_{i\in\bbn}B_i$ collapsing $Y$ to a point induces an isomorphism on $\pi_n$.
\end{itemize}
\end{corollary}

\begin{example}
Let $Z=\shadj(Y,y_i,B_i,b_i)$ where $Y$ is any dendrite and $B_i=S^n$ is a $n$-sphere for $i\in \bbn$ (recall Figures \ref{fig1} and \ref{fig7}). Then $Z$ is $(n-1)$-connected and $\pi_n(Z)\cong \prod_{i\in\bbn}\pi_n(S^n)\cong \bbz^{\bbn}$. Recalling Theorem \ref{splittheorem}, the groups $\pi_k(Z)$, $k>n$ will be non-trivial whenever $\pi_k(S^n)$ is non-trivial. If we allow for the core $Y$ to be a non-compact, simply connected, one-dimensional Hausdorff space, then it is possible that some subsequences of $\{y_i\}$ will have no convergence subsequences. For an dendrite $D\subseteq Y$, note that $Z(D)=D\cup \bigcup\{B_i\mid y_i\in D\}$ is either a finite or shrinking adjunction space and so $\pi_n(Z(D))$ is either trivial or isomorphic to $\bbz^n$ or $\bbz^{\bbn}$. Now $\pi_n(Z)\cong\varinjlim_{D}\pi_n(Z(D))$ where the directed limit is directed by all dendrites in $Y$ containing $z_0$.
\end{example}

\section{Generalized Covering Spaces and Higher Homotopy Groups}\label{sectiongencovers}

\subsection{Generalized covering space theory}

In this section, we analyze a particular case of generalized universal covering spaces introduced in \cite{FZ07}. The more general notion that includes intermediate coverings was first given in \cite{Brazcat} under the name ``$\mathbf{lpc_0}$-covering."

\begin{definition}
A map $p:\tX\to X$ is a \textbf{generalized covering map} if $\tX$ is non-empty, path connected, and locally path connected and if for any map $f:(Y,y)\to (X,x)$ from a path-connected, locally path-connected space $Y$ and point $\wt{x}\in p^{-1}(x)$ such that $f_{\#}(\pi_1(Y,y))\leq p_{\#}(\pi_1(\tX,\wt{x}))$, there is a unique map $\wt{f}:(Y,y)\to (\tX,\wt{x})$ such that $p\circ \wt{f}=f$. If $\tX$ is simply connected, we call $p$ a \textit{generalized universal covering map} and $\tX$ a \textit{generalized universal covering space}.
\end{definition}

Note that a generalized universal covering space over $X$, if it exists, is unique up to homeomorphism. If $X$ is path connected, locally path connected, and semilocally simply connected, then the existence and uniqueness of ordinary covering maps guarantees that any generalized covering map $p:\tX\to X$ is a covering map. In general, a space $X$ need not admit a generalized universal covering space. In \cite{BFTestMap} necessary and sufficient conditions for the existence of a generalized universal covering map are given for metrizable $X$.

\begin{theorem} \cite{FZ07}\label{onedgenunivcovspace}
Every path-connected, one-dimensional Hausdorff space $X$ admits a generalized universal covering space $\tX$. If $X$ is metrizable, then $\tX$ is a topological $\bbr$-tree, i.e. a uniquely arc-wise connected, locally arc-wise connected metrizable space.
\end{theorem}

Since shrinking adjunction spaces of interest will generally not be one-dimensional, we will also employ the following generalization, which implies part of the previous theorem. Recall that a group is residually free if it homomorphically embeds into a product of free groups. Fundamental groups of path-connected, one-dimensional Hausdorff spaces are residually free (See \cite{CConedim} or \cite{EdaSpatial}).

\begin{theorem}\cite[Prop. 6.4 and Thm. 6.9]{BFTestMap}\label{residuallyfreethm}
Let $X$ be a path-connected metric space. If $\pi_1(X,x_0)$ is residually free or if $X$ is $\pi_1$-residual, then $X$ admits a generalized universal covering space.
\end{theorem}

\begin{remark}[Whisker Topology]\label{whiskertop}
Generalized universal coverings may be constructed as follows. Fix a point $x_0\in X$ and let $\wt{X}$ be the space of path-homotopy classes $[\alpha]$ of paths $\alpha:(\ui,0)\to (X,x_0)$. An open neighborhood of $[\alpha]$ is a set of the form $N([\alpha],U)=\{[\alpha\cdot\epsilon]\mid \epsilon(\ui)\subseteq U\}$ where $U$ is an open neighborhood of $\alpha(1)$ in $X$. This topology is the so-called \textit{standard topology} or \textit{whisker topology} on $\wt{X}$. The homotopy class of the constant path at $x_0$, which we denote as $\wt{x}_0$, is the basepoint of $\wt{X}$. The endpoint projection map $p:\tX\to X$, $p([\alpha])=\alpha(1)$ is a continuous surjection, which is open if and only if $X$ is locally path connected. 

According to \cite[Prop. 2.14]{FZ07}, the map $p:\tX\to X$ constructed in the previous paragraph is a generalized universal covering over $X$ if and only if $p$ has the unique path-lifting property. Moreover, it is shown in \cite[Section 5]{Brazcat} that if there exists a generalized universal covering $q:\hX\to X$, then there exists a homeomorphism $h:\tX\to \hX$ such that $q\circ h=p$. Therefore, whenever $X$ admits a generalized universal covering map, it is equivalent to $p:\tX\to X$ and so we always assume generalized universal covering spaces are of the form $\tX$.
\end{remark}

\begin{remark}\label{fiberbijectionremark}
Since we assume that a generalized universal covering space $\tX$ has the structure as described in Remark \ref{whiskertop}, each fiber $p^{-1}(x)$, $x\in X$ is in bijective correspondence with $\pionex$. We have $p^{-1}(x_0)=\pionex$ but if $x\neq x_0$, this bijection requires a choice. Namely, a choice of path $\beta:\ui\to X$ from $x$ to $x_0$ determines the bijection $p^{-1}(x)\to \pionex$, $[\alpha]\mapsto [\alpha\cdot\beta]$.
\end{remark}

\begin{definition}[A topology on the fundamental group]
We give the fundamental group $p^{-1}(x_0)=\pionex$ the subspace topology inherited from $\tX$. In particular, a neighborhood of $[\alpha]\in \pionex$ is a set of the form $N([\alpha],U)=\{[\alpha\cdot\beta]\mid \beta(\ui)\subseteq U\}$ where $U$ is an open neighborhood of $x_0$ in $X$.
\end{definition}

\begin{remark}[A Formula for Lifts of Paths]\label{standardliftremark}
Suppose $[\beta]\in\tX$ and $\alpha:(I,0)\to (X,\beta(1))$ is a path. Define paths $\alpha_s(t)=\alpha(st)$, $s\in\ui$. Then the unique lift $\wt{\alpha}:(I,0)\to (\wt{X},[\beta])$ of $\alpha$ starting at $[\beta]$ is given by $\wt{\alpha}(s)=[\beta\cdot\alpha_s]$.
\end{remark}

Since generalized covering maps have the same lifting property as ordinary covering maps, many results related to covering maps also follow for generalized covering maps. Since $S^n$, $\bbh_n$ (for $n\geq 2$), and the reduced cone over $(\bbh_n,b_0)$ (for $n\geq 1$) are locally path connected and simply connected, the unique lifting property gives the following proposition.

\begin{proposition}\label{coveringisomorphism}
If $p:\tX\to X$ is a generalized covering map, then the induced homomorphisms $p_{\#}:\pi_n(\tX,\wt{x}_0)\to \pi_n(X,x_0)$ and $p_{\#}:[(\bbh_n,b_0),(\tX,\wt{x}_0)]\to [(\bbh_n,b_0),(X,x_0)]$ are injections for $n=1$ and isomorphisms for all $n\geq 2$.
\end{proposition}

Analogous to uses of covering spaces in traditional algebraic topology, if a path-connected space $X$ admits a generalized universal covering space, then to analyze $\pi_n(X,x_0)$, $n\geq 2$, it is sufficient to understand the higher homotopy groups of the simply connected space $\wt{X}$.

In the next lemma, we show how the sequential properties introduced in Section \ref{subsectionsequentialprops} are preserved by generalized universal covering maps. We encourage the reader to recall their definitions in terms of the groups $[(\bbh_k,b_0),(X,x)]$ and the homomorphism $\Theta_k:[(\bbh_k,b_0),(X,x)]\to \pi_k(X,x)^{\bbn}$.

\begin{lemma}\label{relativelytriviallemma}
If $p:\tX\to X$ is a generalized universal covering map, then
\begin{enumerate}
\item $X$ is $\pi_1$-residual at $x\in X$ if and only if $\tX$ is sequentially $1$-connected at every point in $p^{-1}(x)$,
\item for $n\geq 2$, $X$ is $\pi_n$-residual (resp. sequentially $n$-connected, $n$-tame) at $x$ if and only if $\tX$ is $\pi_n$-residual (resp. sequentially $n$-connected, $n$-tame) at every point in $p^{-1}(x)$.
\end{enumerate}
\end{lemma}

\begin{proof}
Fix points $p(\tx)=x$ and note that $\tX$ is simply connected. For (1), suppose $X$ is $\pi_1$-residual at $x$ and $\wt{f}:(\bbh_1,b_0)\to (\tX,\tx)$ is a map. Since $\tX$ is simply connected, $f=p\circ \wt{f}:(\bbh_1,b_0)\to (\tX,\tx)$ is a map for which $\Theta([f])=1$. This $[f]=1$. Let $H:C_{\ast}\bbh_1\to X$ be an extension of $f$ to the reduced cone over $\bbh_1$. Since $C_{\ast}\bbh_1$ is contractible and locally path-connected, there exists a lift $\wt{H}:C_{\ast}\bbh_1\to \tX$, which by unique lifting properties of $p$ must be an extension of $\wt{f}$. Therefore $[\wt{f}]=1$. Thus $\tX$ is sequentially $n$-connected. For the converse, suppose $\tX$ is sequentially $n$-connected at $\tx$ and $f:(\bbh_1,b_0)\to (X,x)$ is a map for which $\Theta_1([f])=1$. Since each loop $\alpha_j=f\circ \ell_j:S^1\to X$ is null-homotopic, there is a lift $\wt{\alpha}_j:S^1\to \wt{X}$ based at $\tx$. Since $\tX$ has the Whisker topology and $\{\alpha_j\}_{j\in\bbn}$ converges to $x$, it follows that every basic neighborhood of $\tx$ contains $\im(\wt{\alpha}_j)$ for all but finitely many $j$. Therefore, there is a map $\wt{f}:(\bbh_1,b_0)\to (\tX,\tx)$ such that $\wt{f}\circ\ell_j=\wt{\alpha}_j$. By assumption $[\wt{f}]=1$ and so $\wt{f}$ is null-homotopic rel. basepoint in $\wt{X}$. It follows that $f=p\circ\wt{f}$ is homotopic rel. basepoint, i.e. $[f]=1$.

For (2), consider the following commutative square. By Proposition \ref{coveringisomorphism}, the vertical maps are isomorphisms for any $n\geq 2$.
\[\xymatrix{
[(\bbh_n,b_0),(\wt{X},\wt{x})] \ar[rr]^-{\Theta_n} \ar[d]_-{p_{\#}} && \prod_{j\in\bbn}\pi_n(\tX,\wt{x}) \ar[d]^-{\prod_{j}p_{\#}}\\
[(\bbh_n,b_0),(X,x)] \ar[rr]_-{\Theta_n} && \prod_{j\in\bbn}\pi_n(X,x)
}\]
Using this diagram, all cases of (2) follow from the definitions of the respective properties in terms of $\Theta_n$.
\end{proof}

\subsection{Generalized coverings of shrinking adjunction spaces}\label{coveringsubsection}

Here, we begin to develop a theory of generalized coverings of shrinking adjunction spaces. The conclusions in the following lemma are related to the theory of pullbacks of generalized universal coverings developed by H. Fischer \cite{Fischerpullback}. We include the short proofs since the overlap with Fischer's work is only partial and since our applications are fairly specific.

\begin{lemma}\label{subcoveringlemma1}
Suppose $p:\wt{Y}\to Y$ is a generalized covering map, $X\subseteq Y$ is a path-connected subspace, and $y_0\in X$.
\begin{enumerate}
\item If the inclusion $i:X\to Y$ is $\pi_1$-surjective, then $p^{-1}(X)$ is path connected.
\item If for every $x\in X$ and neighborhood $U$ of $x$ in $Y$, there exists an open set $V$ of $x$ contained in $U$ such that every path in $V$ with endpoints in $V\cap X$ is path-homotopic in $Y$ to a path in $V\cap X$, then $p^{-1}(X)$ is locally path connected.
\item If the hypotheses of (1) and (2) hold, then the restriction $q:p^{-1}(X)\to X$ of $p$ is a generalized covering map.
\item If $p$ is a generalized universal covering map, $i:X\to Y$ induces an isomorphism on $\pi_1$, and the hypothesis of (2) holds, then the restriction $q:p^{-1}(X)\to X$ of $p$ is a generalized universal covering map.
\end{enumerate}
\end{lemma}

\begin{proof}
According to Remark \ref{whiskertop}, we may assume $\wt{Y}$ has the whisker topology construction on basepoint $y_0\in Y$. 

For Part (1), notice that the basepoint $\wt{y}_0=[c_{y_0}]$ of $\wt{Y}$ lies in $p^{-1}(X)$. Suppose $[\alpha]\in p^{-1}(X)$ for some path $\alpha:\ui\to Y$ from $y_0$ to $\alpha(1)\in X$. Since $i_{\#}:\pi_1(X,y_0)\to \pi_1(Y,y_0)$ is surjective, $\alpha$ is path-homotopic in $Y$ to a path $\beta:\ui\to X$. The standard lift $\wt{\beta}_s$ is now a path in $p^{-1}(X)$ from $\wt{y}_0$ to $[\alpha]=[\beta]$ with image in $p^{-1}(X)$. Thus $p^{-1}(X)$ is path connected.

For (2), fix $[\alpha]\in p^{-1}(X)$ and a basic open set $N([\alpha],U)\cap p^{-1}(X)$ in $p^{-1}(X)$. Find an open neighborhood $V$ of $\alpha(1)$ in $Y$ contained in $U$ as described in the hypothesis of Part (2). Certainly, $N([\alpha],V)\cap p^{-1}(X)\subseteq N([\alpha],U)\cap p^{-1}(X)$ and so it suffices to check that $N([\alpha],V)\cap p^{-1}(X)$ is path connected. Let $[\beta]\in N([\alpha],V)\cap p^{-1}(X)$ for some path $\beta:\ui\to Y$ from $y_0$ to a point $\beta(1)\in X$. Then $\beta\simeq \alpha\cdot\gamma$ in $Y$ for some path $\gamma:\ui\to V$. Since $\alpha(1)$ and $\beta(1)$ lie in $X$, $\gamma$ has endpoints in $X$. By assumption, there is a path $\delta:\ui\to V\cap X$ that is path-homotopic in $Y$ to $\gamma$. Let $\wt{\delta}_s:\ui\to \wt{Y}$ be the standard lift of $\delta$ starting at $[\beta]$. Then $\wt{\delta}_s$ is a path in $N([\alpha],V)\cap p^{-1}(X)$ from $[\alpha]$ to $[\beta]$. This proves $N([\alpha],V)\cap p^{-1}(X)$ is path connected.

For (3), suppose the hypotheses of (1) and (2) hold. Based on the previous two parts, it suffices to check that $q=p|_{p^{-1}(X)}$ has the required unique lifting property. Suppose $(Z,z_0)$ is path connected and locally path-connected and $f:(Z,z_0)\to (X,y_0)$ is a map such that $f_{\#}(\pi_1(Z,z_0))\leq q_{\#}(\pi_1(p^{-1}(X),\wt{y}_0))$. If $j:p^{-1}(X)\to Y$ is the inclusion, then 
\begin{eqnarray*}
(i\circ f)_{\#}(\pi_1(Z,z_0)) &\leq& (i\circ q)_{\#}(\pi_1(p^{-1}(X),\wt{y}_0))\\
&=& (p\circ j)_{\#}(\pi_1(p^{-1}(X),\wt{y}_0))\\
&\leq&  p_{\#}(\pi_1(\wt{Y},\wt{y}_0)).
\end{eqnarray*}
Hence, there exists a unique map $\wt{g}:(Z,z_0)\to (\wt{Y},\wt{y}_0)$ such that $p\circ \wt{g}=i\circ f$. Since $p(\wt{g}(Z))=i(f(Z))\subseteq X$, we have $\wt{g}(Z)\subseteq p^{-1}(X)$. Thus $\wt{g}$ is a unique map satisfying $q\circ \wt{g}=g$. We conclude that $q$ is a generalized covering map. 

Finally, suppose the hypotheses of (4) holds. Since the conclusions of Parts (1)-(3) hold, it suffices to show that $p^{-1}(X)$ is simply connected. Let $\wt{\alpha}:\ui\to p^{-1}(X)$ be a loop based at $\wt{y}_0$. Since $\wt{Y}$ is simply connected, $\alpha=p\circ\wt{\alpha}$ is a loop in $X$ that is null-homotopic in $Y$. However, since the inclusion $i:X\to Y$ is $\pi_1$-injective, $\alpha$ must also be null-homotopic in $X$. By the unique lifting properties of $q$, a null-homotopy of $\alpha$ in $X$ lifts to a null-homotopy of $\wt{\alpha}$ in $p^{-1}(X)$.
\end{proof}

\begin{corollary}\label{subcoveringcorollary}
Suppose $X\subseteq Y$ is a retract such that the inclusion $i:X\to Y$ induces an isomorphism on $\pi_1$. Then $X$ admits a generalized universal covering if and only if $Y$ admits a generalized universal covering. Moreover, if $p:\wt{Y}\to Y$ is a generalized universal covering map, then the restriction $p|_{p^{-1}(X)}:p^{-1}(X)\to X$ is a generalized universal covering map.
\end{corollary}

\begin{proof}
First, suppose that $Y$ admits a generalized universal covering $p:\wt{Y}\to Y$. We prove that $X$ admits a generalized universal covering by proving the second statement. The second statement will follow from Part (4) of Proposition \ref{subcoveringlemma1} once the hypotheses are verified. The only non-trivial hypothesis to check is that the hypothesis of Part (2) holds. Suppose $x\in X$ and $V$ is a neighborhood of $x$ in $Y$. Set $U=V$. Let $r:Y\to X$ be a retraction for which $i:X\to Y$ is a section. Suppose $\alpha:\ui\to U$ is any path with endpoints in $U\cap X$. Now $r\circ \alpha:\ui \to U\cap X$ is a path with the same endpoints. Moreover, since $r_{\#}:\pi_1(Y,u_0)\to \pi_1(X,u_0)$ is an injective and $r_{\#}([\alpha])=r_{\#}([r\circ\alpha])$, we have $\alpha\simeq r\circ\alpha$ in $Y$. Thus $\alpha$ is path-homotopic in $Y$ to a path in $U\cap X$. This verifies the hypothesis of (2) and the result follows.

For the other direction, suppose that $X$ admits a generalized universal covering $q:\tX\to X$. It follows from \cite[Lemma 2.34]{Brazcat} if $f:(Z,z_0)\to (X,y_0)$ is any map, then there exists a generalized covering map $p:\wt{Z}\to Z$ such that $p_{\#}(\pi_1(\wt{Z},\wt{z}))=f_{\#}^{-1}(q_{\#}(\pi_1(\tX,\wt{x})))$. We note that $\wt{Z}$ may be the locally path-connected coreflection of a path component of the usual pullback of $q$ by $f$. Applying this fact to $q$ and $r:Y\to X$, we obtain a generalized covering map $p:\wt{Y}\to Y$ such that $p_{\#}(\pi_1(\wt{Y},\wt{y}_0))=r_{\#}(q_{\#}(\pi_1(\tX,\wt{x})))$. Since $q$ is a generalized universal covering map and $r_{\#}$ is an isomorphism, $p_{\#}(\pi_1(\wt{Y},\wt{y}_0))=1$. Since any generalized covering map induces an injection on $\pi_1$, we conclude that $\wt{Y}$ is simply connected, i.e. $p$ is a generalized \textit{universal} covering map.
\end{proof}

Although Corollary \ref{subcoveringcorollary} provides no new information about fundamental groups, it will be useful for our approach to higher homotopy groups. Its greatest utility includes cases where $Y$ is obtained from $X$ by attaching higher dimensional spaces to $X$ in a way that does not affect the fundamental group. We apply the last few results to the case of shrinking adjunction spaces.


\begin{lemma}\label{structurelemma}
Let $Y=\shadj(X,x_j,A_j,a_j)$ where $X$ is path connected and each $(A_j,a_j)$ is sequentially $1$-connected. Let $y_0\in X$. If $X$ admits a generalized universal covering map $q:\tX\to X$, then
\begin{enumerate}
\item the inclusion $i:X\to Y$ and retraction $r:Y\to X$ induce inverse isomorphisms on $\pi_1$.
\item $Y$ admits a generalized universal covering map $p:\wt{Y}\to Y$ such that $\wt{X}$ is a retract of $\wt{Y}$ and $p|_{\wt{X}}=q$.
\item $\wt{Y}\backslash \wt{X}$ is the disjoint union of open subspaces \[N_{j,[\alpha]}=\{[\alpha\cdot\epsilon]\in\wt{Y}\mid \epsilon((0,1])\subseteq A_j\backslash\{a_j\}\}\]
for $j\in\bbn$ and paths $\alpha:(\ui,0,1)\to (X,y_0,x_j)$,

\item $\mathcal{A}_{j,[\alpha]}$ is closed in $\wt{Y}$,
\item If $A_j$ is locally path connected, then $p$ maps $\mathcal{A}_{j,[\alpha]}=N_{j,[\alpha]}\cup\{[\alpha]\}$ homeomorphically onto $A_j$.
\end{enumerate}
\end{lemma}

\begin{proof}
Part (1) is a special case of Corollary \ref{retractioniso}. With Part (1) established, Corollary \ref{subcoveringcorollary} gives the existence of a generalized universal covering map $p:\wt{Y}\to Y$ whose restriction to $p^{-1}(X)$ is a generalized universal covering map over $X$. The map $r\circ p:\wt{Y} \to X$ induces a map $\wt{r}:\wt{Y}\to \wt{X}$ such that $q\circ\wt{r}=r\circ p$. Consider the diagram below where $i$ and $\wt{i}$ are inclusions. 
\[\xymatrix{
\tX \ar[r]^-{\wt{i}} \ar[d]_-{q} & \wt{Y} \ar@{-->}[r]^-{\wt{r}} \ar[d]^-{p} & \tX \ar[d]^-{q}\\
X \ar[r]_-{i} & Y \ar[r]_-{r} & X
}\]Since $q\circ (\wt{r}\circ \wt{i})=q$, it follows that $\wt{r}\circ \wt{i}=id_{\tX}$. Hence $\wt{r}$ is a retraction, completing the proof of (2).

For (3), since $X$ is closed in $Y$, it follows that $\wt{X}=p^{-1}(X)$ is closed as a subspace of $\wt{Y}$. Certainly, $N_{j,[\alpha]}\subseteq \wt{Y}\backslash\tX$ is clear for any $j\in\bbn$ and path $\alpha:(I,0,1)\to (X,y_0,x_j)$. Suppose $\wt{y}\in \wt{Y}\backslash \tX$. Then $p(\wt{y})\in A_j\backslash \{a_j\}$ for some $j\in \bbn$. Find a path $\wt{\gamma}:\ui\to \wt{Y}$ from the basepoint $\wt{y}_0$ to $\wt{y}$ and notice that $\gamma=p\circ\wt{\gamma}$ is a path in $Y$ from $y_0$ to $p(\wt{y})$. Let $t_0=\sup\{t\in\ui\mid \gamma(t)\in X\}$. Then $\gamma(t_0)=x_j$ and $\gamma((t_0,1])\subseteq A_j\backslash \{a_j\}$. Since the inclusion $i:X\to Y$ induces an isomorphism on $\pi_1$, we may assume that $\gamma([0,t_0])\subseteq X$ without changing the path-homotopy class of $\gamma|_{[0,t_0]}$. Let $\alpha=\gamma|_{[0,t_0]}$ and $\epsilon=\gamma|_{[t_0,1]}$. Now $\wt{y}=\wt{\gamma}(1)=[\alpha\cdot\epsilon]$ where $\alpha:\ui\to X$ is path from $y_0$ to $x_j$ and $\epsilon((0,1])\subseteq A_j\backslash\{a_j\}$. Thus $\wt{Y}\backslash \tX= \bigcup\{N_{j,[\alpha]}\mid j\in\bbn,\alpha:(I,0,1)\to (X,y_0,x_j)\}$.

Fix $j\in\bbn$ and a path $\alpha:(I,0,1)\to (X,y_0,x_j)$. For any path $\epsilon:(I,0)\to (A_j,a_j)$ with $\epsilon(1)\subseteq A_j\backslash \{a_j\}$, the set $B([\alpha\cdot\epsilon],A_{j}\backslash\{a_j\})$ is an open neighborhood of $[\alpha\cdot\epsilon]$ in $\wt{Y}$. Therefore,\[N_{j,[\alpha]}=\bigcup\{B([\alpha\cdot\epsilon],A_{j}\backslash\{a_j\})\mid \epsilon:(I,0)\to (A_j,a_j), \epsilon(1)\in A_{j}\backslash \{a_j\}\}\]is open in $\wt{Y}$.

Finally, suppose $\wt{y}\in N_{j,[\alpha]}\cap N_{j',[\beta]}$. Then $\wt{y}=[\alpha\cdot\epsilon]=[\beta\cdot\delta]$ for paths $\epsilon:(\ui,0)\to (A_j,a_j)$ with $\epsilon(1)\in A_{j}\backslash\{a_j\}$ and $\delta:(\ui,0)\to (A_{j'},a_{j'})$ with $\epsilon(1)\in A_{j'}\backslash\{a_{j'}\}$. Since $\epsilon(1)=\delta(1)$, we must have $j=j'$. Thus $\alpha(1)=\beta(1)=x_j$. Applying the retraction $r_j:Y\to A_j$ to the equality $[\alpha\cdot\epsilon]=[\beta\cdot\delta]$ gives $[\epsilon]=[\delta]$ as path-homotopy classes in $A_j$. It follows that $[\alpha]=[\beta]$ as path-homotopy classes in $X$. We conclude that $\wt{Y}\backslash\tX$ is the disjoint union of the open sets $N_{j,[\alpha]}$ ranging over pairs $(j,[\alpha])$ where $j\in\bbn$ and $[\alpha]$ is a path-homotopy class of a path $\alpha:(I,0,1)\to (X,y_0,x_j)$.

For (4), $\wt{X}\cap \mathcal{A}_{j,[\alpha]}=\{[\alpha]\}$ is clear since $\tX\cap N_{j,[\alpha]}=\emptyset$. To check that $\mathcal{A}_{j,[\alpha]}$ is closed in $\wt{Y}$, Part (3) makes clear that we only need to check that $\mathcal{A}_{j,[\alpha]}$ is closed in $\wt{X}\cup \mathcal{A}_{j,[\alpha]}$. Suppose $[\beta]\in\tX\backslash\{[\alpha]\}$. If $\beta(1)\neq \alpha(1)=x_j$, then $\beta(1)$ lies in the open set $Y\backslash A_j$ of $Y$. Now $U=B([\beta],Y\backslash A_j)\cap (\tX\cup \mathcal{A}_{j,[\alpha]})$ is an open neighborhood of $[\beta]$ in $\wt{X}\cup \mathcal{A}_{j,[\alpha]}$. If $[\beta\cdot\epsilon]\in U$ with $\epsilon:(I,0)\to (Y\backslash A_j,\beta(1))$, then $\epsilon(1)\notin A_j$ and so $[\beta\cdot\epsilon]\notin \mathcal{A}_{j,[\alpha]}$. Therefore, $U\cap \mathcal{A}_{j,[\alpha]}=\emptyset$. On the other hand, suppose $\beta(1)=\alpha(1)=x_j$. Then $[\alpha]$ and $[\beta]$ are distinct points in the fiber $q_j^{-1}(x_j)=p^{-1}(x_j)$. Since the fibers of a generalized universal covering spaces are Hausdorff, we may find an open neighborhood $W$ of $x_j$ in $Y$ such that $[\alpha]\notin B([\beta],W)$. Now $U=B([\beta],W)\cap (\tX\cup \mathcal{A}_{j,[\alpha]})$ is an open neighborhood of $[\beta]$ in $\tX\cup \mathcal{A}_{j,[\alpha]}$. It suffices to show that $U\cap \mathcal{A}_{j,[\alpha]}=\emptyset$. Suppose otherwise. Then we have paths $\epsilon:(I,0)\to (A_j,a_j)$ and $\delta:(I,0)\to (W,x_j)$ such that $[\alpha\cdot\epsilon]=[\beta\cdot\delta]\in B([\beta],W)\cap \mathcal{A}_{j,[\alpha]}$. If $\epsilon(1)=a_j$, then, since $A_j$ is simply connected, we have $[\alpha]=[\beta\cdot\delta]\in B([\beta],W)$ a contradiction. If $\alpha(1)\in A_j\backslash \{a_j\}$, then we may find $t_0=\max\{t\in \ui\mid \delta(t)=a_j\}$. Since $A_j$ is simply connected, we have $[\epsilon]=[\delta|_{[t_0,1]}]$. Thus $[\alpha]=[\beta\cdot\delta|_{[0,t_0]}]\in B([\beta],W)$, which is another contradiction. We conclude that $U\cap \mathcal{A}_{j,[\alpha]}=\emptyset$.

For Part (5), notice that $p$ maps $[\alpha]\in\mca_{j,[\alpha]}$ to $x_j$ and all points of $N_{j,[\alpha]}$ into $A_j\backslash\{a_j\}$. Suppose $[\alpha\cdot\epsilon]$ and $[\alpha\cdot\delta]$ are elements of $N_{j,[\alpha]}$ with $p([\alpha\cdot\epsilon])=\epsilon(1)=\delta(1)=p([\alpha\cdot\delta])$. Now $\epsilon\cdot\delta^{-}$ is a loop in $A_j$ based at $a_j$. However, $A_j$ is simply connected. Thus, $[\epsilon]=[\delta]$, giving $[\alpha\cdot\epsilon]=[\alpha\cdot\delta]$. We conclude that $p$ maps $\mathcal{A}_{j,[\alpha]}$ bijectively onto $A_j$. If $\gamma:(\ui,0)\to (A_j,a_j)$ is any path and $V$ is a path-connected open neighborhood of $\gamma(1)$, then $p$ maps the open set $B([\alpha\cdot\gamma],V)\cap \mathcal{A}_{j,[\alpha]}$ in the subspace $\mathcal{A}_{j,[\alpha]}$ of $\wt{Y}$ onto $V$. Hence, $p$ is a homeomorphism.
\end{proof}

Although $\wt{Y}$ will be a shrinking adjunction space only in trivial cases, Lemma \ref{structurelemma} shows that $\wt{Y}$ is at least an indeterminate adjunction space (Recall Section \ref{subsectionintermediateadjunctionspaces}) with core $\wt{X}$ and attachment spaces $\mca_{j,[\alpha]}$, $[\alpha]\in p^{-1}(x_j)$.

\begin{corollary}\label{indeterminatecorollary}
Let $Y=\shadj(X,x_j,A_j,a_j)$ where $X$ is path connected and each $(A_j,a_j)$ is sequentially $1$-connected. Let $y_0\in X$. If $X$ admits a generalized universal covering map $q:\tX\to X$, then, $\wt{Y}$ is an indeterminate adjunction space with core $\tX$, attachment spaces $\mathcal{A}_{j,[\alpha]}$ and attachment points $[\alpha]\in\tX$ ranging over all path-homotopy classes of paths $\alpha:I\to X$ from $y_0$ to the attachment points $x_j\in X$.
\end{corollary}

\begin{remark}
The condition that each $A_j$ is locally path connected in (5) of Lemma \ref{structurelemma} is only needed if one wants to ensure that $p$ maps $\mathcal{A}_{j,[\alpha]}$ to $A_j$ by a homeomorphism. If $A_j$ is not locally path connected, then $\mathcal{A}_{j,[\alpha]}$ is canonically homeomorphic to the locally path-connected coreflection $lpc(A_j)$ and the map $p:\mathcal{A}_{j,[\alpha]}\to A_j$ corresponds to the continuous identity map $lpc(A_j)\to A_j$. Hence, $p$ maps $\mathcal{A}_{j,[\alpha]}$ to $A_j$ by a bijective weak homotopy equivalence.
\end{remark}

\begin{lemma}
If $n\geq 2$ and $A_j$ is sequentially $n$-connected at $a_j$ for all $j\in \bbn$, then the retraction $\wt{r}:\wt{Y}\to\tX$ induces an isomorphism $\wt{r}_{\#}:\pi_m(\wt{Y},\wt{y}_0)\to \pi_m(\tX,\wt{y}_0)$ for all $2\leq m\leq {n}$. 
\end{lemma}

\begin{proof}
For all $m\geq 1$, the following square commutes. 
\[\xymatrix{
\pi_m(\wt{Y},\wt{y}_0) \ar[r]^-{\wt{r}_{\#}} \ar[d]_-{p} & \pi_m(\tX,\wt{y}_0) \ar[d]^-{q}\\
\pi_m(Y,y_0) \ar[r]_-{r_{\#}} & \pi_m(X,y_0)
}\]
When $m\geq 2$, the vertical maps are isomorphisms by Theorem \ref{coveringisomorphism}. According to Theorem \ref{retractioniso}, since each $(A_j,a_j)$ is sequentially $n$-connected, $r_{\#}$ is an isomorphism for all $1\leq m\leq n$. Therefore, the top map is an isomorphism when $2\leq m\leq n$.
\end{proof}

Recall that $q^{-1}(x_j)\subseteq \tX$ is a convenient way to denote the set of path-homotopy classes in $X$ from $y_0$ to $x_j$. For each pair $(j,[\alpha])$ with $[\alpha]\in q^{-1}(x_j)$, let $r_{j,[\alpha]}:\wt{Y}\to \mathcal{A}_{j,[\alpha]}$ be the canonical retraction.

\begin{lemma}\label{countablelemma}
Let $X\subseteq Y$ and $\tX\subseteq \wt{Y}$ be as in Lemma \ref{structurelemma}. Let $Z\subseteq \wt{Y}$ be a subspace, which is a Peano continuum and let $Z_{j,[\alpha]}=  Z\cap \mca_{j,[\alpha]}$. Let $K$ be the set of pairs $(j,[\alpha])$ such that $Z\cap N_{j,[\alpha]}\neq\emptyset$. Then
\begin{enumerate}
\item $K$ is countable,
\item $Z\cap \tX$ and $Z\cap \mca_{j,[\alpha]}$, $(j,[\alpha])\in K$ are Peano continua,
\item If $K$ is infinite, then $Z$ is homeomorphic to the shrinking adjunction space with core $Z\cap \tX$ and attachment spaces $Z_{j,[\alpha]}$, $(j,[\alpha])\in K$.
\end{enumerate}
\end{lemma}

\begin{proof}
Since $Z$ is separable and the subspaces $N_{j,[\alpha]}$ are open and disjoint by (3) of Lemma \ref{structurelemma}, it is only possible for $Z$ to meet $N_{j,[\alpha]}$ for countably many pairs $(j,[\alpha])$. Hence $K$ is countable.

For (2), recall that there is a retraction $\wt{r}:\wt{Y}\to \tX$ such that $\wt{r}(\mathcal{A}_{j,[\alpha]})=[\alpha]$ for all $j\in\bbn$ and $\alpha:(\ui,0,1)\to (X,y_0,x_j)$. Then $\wt{r}(Z)=Z\cap\tX$ is a Peano continuum. Also, for all $(j,[\alpha])\in K$, $\wt{r}_{j,[\alpha]}(Z)=Z\cap \mca_{j,[\alpha]}$ is a Peano continuum.

For (3), recall that Corollary \ref{indeterminatecorollary} gives that $\wt{Y}$ is an indeterminate adjunction space with core $\tX$, attachment spaces $\mathcal{A}_{j,[\alpha]}$. The conclusion of (3) now follows from Proposition \ref{indeterminateproppeano}.
\end{proof}

\subsection{The main result: injectivity of $\Phi$}

Finally, we prove the strongest results of this paper, which imply Theorem \ref{mainthm} in the introduction.

\begin{theorem}\label{maintheorem}
Let $Y=\shadj(X,x_j,A_j,a_j)$ where $X$ is a one-dimensional Peano continuum and each $(A_j,a_j)$ is sequentially $(n-1)$-connected and $\pi_n$-residual. Then
\begin{enumerate}
\item $\pi_m(Y,y_0)=0$ for all $2\leq m\leq n-1$,
\item the canonical homomorphism
\[\wt{\Phi}:\pi_n(\wt{Y},\wt{y}_0)\to \prod_{j\in\bbn}\prod_{[\alpha]\in p^{-1}(x_j)}\pi_n(\mca_{j,[\alpha]},[\alpha])\] induced by the retractions $r_{j,[\alpha]}:(\wt{Y},\wt{y}_0)\to (\mca_{j,[\alpha]},[\alpha])$ is injective.
\end{enumerate}
\end{theorem}

\begin{proof}
By Theorem \ref{retractioniso}, $r_{\#}:\pi_m(Y,y_0)\to\pi_m(X,y_0)$ is an isomorphism for all $1\leq m\leq n-1$. Since $X$ is one-dimensional, $X$ is aspherical \cite{CFhigher}. Hence, $\pi_m(Y,y_0)=0$ for all $2\leq m\leq n-1$.

For (2), suppose $n\geq 2$ and let $\scrf=\{(j,[\alpha])\in \bbn\times \tX\mid \alpha(1)=x_j\}$. Let $\wt{f}:(I^n,\partial I^n)\to (\wt{Y},\wt{y}_0)$ be a map such that $\wt{f}_{j,[\alpha]}=r_{j,[\alpha]}\circ \wt{f}$ is null-homotopic in $\mca_{j,[\alpha]}$ for all $(j,[\alpha])\in\scrf$. We will show that $\wt{f}$ is null-homotopic in $\wt{Y}$. By Lemma \ref{factoredformsequencelemma1}, $f=p\circ \wt{f}$ is homotopic to an $n$-loop $g$ in $Y$ that is in factored form. This homotopy lifts to a homotopy $\wt{f}\simeq \wt{g}$ where $\wt{g}$ is the lift of $g$. Therefore, we may assume from the start that $f$ is in factored form. Choose a Whitney cover $\scrc_j$ of $V_j=f^{-1}(A_j\backslash\{a_j\})$ so that $\{\scrc_j\}_{j\in\bbn}$ forms a factorization of $f$ (where $\scrc_j=\emptyset$ if $V_j=\emptyset$). Since $\im(\wt{f})$, is a Peano continuum, Lemma \ref{countablelemma} ensures that the set $K=\{(j,[\alpha])\in\scrf\mid \im(\wt{f})\cap N_{j,[\alpha]}\neq\emptyset\}$ is countable. Then for all $j\in \bbn$, the set $K_j=\{[\alpha]\in q^{-1}(x_j)\mid (j,[\alpha])\in K\}$ is countable.

Let $U_{j,[\alpha]}=\wt{f}^{-1}(N_{j,[\alpha]})$ and notice that $V_j$ is the disjoint union of the open sets $U_{j,[\alpha]}$, $[\alpha]\in K_j$. Therefore, $\scrc_{j,[\alpha]}=\{C\in\scrc_j\mid C\subseteq U_{j,[\alpha]}\}$ is a Whitney cover of $U_{j,[\alpha]}$ for all $(j,[\alpha])\in K$. Moreover, if $C\in \scrc_{j,[\alpha]}$, then $\wt{f}(C)\subseteq \mca_{j,[\alpha]}$ and $f(\partial C)=a_j$. Therefore, $\wt{f}(\partial C)=[\alpha]$. We conclude that for all $(j,[\alpha])\in K$, $\wt{f}_{j,[\alpha]}$ is a $\scrc_{j,[\alpha]}$-concatenation. By assumption, $\wt{f}_{j,[\alpha]}$ is null-homotopic for all $(j,[\alpha])\in K$ and therefore $f_{j,[\alpha]}=r_j\circ f=p\circ \wt{f}_{j,[\alpha]}$ is null-homotopic in $A_j$ for all $(j,[\alpha])\in K$. 

Let $\mcd=\im(\wt{f})\cap \wt{X}$. Since $\wt{X}$ is a topological $\bbr$-tree, $\mcd$ is a dendrite. For each $(j,[\alpha])\in K$, let $\mcz_{j,[\alpha]}=\im(\wt{f})\cap \mca_{j,[\alpha]}$. Now $\im(\wt{f})=\mcd\cup \bigcup \{\mcz_{j,[\alpha]}\mid (j,[\alpha])\in K\}$ is a Peano continuum in the indeterminate adjunction space $\mcd\cup \bigcup \{\mca_{j,[\alpha]}\mid (j,[\alpha])\in K\}$ with core $\mcd$ and attachments spaces $\mca_{j,[\alpha]}$, $(j,[\alpha])\in K$. When $K$ is finite, $\im(\wt{f})$ is a finite adjunction space. When $K$ is infinite, Proposition \ref{indeterminateproppeano} gives that $\im(\wt{f})$ is a shrinking adjunction space with core $\mcd$ and attachment spaces $\mcz_{j,[\alpha]}$, $(j,[\alpha])\in K$ attached at the corresponding points $[\alpha]\in\mcd$. Either way, $\{\scrc_{j,[\alpha]}\mid (j,[\alpha])\in K\}$ forms a factorization of $\wt{f}$ with respect to given decomposition of $\im(\wt{f})$. Therefore, $\wt{f}$ is in factored form with respect to this decomposition of $\im(\wt{f})$. 

Even though $\mcz_{j,[\alpha]}$ need not be sequentially $(n-1)$-connected at $[\alpha]$, Theorem \ref{dendritetheorem} applies since $\wt{f}$ is in already in factored form. Therefore, $\wt{f}$ is homotopic rel. $\partial I^n$ to an $n$-loop $\wt{g}\in \Omega^n(\im(\wt{f}),\wt{y}_0)$ that is in single factor form with respect to the indicated decomposition of $\im(\wt{f})$. Let $\{\scrr_{j,[\alpha]}\mid (j,[\alpha])\in K\}$ be a factorization of $\wt{g}$ where $\scrr_{j,[\alpha]}$ consists of a single $n$-cube $R_{j,[\alpha]}$ (each $\scrr_{j,[\alpha]}$ cannot be empty since $\im(\wt{f})\cap \mcz_{j,[\alpha]}\backslash\{[\alpha]\}\neq\emptyset$). Let $g=p\circ \wt{g}$. 

Then $\wt{g}_{j,[\alpha]}=\wt{r}_{j,[\alpha]}\circ \wt{g}\simeq \wt{f}_{j,[\alpha]}$ is null-homotopic in $\mca_{j,[\alpha]}$ for all $(j,[\alpha])\in K$. It follows that $g_{j,[\alpha]}=p\circ \wt{g}_{j,[\alpha]}$ is null-homotopic in $A_j$ for all $(j,[\alpha])\in K$.

If $K_j$ is finite, then for each $[\alpha]\in K_j$, choose any null-homotopy $H_{j,[\alpha]}:R_{j,[\alpha]}\times I\to \mca_{j,[\alpha]}$ of $\wt{g}|_{R_{j,[\alpha]}}$, i.e. where $H_{j,[\alpha]}(\bfx,0)=g(\bfx)$ and $H_{j,[\alpha]}(\partial R_{j,[\alpha]}\times I\cup R_{j,[\alpha]}\times \{1\})=[\alpha]$. 

If $K_j$, is infinite, we claim that the set $\{g|_{R_{j,[\alpha]}}\mid [\alpha]\in K_j\}$ clusters at $x_j=a_j$ in $A_j$. Indeed, these maps have image in $A_j$ and if $U$ is an open neighborhood of $a_j$ in $Y$, then $p^{-1}(U)\cap \im(\wt{f})$ is an open neighborhood of the fiber $p^{-1}(x_j)\cap \im(\wt{f})$. Since $K_j\subseteq p^{-1}(x_j)\cap \im(\wt{f})$, $K_j$ lies in the core $\mcd$ of the shrinking adjunction space $\im(\wt{f})$. Recalling Lemma \ref{opensetlemma}, we must have $\wt{g}(R_{j,[\alpha]})\subseteq \mcz_{j,[\alpha]}\subseteq p^{-1}(U)\cap \im(\wt{f})$ for all but finitely many $[\alpha]\in K_j$. Therefore, $g(R_{j,[\alpha]})\subseteq U$ for all but finitely many $[\alpha]\in K_j$. Since the set $\{g|_{R_{j,[\alpha]}}\mid [\alpha]\in K_j\}$ of null-homotopic $n$-loops clusters at $x_j=a_j$ in $A_j$ and since $A_j$ is assumed to be $\pi_n$-residual at $a_j$, we may choose a set $\{H_{j,[\alpha]}\mid [\alpha]\in K_j\}$ of corresponding null-homotopies that cluster at $a_j$, that is, where $H_{j,[\alpha]}:R_{j,[\alpha]}\times I\to A_j$ is a null-homotopy of $g|_{R_{j,[\alpha]}}$.

Although $g$ is not in single factor form in $Y$, the sets $\scrs_j=\{R_{j,[\alpha]}\mid [\alpha]\in K_j\}$ form a factorization of $g$ (define $\scrs_j=\emptyset$ if $K_j=\emptyset$). Define $H:I^n\times I\to Y$ so that $H$ is the constant homotopy on $g^{-1}(X)=\wt{g}^{-1}(\mcd)$ and which agrees with $H_{j,[\alpha]}$ on $R_{j,[\alpha]}\times I$. Certainly, $H$ is well-defined. Since $g$ is not in single factor form, the continuity of $H$ is not completely obvious. However, the only non-trivial situation to consider is a sequence of distinct pairs $(j_i,[\alpha_i])$, $i\in\bbn$ in $K$ and points $(\bfx_i,t_i)\in R_{j_i,[\alpha_i]}\times I$ where $\{(\bfx_i,t_i)\}\to (\bfx,t)$ in $I^n\times I$. We have $H(\bfx_i,t_i)\in A_{j_i}$ and since $(\bfx,t)\in \partial \bigcup\{R_{j_i,[\alpha_i]}\mid (j,[\alpha])\in K\}$, we must have $x=H(\bfx,t)\in X$. Let $U$ be an open neighborhood of $H(\bfx,t)$ in $Y$ and $x_{j_i}=g(\partial R_{j_i,[\alpha_i]})=x_{j_i}$ for each $i\in\bbn$. By the continuity of $g$, we must have $\{x_{j_i}\}\to x$ in $X$ and therefore $x_{j_i}\in U$ for all but finitely many $i$. By replacing $(j_i,[\alpha_i])$ with a cofinal subsequence, if necessary, we may assume that $x_{j_i}\in U$ for all $i\in\bbn$. 

Let $P=\{j\in\bbn\mid x_j=x_{j_i}\text{ for some }i\in\bbn\}$. First, suppose $P$ is finite. If $H(\bfx_i,t_i)\notin U$ for infinitely many $i$, then there is some $j_0\in P$ and $i_1<i_2<i_3<\cdots $ such that $x_{j_0}=x_{j_{i_m}}$ and $H(\bfx_{i_m},t_{i_m})\notin U$. However, we chose the maps $\{H_{j_0,[\alpha]}\mid [\alpha]\in K_{j_0}\}$ to cluster at $a_{j_0}$. Therefore, we must have $\im(H_{j_0,[\alpha]})=H(R_{j_{i_m},[\alpha_{i_m}]}\times I)\subseteq U$ for all but finitely many $m\in\bbn$; a contradiction. Next, suppose $P$ is infinite. In this case, the sequence $\{x_j\}_{j\in P}$ converges to $x$ and so (by Lemma \ref{opensetlemma}) we have $A_{j}\subseteq U$ for all but finitely many $j\in P$. Suppose $j_1<j_2<\cdots <j_q$ are the integers for which $A_{j}\nsubseteq U$. Now, if $H(\bfx_i,t_i)\notin U$ for infinitely many $i$, there must be some $\ell\in\{1,2,\dots,q\}$ and $i_1<i_2<i_3<\cdots $ such that for all $m\in\bbn$, we have $\ell=i_{m}$ and $H(\bfx_{i_m},t_{i_m})\notin U$. Now we are in the same situation as the finite case and we arrive at a contradiction. We conclude that $H(\bfx_i,t_i)\in U$ for all but finitely many $i$, proving the continuity of $H$.

Let $h\in \Omega^n(Y,y_0)$ be the map $h(\bfx)=H(\bfx,1)$. Note that $h$ agrees with $g$ on $g^{-1}(X)$ and maps $R_{j,[\alpha]}$ to $x_j\in X$. Therefore $\im(h)\subseteq X$. Since $X$ is aspherical, $h$ is null-homotopic in $X$. Since $f\simeq g\simeq h$, we conclude that $f$ is null-homotopic in $Y$. Therefore, $\wt{f}$ is null-homotopic in $\wt{Y}$.
\end{proof}

Recall from Remark \ref{fiberbijectionremark} that each fiber $p^{-1}(x_j)$, $x_j\neq y_0$ may be identified with $\pi_1(X,y_0)$ once a choice of paths $\beta_j:(I,0,1)\to (X,x_j,y_0)$ is made. Also, each group $\pi_n(\mca_{j,[\alpha]},[\alpha])$ may be identified canonically with $\pi_n(A_j,a_j)$. Consequently, we have a homomorphism $\Phi:\pi_n(Y)\to \prod_{j\in\bbn}\prod_{\pi_1(X)}\pi_n(A_j)$ which is only canonical relative to the choice of paths $\beta_j$. Since $\Phi$ is not entirely canonical, we now find it acceptable to remove basepoints from the notation. In the following diagram of isomorphisms, the top and left map are canonical and the other two require choice for the square to commute.
\[\xymatrix{
\pi_n(\wt{Y}) \ar[d]_-{p_{\#}}^-{\cong} \ar[rr]^-{\wtphi} &&  \prod_{j\in\bbn}\prod_{[\alpha]\in p^{-1}(x_j)}\pi_n(\mca_{j,[\alpha]}) \ar[d]^-{\cong}\\
\pi_n(Y) \ar[rr]_-{\Phi} && \prod_{j\in\bbn}\prod_{\pi_1(X)}\pi_n(A_j)
}\]

\begin{corollary}\label{maincorollary}
Let $Y=\shadj(X,x_j,A_j,a_j)$ where $X$ is a one-dimensional Peano continuum and each $(A_j,a_j)$ is sequentially $(n-1)$-connected and $\pi_n$-residual. Then there is an injective homomorphism
\[\Phi:\pi_n(Y,y_0)\to \prod_{j\in\bbn}\prod_{\pi_1(X,x_0)}\pi_n(A_j,a_j).\]
\end{corollary}

\begin{example}	`
To illustrate the need for the $\pi_n$-residual property in Theorem \ref{maintheorem}, recall from Section \ref{subsectionsequentialprops} that $C\bbh_n$ is sequentially $(n-1)$-connected but is not $\pi_n$-residual at the usual basepoint $x_0$ in the base of the cone. Consider the shrinking adjunction space $Y=\bbh_1\vee C\bbh_n$ with core $X=\bbh_1$ is the Hawaiian earring with wedgepoint $y_0$, $A_1=C\bbh_n$ with basepoint $y_0$, and where $A_j=\{y_0\}$ for $j\geq 2$ (See Figure \ref{figx3}). Since each $A_j$ is sequentially $(n-1)$-connected, there is a generalized universal covering $p:\wt{Y}\to Y$ where $\wt{Y}$ consists of the topological $\bbr$-tree $\wt{\bbh_1}$ with a copy $\mca_{1,[\alpha]}$ of $C\bbh_n$ attached at each $[\alpha]\in p^{-1}(y_0)$. We may identify $p^{-1}(y_0)=\pi_1(\bbh_1,y_0)$. Applying the same argument as in Example \ref{coneintermediateexample}, gives $\pi_n(\wt{Y},\wt{y}_0)\neq 0$. Since $\pi_n(\wt{Y},\wt{y}_0)\cong \pi_n(Y,y_0)$, it follows that the trivial homomorphism\[\Phi:\pi_n(Y,y_0)\to \prod_{[\alpha]\in \pi_1(\bbh_1)}\pi_n(C\bbh_n,x_0)\]is not injective. Indeed, if $\ell_k:S^1\to \bbh_1$ is the inclusion of the $k$-th circle and $f_k:S^n\to C\bbh_n$ is the inclusion of the $k$-th sphere in the base of the cone, then $\left[\prod_{k\in\bbn}(\ell_k\ast f_k)\right]\neq 0$ in $\pi_n(Y,y_0)$ but $\Phi(\left[\prod_{k\in\bbn}(\ell_k\ast f_k)\right])=0$.
\end{example}

\begin{figure}[h]
\centering \includegraphics[height=2.5in]{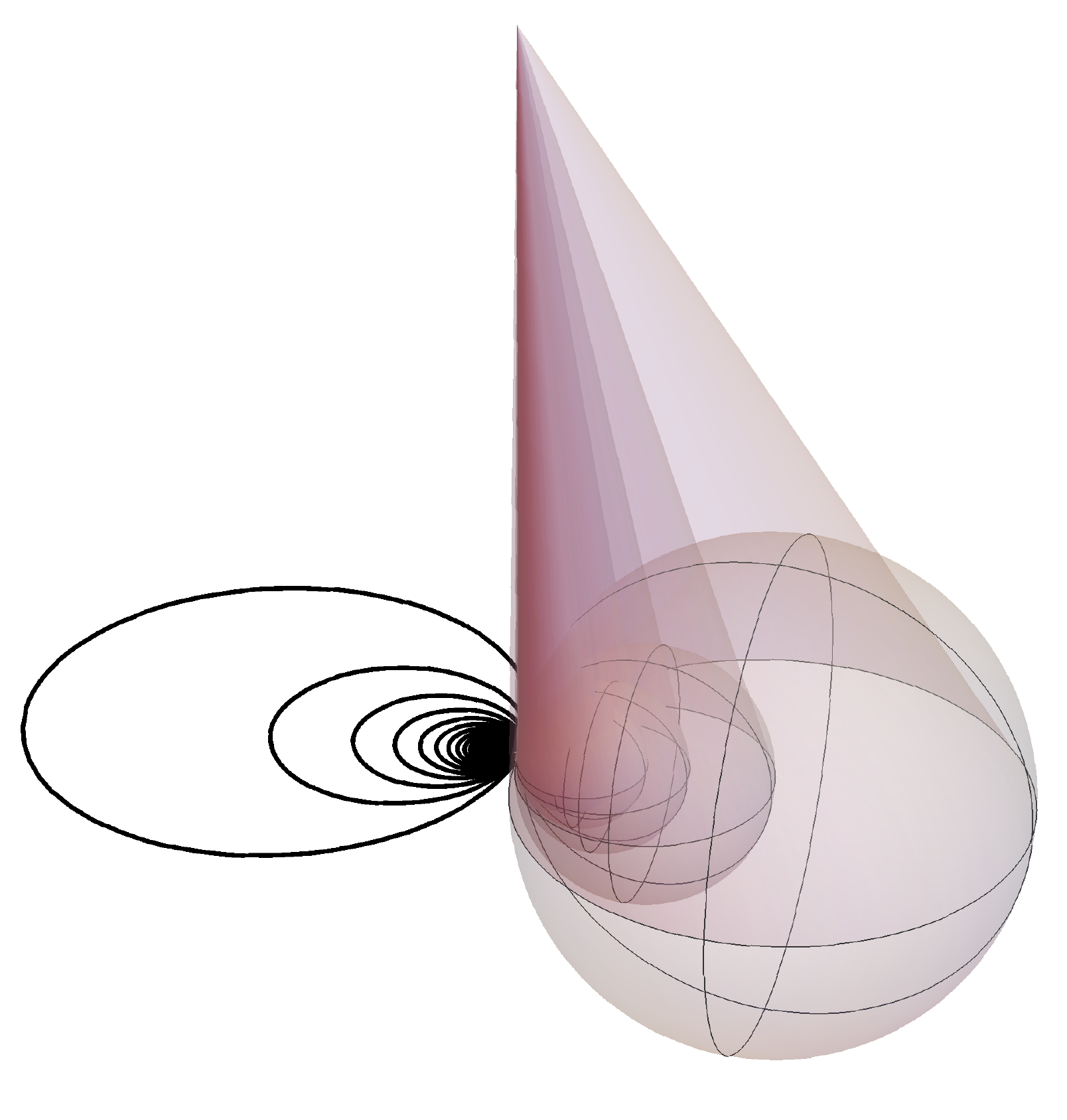}
\caption{\label{figx3}The space $\bbh_1\vee C\bbh_2$.}
\end{figure}

\subsection{The image of $\Phi$ and examples}

In this section, we characterize the image of the canonical injection \[\wtphi:\pi_n(\wt{Y},\wt{y}_0)\to \prod_{j\in\bbn}\prod_{[\alpha]\in p^{-1}(x_j)}\pi_n(\mca_{j,[\alpha]},[\alpha])\] from Theorem \ref{maintheorem}. In particular, we assume $Y=\shadj(X,x_j,A_j,a_j)$ where $X$ is a one-dimensional Peano continuum and each $(A_j,a_j)$ is sequentially $(n-1)$-connected and $\pi_n$-residual. We also re-use the notation for the subspace $\mca_{j,[\alpha]}$ of $\wt{Y}$, which is the homeomorphic copy of $A_j$ attached at $[\alpha]\in p^{-1}(x_j)$.

For convenience, we identify $\pi_n(\mca_{j,[\alpha]},[\alpha])=\pi_n(A_j,a_j)$ which allows us to view \[G=\prod_{j\in\bbn}\prod_{[\alpha]\in p^{-1}(x_j)}\pi_n(\mca_{j,[\alpha]},[\alpha])\] as the group of all functions $g:\tX\to \bigoplus_{j\in\bbn}\pi_n(A_j,a_j)$ such that $g(p^{-1}(x_j))\subseteq \pi_n(A_j,a_j)$ for all $j\in\bbn$ and $g([\alpha])=0$ if $[\alpha]\in\tX\backslash \bigcup_{j\in\bbn}p^{-1}(x_j)$. In other words, if $\supp(g)=\{[\alpha]\in\tX\mid g([\alpha])\neq 0\}$ is the \textit{support} of $g$, then $\supp(g)\subseteq p^{-1}(\{x_j\mid j\in\bbn\})$. For $g\in G$ and $j\in\bbn$, let $g_j\in G$ be the ``$j$th projection" satisfying $\supp(g_j)\subseteq p^{-1}(x_j)$ and $g_j|_{p^{-1}(x_j)}=g|_{p^{-1}(x_j)}$.

Suppose that $C\subseteq p^{-1}(x_j)$ is countably infinite and well-ordered as $C=\{[\alpha_1],[\alpha_2],[\alpha_3],\dots\}$, then there is a canonical homomorphism \[\Theta_{C}:[(\bbh_n,b_0),(A_j,a_j)]\to G\] so that if $\ell_{i}:S^n\to \bbh_n$ is the inclusion of the $i$-th sphere, then $\Theta_{C}([f])([\alpha_i])=[f\circ \ell_{i}]$, $i\in\bbn$ and $\Theta_{C}([f])([\alpha])=0$ for all $[\alpha]\in \tX\backslash C$.

\begin{theorem}\label{imagetheorem}
Consider an element $g\in G$. Then $g\in \im(\wtphi)$ if and only if 
\begin{enumerate}
\item $\ov{\supp(g)}$ is compact,
\item for all $j\in \bbn$, either $\supp(g_j)$ is finite or $\supp(g_j)$ is countably infinite and $g_j\in \im(\Theta_{\supp(g_j)})$.
\end{enumerate}
\end{theorem}

\begin{proof}
Suppose $\wt{f}\in \Omega^{n}(\wt{Y},\wt{y}_0)$ and $g=\wtphi(\wt{f})$. Let \[\scrf=\{[\alpha]\in\tX\mid \alpha(1)=x_j\text{ and }\im(\wt{f})\cap N_{j,[\alpha]}\neq\emptyset\}.\] It follows from Lemma \ref{countablelemma} that $\scrf$ is countable. Since $\supp(g)\subseteq\scrf$, $\supp(g)$ is countable. Therefore, $\supp(g_j)$ is countable for all $j\in\bbn$. Since $\supp(g)$ lies in the dendrite $\im\left(\wt{f}\right)\cap \tX$, the closure $\ov{\supp(g)}$ (taken in $\tX$) is compact. Next, fix $j\in\bbn$ such that $\supp(g_j)=\{[\alpha_1],[\alpha_2],[\alpha_3],\dots\}$ is infinite. According to the proof of Theorem \ref{maintheorem}, we may assume, without altering the homotopy class of $\wt{f}$, that $\wt{f}$ is in single factor form in the shrinking adjunction space $\im\left(\wt{f}\right)$. In particular, we may find disjoint $n$-cubes $R_{j,[\alpha_i]}$, $i\in\bbn$ such that $\wt{f}(\partial  R_{j,[\alpha_i]})=[\alpha_i]$, $\wt{f}(R_{j,[\alpha_i]})\subseteq \mca_{j,[\alpha_i]}$, and $\wt{f}(I^n\backslash \bigcup_{i\in\bbn}\int(R_{j,[\alpha_i]}))\subseteq \im\left(\wt{f}\right)\backslash p^{-1}(A_j\backslash\{a_j\})$. Define $h:(\bbh_n,b_0)\to(A_j,a_j)$ so that if so that $\ell_{i}:S^n\to \bbh_n$ is the inclusion of the $i$-th sphere, then $h\circ \ell_i=p\circ \wt{f}|_{R_{j,[\alpha_i]}}$. We have $\Theta_{\supp(g_j)}([h])=g_j$. Thus $g_j\in \im(\Theta_{\supp(g_j)})$.

For the converse, suppose $g\in G$ satisfies (1) and (2). Let $D\subseteq \tX$ be the union of all arcs connecting the points of the compact set $\ov{\supp(g)}\cup \{\wt{y}_0\}$. An elementary argument shows that $D$ is compact and is therefore a dendrite. Let $K$ be the set of pairs $(j,[\alpha])$ such that $[\alpha]\in p^{-1}(x_j)$ and $[\alpha]\in \supp(g_j)$. For each $(j,[\alpha])\in K$, we choose a map $f_{j,[\alpha]}\in \Omega^n(A_j,a_j)$ as follows: If $\supp(g_j)$ is finite, pick any map $f_{j,[\alpha]}$ representing $g_j([\alpha])\in \pi_n(A_j,a_j)$. If $\supp(g_j)$ is infinite and $g_j\in \im(\Theta_{\supp(g_j)})$, write $\supp(g_j)=\{[\alpha_1],[\alpha_2],[\alpha_3],\dots\}$ and find a map $h:(\bbh_n,b_0)\to (A_j,a_j)$ such that $\Theta_{\supp(g_j)}(h)=g_j$. Set $f_{j,[\alpha_i]}=h\circ \ell_i$ for all $i\in\bbn$. For each $(j,[\alpha])\in K$, let $\wt{f}_{j,[\alpha]}\in\Omega^n(\mca_{j,[\alpha]},[\alpha])$ be the lift of $f_{j,[\alpha_i]}$ based at $[\alpha]$. 

Since $D$ is a Peano continuum, there exists a surjective path $\wt{\beta}:(I,0)\to (D,\wt{y}_0)$. Let $\beta=p\circ\wt{\beta}$. We may assume that for each $(j,[\alpha])\in K$, there exists an open interval $O_{j,[\alpha]}$ in $(0,1)$ such that $\beta(\ov{O_{j,[\alpha]}})=[\alpha]$. As in the proof of Theorem \ref{splittheorem}, let $C_t= [\frac{t}{2},\frac{1-t}{2}]^n$ for $0\leq t<1$ and let $C_1=\{(1/2,1/2,\dots,1/2)\}$. Consider the $n$-loop $\wt{F}\in \Omega^n(\wt{X},\wt{y}_0)$, which maps $\partial C_t$ to $\wt{\beta}(t)$ for $0\leq t<1$ and $F(C_1)=\wt{\beta}(1)$. Then $F=p\circ\wt{F}\in \Omega^n(X,y_0)$ maps $\partial C_t$ to $\beta(t)$ for $0\leq t<1$ and $F(C_1)=\beta(1)$. 

Pick an $n$-cube $R_{j,[\alpha]}\subseteq \bigcup\{\partial C_t\mid t\in O_{j,[\alpha]}\}$. Let $\wt{f}:I^n\to \wt{Y}$ be the function which agrees with $\wt{F}$ on $I^n\backslash \bigcup_{(j,[\alpha])\in K}\int(R_{j,[\alpha]})$ and such that $\wt{f}|_{R_{j,[\alpha]}}\equiv \wt{f}_{j,[\alpha]}$ for all $(j,[\alpha])\in K$. Clearly $\wt{f}$ is well-defined. Set $f=p\circ \wt{f}$. Essentially the same argument used in the proof of Theorem \ref{maintheorem} (to verify the continuity of $H$) gives the continuity of $f$. The lift $\wt{f}$ must then be continuous. By construction, $\wtphi\left([\wt{f}]\right)=g$.
\end{proof}

The subgroup $\prod_{j\in\bbn}\bigoplus_{\pi_1(X,y_0)}\pi_n(A_j,a_j)$ of $G$ consists of all $g\in G$ such that $\supp(g_j)$ is finite for all $j\in\bbn$. We identify some situations where $\im\left(\wtphi\right)$ lies in this subgroup.

\begin{corollary}\label{finiteimagecor1}
Let $g\in \im(\wtphi)$ and fix $j\in\bbn$. If $X$ is $\pi_1$-finitary at $x_j$ or if $A_j$ is $\pi_n$-finitary at $a_j$, then $\supp(g_j)$ is finite.
\end{corollary}

\begin{proof}
Suppose $X$ is $\pi_1$-finitary at $x_j$. We prove, by contrapositive, that the fiber $p^{-1}(x_j)$ is discrete. Indeed if $p^{-1}(x_j)$ is not discrete, then since $\wt{X}$ is a uniquely arc-wise connected, locally path-connected metric space, there exists an injective path $\beta:\ui\to \wt{X}$ that maps $\{1/m\mid m\in\bbn\}\cup \{0\}$ into $p^{-1}(x_j)$. Setting $\gamma_m=p\circ \beta|_{[\frac{1}{m+1},\frac{1}{m}]}$ gives a sequence $\{\gamma_m\}_{m\in\bbn}$ of non-null-homotopic loops that converges to $x_j$. Therefore $X$ is not $\pi_1$-finitary at $x_j$. We conclude that $p^{-1}(x_j)$ is a closed, discrete subspace of $\tX$. It follows that $\supp(g_j)\subseteq p^{-1}(x_j)$ is a closed, discrete subspace of the compact space $\ov{\supp(g)}$. Thus $\supp(g_j)$ is finite.

Next, suppose $A_j$ is $\pi_n$-finitary at $a_j$. If $\supp(g_j)=\{[\alpha_1],[\alpha_2],[\alpha_3],\dots\}$ is countably infinite, we can find a map $h:(\bbh_n,b_0)\to (A_j,a_j)$ with $\Theta_{\supp(g_j)}([h])=g_j$. In particular, $\Theta_{\supp(g_j)}([h])([\alpha_i])=[h\circ\ell_i]$ for $i\in\bbn$. However, since $A_j$ is $\pi_n$-finitary at $a_j$, there exists an $i_0\in\bbn$ such that $\{h\circ \ell_i\}_{i\geq i_0}$ is sequentially null-homotopic. Thus $\Theta_{\supp(g_j)}([h])([\alpha_i])=0$ for all $i\geq i_0$; a contradiction.
\end{proof}

Applying Corollary \ref{finiteimagecor1} when the hypothesis applies to all attachment points, we obtain the following.

\begin{corollary}\label{finiteimagemaincor}
Suppose that for every $j\in\bbn$, either $X$ is $\pi_1$-finitary at $x_j$ or $A_j$ is $\pi_n$-finitary at $a_j$. Then $\wtphi$ maps $\pi_n(\wt{Y},\wt{y}_0)$ isomorphically onto the subgroup of $ \prod_{j\in\bbn}\bigoplus_{\pi_1(X,y_0)}\pi_n(A_j,a_j)$ consisting of $g$ with $\ov{\supp(g)}$ compact.
\end{corollary}

\begin{example}
Corollary \ref{finiteimagemaincor} includes the case where the core $X$ is arbitrary and all attachment spaces are $CW$-complexes. For instance, suppose $X$ is any $1$-dimensional Peano continuum such as the Menger Cube and $A_j$ is an $n$-sphere for all $j\in\bbn$ and $Y=\shadj(X,x_j,A_j,a_j)$ for any sequence of attachment points $\{x_j\}_{j\in\bbn}$. All attachment spaces $A_j=S^n$ are sequentially $(n-1)$-connected and $\pi_n$-finitary. Therefore, Corollary \ref{directsumcor} implies that $\wtphi$ maps $\pi_n(\wt{Y},\wt{y}_0)$ isomorphically onto a subgroup of
\[\prod_{j\in\bbn}\bigoplus_{\pi_1(X,y_0)}\pi_n(A_j,a_j)\cong \prod_{j\in\bbn}\bigoplus_{\pi_1(X,y_0)}\bbz.\]
\end{example}

\begin{example}[$\bbh_1\vee\bbh_n$]\label{h1hn}
A motivating example is the space $Y=\bbh_1\vee \bbh_n$ with wedgepoint $b_0$ (See Figure \ref{figx2}). We may regard $Y$ as having core $\bbh_1$ and attachment spaces $(A_j,a_j)=(S^n,e_n)$, $j\in\bbn$. By Corollary \ref{finiteimagemaincor}, $\Phi$ maps $\pi_n(\bbh_1\vee\bbh_n,b_0)$ isomorphically onto a subgroup of $\prod_{j\in\bbn}\bigoplus_{\pi_1(\bbh_1)}\bbz$, which we regards as a subgroup of the group of $\bbz^{\bbn\times\pi_1(\bbh_1)}$. In particular, $\pi_n(\bbh_1\vee\bbh_n)$ is cotorsion-free. It follows from Corollary \ref{finiteimagemaincor} that $g\in \im(\Phi)$ if and only if
\begin{itemize}
\item $\supp(g)=\{[\alpha]\in\pi_1(\bbh_1,b_0)\mid g(j,[\alpha])\neq 0\}$ is countable and has compact closure in $\pi_1(\bbh_1,b_0)$,
\item $g$ has finite support din the second variable in the sense that for each $j\in\bbn$, $\{[\alpha]\mid g(j,[\alpha])\neq 0\}$ is finite.
\end{itemize}
If so desired, one could represent $g$ as a formal sum $\sum_{j,[\alpha]}g(j,[\alpha])$ indexed over $\bbn\times\pi_1(\bbh_1,b_0)$ or as a formal sum $\sum_{[\alpha]}g_{[\alpha]}$ indexed over $\pi_1(\bbh_1,b_0)$ where $g_{[\alpha]}\in\bbz^{\bbn}$ is given by $g_{[\alpha]}(j)=g(j,[\alpha])$. 

Although the above characterization of $\pi_n(\bbh_1\vee\bbh_n)$ as a subgroup of the $\sigma$-product $\prod_{\mathfrak{c}}^{\sigma}\bbz$ is sufficient for carrying out computations in the group, the author does not know if the isomorphism type of $\pi_n(\bbh_1\vee\bbh_n)$ has another description.
\end{example}

The double direct sum $\bigoplus_{j\in\bbn}\bigoplus_{\pi_1(X,y_0)}\pi_n(A_j,a_j)$ consists of all $g\in G$ for which $\supp(g)$ is finite. Certainly, if $\supp(g)$ is finite, then $\supp(g)=\ov{\supp(g)}$ is compact. This case occurs in the following ``tame" situation.

\begin{corollary}\label{directsumcor}
Suppose the following two conditions hold.
\begin{enumerate}
\item All but finitely many attachment space $(A_j,a_j)$ are sequentially $n$-connected.
\item For every $j\in\bbn$, either $X$ is $\pi_1$-finitary at $x_j$ or $A_j$ is $\pi_n$-finitary at $a_j$.
\end{enumerate}
 Then $\Phi$ maps $\pi_n(Y,y_0)$ isomorphically onto $\bigoplus_{j\in\bbn}\bigoplus_{\pi_1(X,y_0)}\pi_n(A_j,a_j)$.
\end{corollary}

\begin{example}
Corollary \ref{directsumcor} includes the classical case where $X$ is a finite graph but also includes case where the core $X$ is wild and finitely many CW-complexes are attached at (possibly wild) points of $X$. For instance, suppose $X=\bbh_1$ and that $Y$ is obtained by attaching finitely many $n$-spheres to $X$. By regarding this space as a shrinking adjunction space with $A_j=\{a_j\}$ for all but finitely many $j$, Corollaries \ref{maincorollary} and \ref{directsumcor} give that $\Phi$ maps $\pi_n(Y,y_0)$ isomorphically onto $\bigoplus_{j\in\bbn}\bigoplus_{\pi_1(\bbh_1,b_0)}\pi_n(A_j,a_j)$, which is isomorphic to \[\bigoplus_{j=1}^{m}\bigoplus_{\pi_1(\bbh_1,b_0)}\bbz\cong \bigoplus_{\mathfrak{c}}\bbz.\]
\end{example}

\subsection{The inverse limit interpretation}\label{subsectionshapetheory}

Although shrinking adjunction spaces have a natural description as an inverse limit, it does not seem possible to prove our main technical results using only an inverse limit framework. For instance, to prove Theorems \ref{arctheorem} and \ref{dendritetheorem}, one must construct homotopies that deforms an $n$-loop non-trivially on an entire cofinal portion of the inverse limit and this requires working with the entire limit at each step. Moreover, it is surprisingly tedious to formalize the sense in which the isomorphisms $\wt{\Phi}$ and $\Phi$ are components of a natural isomorphism. Therefore, now that Theorem \ref{maintheorem} is established, we expose the shape-theoretic interpretation where naturality is immediate. We refer to \cite{MS82} for preliminaries of shape theory. From now on, we assume that $X$ is a one-dimensional Peano continuum and that each attachment space $A_j$ is an $(n-1)$-connected polyhedron.

It is well-known that a one-dimensional Peano continuum $X$ is homeomorphic to the inverse limit $\varprojlim_{m}(X_m,b_{m+1,m})$ of finite graphs $X_m$ \cite{Rodgers}. The bonding maps $b_{m+1,m}:X_{m+1}\to X_m$ are topological retractions which map each edge of $X_{m+1}$ linearly onto an edge of some finite subdivision of $X_{m}$. By choosing (topological) sections $s_{m,m+1}:X_{m}\to X_{m+1}$, we see that each $X_m$ is a retract of $X$. Hence, there are embeddings $s_{m}:X_m\to X$, which are sections to the projections $b_m:X\to X_m$. Identifying $X_m$ with its image in $X$, we choose $y_0\in X_1\subseteq X_2\subseteq X_3\subseteq\cdots \subseteq X$. Each fundamental group $\pi_1(X_m,y_0)$ is a finitely generated free group $F_{i_m}$ of rank $i_m$ and it is known that the canonical homomorphism $\Lambda_1:\pi_1(X,y_0)\to \varprojlim_{m}(F_{i_m},(b_{m+1,m})_{\#})$ is injective \cite{CConedim,EK98onedimfgs}. However, $\Lambda_1$ is not surjective unless $X$ itself is a finite graph. For notational convenience, we may write $X_{\infty}$ for $X=\varprojlim_{m}(X_m,b_{m+1,m})$ and $b_{\infty}:X_{\infty}\to X_{\infty}$ for the identity map. The homomorphism $\Phi$ from Corollary \ref{maincorollary} is not surjective, in part, because $\Lambda$ is not surjective.

For $k,m\in\bbn\cup\{\infty\}$, let $Y_{k,m}$ be the space obtained by attaching $A_j$, $1\leq j\leq k$ to $X_m$ by $a_j\sim b_{m}(x_j)$. Here the first variable $k$ corresponds to the number of spaces $A_j$ attached and the second variable $m$ corresponds to the graph approximation $X_m$ to $X=X_{\infty}$. For example, $Y_{\infty,m}$ is the shrinking adjunction space obtained by attaching all $A_j$ to $X_m$ along the sequence of points $\{b_{m}(x_j)\}_{j\in\bbn}$ and the space $Y_{k,\infty}$ was previously denoted $Y_k$. We have bonding maps:
\begin{enumerate}
\item $\zeta_{(k+1,k),m}:Y_{k+1,m}\to Y_{k,m}$ which is the identity on $X_{m}$, which maps $A_1,A_2,\dots,A_k$ homeomorphically to the corresponding copy in $Y_{k,m}$, and which collapses $A_{k+1}$ to $b_{m}(x_{k+1})\in X_m$.
\item $\zeta_{k,(m+1,m)}:Y_{k,m+1}\to Y_{k,m}$ which agrees with $b_{m+1,m}$ on $X_{m+1}$ and which maps $A_1,A_2,\dots,A_k$ homeomorphically to the corresponding copy in $Y_{k,m}$.
\item $\zeta_{(k+1,k),\infty}:Y_{k+1,\infty}\to Y_{k,\infty}$ agrees with the previously defined map $\rho_{k+1,k}$.
\item $\zeta_{\infty,(m+1,m)}:Y_{\infty,m+1}\to Y_{\infty,m}$, which agrees with $b_{m+1,m}$ on $X_m$ and which maps $A_j$ (attached at $b_{m+1}(x_j)$) homeomorphically to the corresponding copy of $A_j$ in $Y_{\infty,m}$ (attached at $b_{m}(x_j)$).
\end{enumerate}
Using the sections to $b_m$ and $b_{m+1,m}$, one can directly construct sections to each of these maps. Hence, all of these maps are topological retractions. Whenever $k\leq k'\leq \infty$ and $m\leq m'\leq \infty$, there is a map $\zeta_{(k',k),(m'm)}:Y_{k',m'}\to Y_{k,m}$ constructed as the corresponding composition of the above maps. Additionally, for any pair $k,m\in\bbn\cup\{\infty\}$, there is a retraction $\zeta_{k,m}:Y_{\infty,\infty}\to Y_{k,m}$. We may make the following identifications:
\begin{enumerate}
\item $Y_k=Y_{k,\infty}=\varprojlim_{m}(Y_{k,m},\zeta_{k,(m+1,m)})$
\item $Y_{\infty,m}=\varprojlim_{k}(Y_{k,m},\zeta_{(k+1,k),m})$
\item $Y=Y_{\infty,\infty}=\varprojlim_{k}Y_{k,\infty}=\varprojlim_{m}Y_{\infty,m}$
\end{enumerate}

For any pair $k,m\in\bbn\cup\{\infty\}$, the space $Y_{k,m}$ is the shrinking adjunction space obtained by attaching a sequence of spaces $A_j$, $1\leq j\leq k$ to the one-dimensional Peano continuum $X_m$ along the points $\{b_m(x_j)\}_{1\leq j\leq k}$ in $X_m$. Hence, the results of the previous section apply to $Y_{k,m}$.

Specifically, there are generalized universal covering maps $q_{k,m}:\wt{X}_{m}\to X_m$ and $p_{k,m}:\wt{Y}_{k,m}\to Y_{k,m}$ such that the restriction of $p_{k,m}$ to $p_{k,m}^{-1}(X_m)$ agrees with $q_{k,m}$.  When $m<\infty$, $\wt{X}_m$ is a tree and when $m=\infty$, $\wt{X}_{\infty}$ is a topological $\bbr$-tree.

When $k<\infty$, $Y_{k,m}$ is an ordinary adjunction space with the weak topology with respect to $X_m,A_1,A_2,\dots, A_k$. It is straightforward to check that $\wt{Y}_{k,m}$ is the adjunction space (with the weak topology) obtained by attaching a copy of $A_j$, $1\leq j\leq k$ to $\wt{X}_m$ at each point in the fiber of $p_{k,m}^{-1}(b_m(x_j))$. Corollary \ref{directsumcor} gives that \[\pi_n(\wt{Y}_{k,m})\cong \bigoplus_{j=1}^{k}\bigoplus_{[\alpha]\in F_{i_m}}\pi_n(A_j).\]
Under this identification, the bonding maps may be described as follows.
\begin{itemize}
\item $\pi_n(\wt{Y}_{k+1,m})\to \pi_n(\wt{Y}_{k,m})$ collapses the $k+1$-th summand $\bigoplus_{F_{i_m}}\pi_n(A_{k+1})$ to the identity.
\item $\pi_n(\wt{Y}_{k,m+1})\to \pi_n(\wt{Y}_{k,m})$ maps the summand $\pi_n(A_j)$ with index $(j,[\alpha])\in \{1,2,\dots,k\}\times F_{i_{m+1}}$ isomorphically onto the summand $\pi_n(A_j)$ with index $(j,(b_{m+1,m})_{\#}([\alpha]))\in \{1,2,\dots,k\}\times F_{i_{m}}$.
\end{itemize}

The spaces $Y_{k,m}$, $(k,m)\in\bbn^2$, approximate $Y$ and form an $HPol$-expansion of $Y$. Hence the $n$-th \v{C}ech homotopy group of $Y$ is
\[\check{\pi}_n(Y)\cong \varprojlim_{(k,m)\in\bbn^2}\pi_n(Y_{k,m})\]
Here, $\bbn^2$ is endowed the product directed order $(k,m)\leq (k',m')$ if $k\leq k'$ and $m\leq m'$. There are three natural ways one can describe this inverse limit, as a single limit along the diagonal or taking the limit in each variable separately:
\[\check{\pi}_n(Y)\cong \varprojlim_{k\in\bbn}\pi_n(Y_{k,k})\cong \varprojlim_{m}\prod_{k}\pi_n(Y_{k,m})\cong  \prod_{k}\varprojlim_{m}\pi_n(Y_{k,m}).\]

Regardless of how one chooses to represent the $n$-th shape homotopy group, Theorem \ref{maintheorem} implies the following.

\begin{theorem}
Let $Y=\shadj(X,x_j,A_j,a_j)$ where $X$ is a one-dimensional Peano continuum and each $A_j$ is a $(n-1)$-connected polyhedron. Then the natural homomorphism $\Lambda_n:\pi_n(Y)\to \check{\pi}_n(Y)$ is injective.
\end{theorem}

\begin{proof}
Suppose $0\neq[f]\in \pi_n(Y,y_0)$. Let $\wt{f}\in \Omega^n(\wt{Y},\wt{y}_0)$ be the lift to the generalized universal covering space $\wt{Y}$. By Theorem \ref{maintheorem}, there exists $k\in\bbn$ and $[\alpha]\in p^{-1}(x_k)$ such that $r_{k,[\alpha]}\circ \wt{f}$ is not null-homotopic in $\mca_{k,[\alpha]}$. Fix any $m\geq 1$ and let $\wt{\zeta}_{k,m}:\wt{Y}\to \wt{Y}_{k,m}$ be the lift of the retraction $\zeta_{k,m}$, which satisfies $p_{k,m}\circ \wt{\zeta}_{k,m}=\zeta_{k,m}\circ p$. Recall that $Y_{k,m}$ consists of the finite graph $X_m$ with spaces $A_1,A_2,\dots, A_k$ attached. Let $B_k$ be the homeomorphic copy of $A_k$ in $\wt{Y}_{k,m}$ attached at $\wt{\zeta}_{k,m}([\alpha])\in p_{k,m}^{-1}(\zeta_{k,m}(x_k))$. Since $\mca_{k,[\alpha]}$ is the locally path connected coreflection of $B_k$, the restriction $\wt{\zeta}_{k,m}:\mca_{k,[\alpha]}\to B_k$ is a bijective homotopy equivalence. In particular, if $R:\wt{Y}_{k,m}\to B_k$ is the canonical retraction, then $[R\circ\wt{\zeta}_{k,m}\circ \wt{f}]\neq 0$. Since $\pi_n(B_k)$ is one of the summands in $\pi_n(Y_{k,m})\cong \bigoplus_{j=1}^{k}\bigoplus_{\pi_1(X_m)}\pi_n(A_j)$, we conclude that $[\zeta_{m,k}\circ f]\neq 0$.
\[\xymatrix{
& \wt{Y} \ar[d]_-{p} \ar[r]^-{\wt{\zeta}_{k,m}} & \wt{Y}_{k,m} \ar[d]^-{p_{k,m}} \ar[dr]^-{R}  \\
S^n \ar[ur]^-{\wt{f}} \ar[r]_-{f} & Y \ar[r]_-{\zeta_{k,m}} & Y_{k,m} \ar[r] & B_k
}\]
\end{proof}

While it is outside the scope of the current paper, the author expects that further algebraic study of $\check{\pi}_n(Y)$ may provide an interesting alternative presentation of $\pi_n(Y)$.


\end{document}